\documentclass[10pt, reqno]{amsart}
\usepackage{graphicx} 
\usepackage{amsthm}
\usepackage{amsmath}
\usepackage{amssymb}
\usepackage{amsfonts}
\usepackage[cm]{fullpage}
\usepackage[colorlinks=true,citecolor=red,linkcolor=blue,pdfpagetransition=Blinds,backref=page]
{hyperref}
\usepackage{bbm}
\usepackage{xcolor}
\usepackage{mathtools}
\usepackage{enumerate}
\usepackage{orcidlink}

\newtheorem{theorem}{Theorem}[section]
\newtheorem{proposition}{Proposition} [section]
\newtheorem{definition}{Definition}[section]
\newtheorem{lemma}{Lemma}[section]

\numberwithin{equation}{section}
\newtheorem{remark}{Remark}[section]
\newcommand{\T}{\mathbb T}
\newcommand{\norm}[1]{\left\|#1\right\|}
\newcommand{\R}{\mathbb{R}}
\newcommand{\mc}{\mathcal}
\newcommand{\pa}{\partial}
\newcommand{\ip}[2]{\left<{#1},{#2}\right>}
\newcommand \F{\mathcal S}

\newcommand\Y{\mathcal{Y}}
\newcommand\V{\mathcal{V}}
\usepackage{cleveref}
\usepackage{comment}
\usepackage{xpatch}

\makeatletter
\xpatchcmd{\@tocline}
{\hfil\hbox to\@pnumwidth{\@tocpagenum{#7}}\par}
{\dotfill\hbox to\@pnumwidth{\@tocpagenum{#7}}\par}
{}{ }
\makeatother

\title[Controllability of some fourth-order parabolic equations by Multiplicative Force]{Small-time global controllability of a class of bilinear fourth-order parabolic equations}
\author[Subrata Majumdar \,\,  and \, \,Debanjit Mondal]{Subrata Majumdar$^{\dagger{\orcidlink{0000-0001-6724-6943}}}$ \, and \, Debanjit Mondal$^{\ddagger{\orcidlink{0009-0002-8709-3774}}}$ }
\thanks{$\dagger$ Centre For Applicable Mathematics,	Tata Institute of Fundamental Research,
	Post Bag No 6503, Sharada Nagar, 
	Bengaluru 560065, India (email: subrata26@tifrbng.res.in, subratamajumdar634@gmail.com)}
\thanks{$^{\ddagger}$Indian Institute of Science Education and Research Kolkata, Campus road, Mohanpur, West Bengal 741246, India (email: wrmarit@gmail.com, dm20ip005@iiserkol.ac.in)}

\begin{document}
\begin{abstract}
In this work, we address the small-time global controllability properties of a class of fourth-order nonlinear parabolic equations driven by bilinear controls posed on the one-dimensional torus. 
The controls depend only on time and act through a prescribed family of spatial profiles. Our first result establishes the small-time global approximate controllability of the system using three scalar controls, between states that share the same sign. This property is obtained by adapting the geometric control approach to the fourth-order setting, using a finite family of frequency-localized controls. We then study the small-time global exact controllability to non-zero constant states for the concerned system. This second result is achieved by analyzing the null controllability of an appropriate linearized fourth-order system and by deducing the controllability of the nonlinear model through a fixed-point argument together with the small-time global approximate control property. To the best of our knowledge, this work provides the first contribution toward the study of controllability properties of fourth-order parabolic equations by means of bilinear controls.
\end{abstract}

\keywords{Global controllability, Geometric control theory, Saturation property, Bilinear control, Agrachev-Sarychev approach, Moment problem}
\subjclass{93B05, 93C10, 35K55, 35Q93}
\date{\today}
\allowdisplaybreaks
\maketitle

\begingroup
\hypersetup{linkcolor=black}
\tableofcontents
\endgroup


\section{Introduction}
\subsection{System under study}
 In this paper, we study global controllability aspects of the following fourth-order nonlinear parabolic equations within a \emph{multiplicative control} framework, defined on the one-dimensional torus $\mathbb{T}:= \mathbb{R}/ 2\pi \mathbb{Z}$
\begin{align}\label{intro_main}
\begin{cases}
\partial_t u(t,x) +\nu_1 \partial_{x}^4 u(t,x) +  \nu_2\partial_{x}^2u(t,x)+\mc{N}(u)(t,x) = Q(t,x) u(t,x), &  t>0, \, x \in  \mathbb{T},\\
u(0, x) = u_0(x), & x \in \mathbb{T},
\end{cases}
\end{align}
where $\nu_1 > 0$ and $\nu_2 > 0$ are real parameters, with $\nu_1$ corresponding to a fourth-order diffusion term, while $\nu_2$ represents an anti-diffusion parameter. $Q$ is a function which plays the role of the multiplicative control and is specified precisely in \eqref{ctrl_form}–\eqref{fourier_mods}.
Depending on the choice of the nonlinear term $\mathcal{N}(u)$, this setting encompasses several classical fourth-order nonlinear parabolic models. In this work, we focus on two prototypical nonlinearities, namely the \emph{Kuramoto–Sivashinsky} \cite{KT75,KT76,Tad86} and the \emph{Cahn–Hilliard} \cite{CH58,CH59,Mir19}: 
\begin{align}
   \label{ks}\mathcal{N}_{KS}(u) &:= u\partial_x u,\\
   \label{ch}\mathcal{N}_{CH}(u) &:= -\partial_x^2(u^3).
   \end{align}
Our approach also applies, with minor modifications, to other representative examples, such as the \emph{Sivashinsky equation} $\mathcal{N}_{Siv}(u):=-\partial_x^2(u^2)$ \cite{Siv83,NG95} and \emph{semilinear fourth-order parabolic equations} of the form $\mathcal{N}_{sem}(u):=u^\gamma$, $\gamma\in\mathbb{N}^*$ \cite{DR98}.

Multiplicative control problems (see \cite{Kh10}) naturally arise in the modeling of distributed systems where the control acts by modifying intrinsic properties of the medium rather than through external forcing. In such situations, the control enters the evolution equation in a multiplicative way, leading to nonlinear control systems even when the underlying dynamics are linear. This type of control is particularly relevant for higher-order parabolic equations, which appear in the modeling of pattern formation, phase separation, and interfacial dynamics, and for which additive control mechanisms may be insufficient. From an applied perspective, classical additive controls are suitable only for processes whose intrinsic physical properties remain unaffected by the control action, modeling instead the influence of externally imposed forces or sources. This framework, however, fails to capture a wide range of emerging and established technologies, such as smart materials and various biomedical, chemical, and nuclear reaction systems, whose fundamental parameters (e.g., frequency response or reaction rates) can be deliberately modified through controlled mechanisms, such as catalytic effects. A notable example of a parabolic system with multiplicative control is the distributed parameter model studied by Lenhart and Bhat \cite{Lenhart_1992}, motivated by wildlife damage management problems involving the regulation of diffusive small mammal populations.
\subsection{Small-time global approximate controllability}
 In this paper, we are concerned with the \emph{global controllability} properties of the fourth-order parabolic system \eqref{intro_main}. Our main motivation is the following question: given initial and target states $u_0$ and $u_1$, and a final time $T>0$, does there exist a shorter time $\tau\in (0,T]$ and a control $Q$ such that the corresponding solution $u$ of \eqref{intro_main}, starting from $u_0$, reaches an arbitrarily small neighborhood of $u_1$ at time $\tau$, in an appropriate norm?
This question naturally leads to the notion of \emph{small-time approximate controllability}. Before introducing this concept, we first define the controlled evolution problem explicitly.
To this end, we consider the following control system:
\begin{align}\label{ctrl_prblm}
\begin{cases}
\partial_t u(t,x) + \partial_{x}^4 u(t,x) +  \partial_{x}^2u(t,x)+\mc{N}(u)(t,x) = Q(t,x) u(t,x), &  t>0, \, x \in \mathbb{T},\\
u(0, x) = u_0(x), & x \in \mathbb{T}.
\end{cases}
\end{align}
For simplicity, we fix $\nu_1=\nu_2=1$ in \eqref{intro_main}. The extension to arbitrary positive values of these parameters follows from the same analysis. We now specify the class of nonlinearities considered in this work.
In what follows, we focus on the cases \eqref{ks} and \eqref{ch}, that is, we take $\mathcal N$ to be either $\mathcal N_{KS}$ or $\mathcal N_{CH}$. Detailed proofs are provided only for these two nonlinearities, since the other cases mentioned above can be treated by similar arguments.
$Q$ is a function which plays the role of the multiplicative control. We control the system through low-mode forcing, meaning that the control function $Q$ is assumed to have the form
\begin{align}\label{ctrl_form}
Q(t, x) = \Big\langle p(t), \mu (x) \Big\rangle = \displaystyle \sum_{i=1}^{3} p_i(t) \mu_i(x),
\end{align}
where $\mu_i$ are chosen to be the first real Fourier modes on the torus, namely 
\begin{align}\label{fourier_mods}
\left(\mu_1(x), \mu_2(x), \mu_3(x)\right) := \left(1, \, \cos x, \, \sin x \right),
\end{align} and $p = (p_1, p_2, p_3) \in L^2_{\mathrm{loc}}(\mathbb{R}^{+}; \mathbb{R}^{3})$ consists of piecewise constant control laws that can be chosen freely. Thus, determining the scalar controls $p_i(t)$ is sufficient to define the control function $Q(t,x).$ This type of multiplicative control $Q$ is also referred to as \emph{bilinear control}, since only the time-dependent intensity acts as the control variable of the evolution, and the corresponding term depends linearly on it.

We introduce below the notion of {small-time approximate controllability} relevant to our analysis.
\begin{definition}\label{defn_approx_ctrl}
Let $H \subset L^2(\mathbb{T})$ be a Hilbert space. Assume that $u_0, u_1 \in H,$ and $T > 0, \, \varepsilon >0,$ are given. 
 \begin{enumerate}[(A)]
\item\label{1} The equation \eqref{ctrl_prblm}--\eqref{fourier_mods} is said to be {small-time $L^2$-approximately controllable}, if  there exists $\tau \in (0, T]$ and a control law $p \in L^2((0,\tau) ;\,\mathbb{R}^{3})$ such that the solution $u$ of \eqref{ctrl_prblm}--\eqref{fourier_mods} associated with the control $p$ and initial condition $u_0$ satisfies $$\left\| u(\tau) - u_1 \right\|_{L^2(\T)} < \varepsilon.$$
	\item\label{2}   The equation \eqref{ctrl_prblm}--\eqref{fourier_mods} is said to be {$H$-approximately controllable}, if  there exists  a control law $p \in L^2((0,T) ; \,\mathbb{R}^{3})$ such that the solution $u$ of \eqref{ctrl_prblm}--\eqref{fourier_mods} associated with the control $p$ and initial condition $u_0$ satisfies $$\left\| u(T) - u_1 \right\|_{H} < \varepsilon.$$
\end{enumerate}
\end{definition}
To formulate our main result, let us denote by $H^s(\T)$ the usual Sobolev space defined on the one-dimensional torus $\T$ (see \Cref{sec_function}). Our first main result is the following.

\begin{theorem}[Small-time global approximate controllability]\label{main_thm_approx}
Let $s>\tfrac12$ and $u_0, u_1 \in H^{s}(\T)$. Then the controlled system \eqref{ctrl_prblm}--\eqref{fourier_mods} enjoys the following small-time approximate controllability properties:
\begin{enumerate}[(A)]
\item\label{point_11} If $\operatorname{sign}(u_0)=\operatorname{sign}(u_1)$, then \eqref{ctrl_prblm}--\eqref{fourier_mods} is small-time $L^2$-approximately controllable. 
\item\label{point_12} If $u_0, u_1>0$ (respectively, $u_0, u_1<0$), then \eqref{ctrl_prblm}--\eqref{fourier_mods} is $H^s$-approximately controllable.
\end{enumerate}
\end{theorem}

\Cref{main_thm_approx} highlights several key features in the study of global controllability for fourth-order parabolic equations. First, it establishes global approximate controllability for a nonlinear system using only three scalar controls. More precisely, the profiles $(\mu_1(x), \mu_2(x), \mu_3(x))$ are fixed in the expression of $Q$, and the control acts solely through the time-dependent function $p$. Moreover, this approximate controllability holds for any arbitrarily small $T>0$, without any restriction on the control time $ T$. Finally, the number of controls is independent of the system parameters, as well as of the initial and target states. 

From a control point of view, bilinear control provides an efficient way to influence infinite-dimensional systems using only a finite number of scalar controls. In particular, low-mode multiplicative controls allow one to act on the large-scale dynamics while keeping the control structure simple. However, the coupling between the control and the state makes the analysis more delicate, especially in the case of higher-order parabolic equations.

We briefly review some existing literature on the fourth-order nonlinear parabolic equations of the form \eqref{ctrl_prblm}. Guzm\'an studied the local exact controllability to the trajectories of the Cahn–Hilliard equation using localized internal control in \cite{Guz20}, whereas Cerpa and Mercado investigated similar issues for the Kuramoto–Sivashinsky equation via boundary controls \cite{CM11}. It is worth noting that the above results are obtained in the framework of internal or boundary controls, are essentially local and rely on Carleman estimates for the associated linearized systems together with a local inverse mapping argument.
In this setting, extending controllability results beyond local controllability appears unattainable through Carleman-type techniques alone, which motivates the use of alternative approaches based on geometric and saturation methods to address global controllability issues.
More recently, with this approach, Gao studied the global approximate controllability of the Kuramoto-Sivashinsky equation posed on $\T$ by means of an additive control consisting of finitely many Fourier modes, see \cite{Peng_Gao_2022}. On the contrary, in the bilinear control framework considered here, the control enters the equation multiplicatively, resulting in a different structure of the controlled dynamics.
To the best of our knowledge, the global controllability of the equation \eqref{ctrl_prblm} within a bilinear control framework has not yet been addressed in the literature. 

Let us mention a few key results on the bilinear control problem using a geometric control approach. For the semilinear heat equation, a small-time global controllability result was obtained recently by means of bilinear control by Duca, Pozzoli, and Urbani in \cite{Duca_Pozzoli_Urbai_JMPA_2025}.
A closely related problem for the Burgers equation has been addressed by Duca and Takahashi in \cite{Duca_Takashi_2025}. For the wave equation, see Pozzoli \cite{Pozzoli_2024}; while results for the Schr\"odinger equation can be found in recent papers by Beauchard and Pozzoli \cite{BP25,beauchard2025small} and the references therein. In
these works, the authors employed a saturating geometric control strategy, originally
introduced by Agrachev and Sarychev in \cite{Agrachev_Sarychev_2005,Agrachev_Sarychev_2006}
for low-mode forcing internal control problems for the Navier--Stokes and Euler equations. This method was first adapted to the bilinear
controllability problem for the Schr\"odinger equation by Duca and Nersesyan in \cite{Duca_Nersesyan_JEMS_2025}. Our approach in the present work is
inspired by \cite{Duca_Nersesyan_JEMS_2025}.

In contrast to existing contributions on bilinear control settings, the present work extends this methodology to a class of fourth-order nonlinear parabolic dynamics, providing, to the best of our knowledge, the first study of controllability for such systems within a bilinear control framework. Beyond the specific model under consideration, our analysis highlights the robustness of the saturating geometric control strategy with respect to both the order of the underlying system and the nature of the nonlinearities involved. This demonstrates that the approach is not limited to the second-order partial differential equations, but can be successfully applied to higher-order equations with nonlinear structures of a different type. Consequently, the present analysis suggests that similar techniques may apply to the study of controllability issues for bilinear dispersive equations, such as the nonlinear Korteweg--de Vries and the Kawahara equations.

\subsection{Small-time global exact controllability to the constant states}
Our second objective is to establish the \emph{small-time global exact controllability} of \eqref{ctrl_prblm} toward non-zero constant steady states associated with a free trajectory. Let us recall the control system \eqref{ctrl_prblm} with the control of the form
\begin{align}\label{ctrl_form1}
Q(t, x) = \Big\langle p(t), \mu (x) \Big\rangle = \displaystyle \sum_{i=1}^{m} p_i(t) \mu_i(x), \, \, m\leq 5,
\end{align}
where $\mu_i$, $i=1,2,3$, are the same as in \eqref{fourier_mods}, and the functions $\mu_4, \mu_5 \in H^1(\mathbb{T})$ will be chosen later depending on the nonlinearities \eqref{ks} and \eqref{ch}. Furthermore, $p = ( p_1, p_2,p_3,p_4,p_5) \in L^2_{\mathrm{loc}}(\mathbb{R}^{+}; \mathbb{R}^{5})$ consists of control laws as before.
\paragraph{\textit{Cahn-Hilliard equation}:} In this case, we fix $m=5.$ Let $\{\widehat \lambda_k\}_{k\in\mathbb{N}}$ denote the ordered eigenvalues of the Laplacian $-\pa_x^2$ on $\mathbb{T}$, counted without multiplicity. Then
\begin{align*}
\widehat\lambda_k = k^2, \qquad \forall\, k \in \mathbb{N}.
\end{align*}
Observe that, except for the first eigenvalue $\widehat\lambda_0=0$, all the eigenvalues are double. 
We denote by $\{c_0,c_k,s_k\}_{k\in\mathbb{N}^*}$ the corresponding orthonormal eigenfunctions of $-\pa_x^2$, given by
\begin{equation*}
c_0(x)=\frac{1}{\sqrt{2\pi}}, \qquad 
c_k(x)=\frac{1}{\sqrt{\pi}}\cos(kx), \qquad 
s_k(x)=\frac{1}{\sqrt{\pi}}\sin(kx), \quad \forall\, k\in\mathbb{N}^*.
\end{equation*}
Consider $\mu_4,\mu_5 \in H^1(\mathbb{T})$ such that there exist positive constants $\theta_i, C_i$, $i=1,2,$ satisfying 
\begin{align}\label{est_mu1}
\begin{cases}
& \hspace{1cm}\ip{\mu_4}{c_0}_{L^2(\mathbb{T})} \neq 0 \text{ and } \ip{\mu_5}{c_0}_{L^2(\mathbb{T})} = 0,\\
&\widehat\lambda_k^{\theta_1}\,
\bigl| \ip{\mu_4}{c_k}_{L^2(\mathbb{T})} \bigr|
\geq C_1, \text{ and }  \ip{\mu_4}{s_k}_{L^2(\mathbb{T})}=0,\qquad \text{ for all } k \in \mathbb{N}^*,\\
&\widehat \lambda_k^{\theta_2}\,
\bigl| \ip{\mu_5}{s_k}_{L^2(\mathbb{T})} \bigr|
\geq C_2, \text{ and }  \ip{\mu_5}{c_k}_{L^2(\mathbb{T})}=0, \qquad \text{ for all } k \in \mathbb{N}^*.
\end{cases}
\end{align} 

\begin{theorem}[Global exact controllability]\label{global_exact}
    Let $(\mu_1,\mu_2,\mu_3,\mu_4,\mu_5)$ be as \eqref{fourier_mods} and \eqref{est_mu1} and  $s>\frac{1}{2}$. Assume that $T>0$, and $u_0\in H^s(\T)$ with $u_0>0.$ Let $\Phi>0$ be a given real number. Then there exists a  control $p\in L^2((0,T);\mathbb{R}^5)$, such that the solution $u$ of \eqref{ch}, \eqref{ctrl_prblm}, and \eqref{ctrl_form1}   satisfies $u(T,\cdot)=\Phi$ in $\T.$

  \noindent
    Analogously, for any $u_0\in H^s(\T)$ with $u_0<0,$ there exists a control $p\in L^2((0,T);\mathbb{R}^5)$, such that $u(T,\cdot)=-\Phi$ in $\T.$
\end{theorem}

\Cref{global_exact} establishes global exact controllability of \eqref{ctrl_prblm}, \eqref{ch}, and \eqref{ctrl_form1} to the nonzero stationary states associated with $p=0$ in arbitrary time horizon. This result is obtained by combining the small-time global approximate controllability provided by \Cref{main_thm_approx} with a local exact controllability result to the stationary states, valid for any positive time. Consequently, the control strategy involves five potentials: three control directions are required to achieve the approximate controllability stated in \Cref{main_thm_approx}, together with two additional controls to ensure local exact controllability.

The controllability of the associated linearized model is established using the method of moments \cite{FR1,FR2}. These seminal works have been significantly extended over the years, leading to numerous important results for a wide range of parabolic problems and control strategies; see, for instance, \cite{FCGBdeT10,AB2014,AKBGBdT16}. For a presentation of the moment method in the case of fourth-order parabolic equations, we refer to \cite{Cer10,HsM25}. Based on an explicit control cost estimate for the corresponding linearized control problem, we further prove local exact controllability of the main nonlinear model by means of the source term method \cite{LTT}.

An interesting feature of \Cref{global_exact} is the validity of an exact controllability result on the torus, where the Laplacian exhibits double eigenvalues. This essentially provides an analogous phenomenon for the bi-Laplacian. The proof of exact controllability relies on the solvability of a suitable joint moment problem, which is more delicate in this setting. This difficulty is overcome by filtering the spectrum of the bi-Laplacian through two additional control potentials, $\mu_4$ and $\mu_5$, chosen so that $\mu_4$ acts only on the cosine modes while $\mu_5$ acts only on the sine modes, see hypothesis \eqref{est_mu1}. As a result, the moment problem can be decomposed into two independent subproblems, each associated with simple eigenvalues. 

\smallskip
A similar framework can also be employed to address the global exact controllability problem for the Kuramoto–Sivashinsky equation \eqref{ctrl_prblm} and \eqref{ks}. In this case, one may employ five control profiles as before. Nevertheless, by exploiting the specific structure of the nonlinearity in \eqref{ks}, it is possible to achieve global exact controllability with only four control laws, 
 upon a slight modification of the assumptions in \eqref{est_mu1}.

\smallskip
\paragraph{\textit{Kuramoto-Sivashinsky equation}:} Let us fix $m=4$ in \eqref{ctrl_form1} and consider $\mu_4 \in H^1(\mathbb{T})$ such that
\begin{equation}\label{est_q1}
(\widehat \lambda_k^{\theta}+1)\,
\bigl| \ip{\mu_4}{e^{ikx}}_{L^2(\mathbb{T})} \bigr|
\geq C, \qquad \text{ for all } k \in \mathbb{Z}, \text{ and for some } C, \theta>0. \qquad\end{equation}   \begin{theorem}[Global exact controllability]\label{global_exact1}
    Let $(\mu_1,\mu_2,\mu_3,\mu_4)$ be as \eqref{fourier_mods} and \eqref{est_q1} and $s>\frac{1}{2}$. Assume $T>0$, and $u_0\in H^s(\T)$ with $u_0>0$. Let $\Phi>0$ be a given real number. Then there exists a  control $p\in L^2((0,T);\mathbb{R}^4)$, such that the solution $u$ of \eqref{ks}, \eqref{ctrl_prblm}, and \eqref{ctrl_form1}   satisfies $u(T,\cdot)=\Phi$ in $\T.$

  \noindent
    Analogously, for any $u_0\in H^s(\T)$ with $u_0<0,$ there exists a control $p\in L^2((0,T);\mathbb{R}^4)$, such that $u(T,\cdot)=-\Phi$ in $\T.$
\end{theorem}
The controllability of parabolic-type equations driven by bilinear (multiplicative) controls is known to be an interesting problem, even for linear systems. A key difficulty is due to a structural obstruction identified in \cite{Ball_marsder_1982}, where it was shown that the reachable set of a linear equation with multiplicative control, starting from any initial data in $L^{2}(\T)$, is contained in a countable union of compact subsets of $L^{2}(\T)$. Consequently, its complement is dense in $L^{2}(\T)$, which rules out the possibility of achieving classical exact controllability in the $L^{2}$ framework.

Due to the lack of exact controllability, several {approximate controllability} results for parabolic equations (second-order) with multiplicative controls have been established in the literature. In \cite{Khapalov_2002}, approximate controllability was obtained for one-dimensional semilinear parabolic equations over sufficiently large time horizons $T>0$, between nonnegative states, with control functions depending on both space and time. Related results were derived in \cite{Cannarsa_Floridia_2011} for linear degenerate parabolic equations subject to Robin boundary conditions. Approximate controllability for nonlinear degenerate parabolic equations with bilinear controls in large time was further investigated in \cite{Florida_2014}. Multiplicative controllability properties for semilinear reaction--diffusion equations allowing finitely many sign changes were studied in \cite{Cannarsa_Floridia_Khapalov_2017}. More precisely, the results of \cite{Cannarsa_Floridia_Khapalov_2017} show that any target state exhibiting the same number of sign changes, in the same order, as the prescribed initial datum can be approximately reached in the $L^{2}$-norm at a sufficiently large time $T>0$. Similar controllability results for nonlinear degenerate parabolic equations between sign-changing states were later obtained in \cite{FloridaINitsch_Trombetti_2020}.

The structural obstruction of \cite{Ball_marsder_1982} does not exclude \emph{exact controllability to the trajectories}.
 This concept was first investigated in \cite{Alabau_Cannarsa_Urbani_21, Alabau_cannarsa_Urbani_22} in an abstract framework for parabolic PDEs with scalar bilinear controls, where local and semi-global controllability results were obtained. The approach was subsequently extended in \cite{Cannarsa_Duca_Urbani_22, Buffe_Duca_25} to address exact controllability to eigensolutions. Motivated by the recent developments, the present study aims to address the problem of global exact controllability for fourth-order parabolic equations driven by bilinear controls, within a geometric control framework.

\subsection{Structure of the paper}
The rest of the paper is organized as follows. In \Cref{sec_pre}, we present preliminary results, including well-posedness, the saturation limit property, and the required density of the saturating subspace. Based on these results, \Cref{sec_smalltime} is devoted to the proof of the main approximate controllability result \Cref{main_thm_approx}. The key ingredient of the main theorem, namely the saturation limit property \Cref{conjugate_limit}, together with the stability estimates \Cref{existence}--\Cref{point_1}, is established in \Cref{section_propn_proofs}. Finally, \Cref{sec_con} addresses global exact controllability to constant states, where the proofs of \Cref{global_exact} and \Cref{global_exact1} are provided.

\section{Preliminaries}\label{sec_pre}
The aim of this section is to present some preliminary results. We first introduce the functional spaces employed in our analysis. Next, we state the existence and uniqueness of solutions to equation \eqref{ctrl_prblm}. Finally, we state a saturation limit result for the conjugated dynamics, which plays a key role in the proof of the main theorem.

\subsection{Function spaces and notations}\label{sec_function}
Let $u\in L^{2}(\mathbb{T})$ admit the Fourier expansion
\begin{align}\label{fourier_expansion}
u(x)=\sum_{k\in\mathbb{Z}} \widehat{u}_k e^{ikx}, \qquad \; \widehat{u}_{-k}=\overline{\widehat{u}_k},
\end{align}
with the convergence of the expansion in $L^2(\mathbb{T})$ provided that $\sum_{k\in\mathbb{Z}} | \widehat{u}_k|^2 < \infty.$
For $s\ge 0$, the Sobolev space $H^s := H^s(\mathbb{T})$ consists of all functions given by \eqref{fourier_expansion} such that $\sum_{k\in\mathbb{Z}} (1+|k|^{2})^{s}|\widehat{u}_k|^{2}<\infty.$ The associated norm $\|\cdot\|_s$ is defined by
\[
\|u\|_{s}^{2}
:= \sum_{k\in\mathbb{Z}} (1+|k|^{2})^{s}|\widehat{u}_k|^{2}.
\]
 When $s=0$ it coincides with the $L^{2}$–norm; that is, $\| \cdot \|_0 : = \| \cdot\|$.  For $s \ge 0$ and $u=\displaystyle\sum_{k\in\mathbb{Z}} \widehat{u}_k e^{ikx}\in H^{s}(\mathbb{T})$, we set
\[
\partial_x^{s} u := \sum_{k\in\mathbb{Z}} |k|^{s} \widehat{u}_k e^{ikx}.
\]
For $s>0$, the $H^s(\mathbb{T})$-norm defined above satisfies the equivalence
\begin{equation*}
\|u\|_{s} 
\simeq 
\|u\| 
+ \|\pa_x^s u\|.
\end{equation*}
\begin{itemize}
\item Throughout this article, the symbol $C $ will represent a generic positive constant. Its value may change from one occurrence to the next. Whenever the dependence of such a constant on specific parameters is relevant, it will be indicated explicitly.

\item We denote by $\mathcal{R}^{KS}_t(u_0,p)$ and $\mathcal{R}^{CH}_t(u_0,p)$ the solution of \eqref{ctrl_prblm} at time $t$, associated with the initial datum $u_0$ and the control $p$, corresponding to the nonlinearities $\mathcal{N}_{KS}$ and $\mathcal{N}_{CH}$, respectively. We implicitly restrict ourselves to situations in which such solutions are well defined. For convenience, throughout the paper, whenever the notation $\mathcal{R}_t$ appears without the superscript $KS$ or $CH$, the corresponding statements are understood to hold for both cases.
\end{itemize}

\subsection{Local well-posedness and semi-global stability}
Let us state some important well-posedness results for the system under consideration.
\begin{proposition}\label{existence}
Assume that $s > \tfrac12$, $u_0 \in H^s(\mathbb{T})$, and $p \in L^2_{\mathrm{loc}}(\mathbb{R}^+;\mathbb{R}^3)$. 
Then there exists a time $T_* = T_*(u_0,p) > 0$ such that the system \eqref{ctrl_prblm}--\eqref{fourier_mods} admits a unique solution 
\[
u \in C\big([0,T_*]; H^s(\mathbb{T})\big).
\]
\begin{enumerate}
    \item\label{point_1} In addition, the following semi-global stability property holds. Let $R>0$ and let 
$p \in L^2_{\mathrm{loc}}(\mathbb{R}^+;\mathbb{R}^3)$. For any $u_0, v_0 \in H^s(\mathbb{T})$ with 
$\|u_0\|_{s}, \, \|v_0\|_{s} \le R$, there exist a time $T^* = T^*(R,p) > 0$ and a constant $C\left(R,\norm{p}_{L^2}\right) > 0$ such that
\begin{equation}\label{stability}
\| \mc R_t(u_0,p) - \mc R_t(v_0,p) \|_{s} 
\le C \| u_0 - v_0 \|_{s}, 
\qquad \forall\, t \in [0,T^*],
\end{equation}
where $\mc R_t(u_0,p)$ and $\mc R_t(v_0,p)$ denote the solutions of \eqref{ctrl_prblm}--\eqref{fourier_mods}  corresponding to the initial data 
$u_0$ and $v_0$, respectively.
\item\label{point_2} Set $\Lambda := \|u\|_{C([0,T_*];H^s(\T))}
     + \|u_0\|_{H^s(\T)}
     + \|p\|_{L^2((0,T_*);\mathbb{R}^3)}.$
There exists a constant $
\delta = \delta\bigl(T_*(u_0,p),\, \Lambda\bigr) > 0$ such that, for any $\widehat{u}_0 \in H^s(\mathbb{T})$ and
$\widehat{p} \in L^2\bigl((0,T);\mathbb{R}^{3}\bigr)$ satisfying $
\|\widehat{u}_0-u_0\|_{s}
+ \|\widehat{p}-p\|_{L^2((0,T);\mathbb{R}^{3})} < \delta,$ equation \eqref{ctrl_prblm}--\eqref{fourier_mods}  admits a unique mild solution 
$
\widehat{u} \in C\big([0,T_*];H^s(\mathbb{T})\big)
$
with initial condition $\widehat{u}_0$ and control $\widehat{p}$.
\end{enumerate}
\end{proposition}
\begin{proof}
Local well-posedness for \eqref{ctrl_prblm} can be established using a fixed-point argument, following a similar strategy as in 
Proposition~2.1 of \cite{Duca_Pozzoli_Urbai_JMPA_2025}. Since this argument is quite standard, we skip the details of the proof of this part. For the sake of completeness, the proof of \Cref{point_1} is provided in \Cref{section_propn_proofs}. 
\end{proof}
\subsection{Small-time limit of conjugated dynamics}
We begin by stating the following small-time limit result for the conjugated dynamics, which is a key ingredient in the proof of our main result \Cref{main_thm_approx}.
\begin{proposition}\label{conjugate_limit}
Let $s > \frac{1}{2}$  and $ u_0 \in H^s(\mathbb{T}).$ Assume that $ p = (p_1, p_2, p_3) \in \mathbb{R}^3$, $\varphi \in H^{2s + 4}(\mathbb{T})$, and  $\varphi \geq 0$. Then, there exists a constant $\delta_0 > 0$, such that for any $\delta \in (0, \delta_0),$ the solution $\mathcal{R}( e^{-\delta^{-\frac{1}{4}} \varphi} u_0, \delta^{-1} p)$ of \eqref{ctrl_prblm}--\eqref{fourier_mods} is well-defined in $[0, \delta].$ Furthermore the following limit holds:
$$e^{\delta^{- \frac{1}{4}} \varphi} \mathcal{R}_{\delta}( e^{-\delta^{-\frac{1}{4}} \varphi} u_0, \delta^{-1} p) \to e^{ - \left(\varphi'\right)^4 + \langle p, \mu \rangle} u_0 \ \text{ in } H^s, \ \text{ as } \delta \to 0^+.$$
\end{proposition}
The proof of this result is postponed to \Cref{sec_asymp}.

{
This asymptotic behaviour is the starting point of the iterative
construction used to establish small-time approximate controllability.
Indeed, Proposition~\ref{conjugate_limit} allows us, after suitable
conjugations, to generate multiplicative factors involving terms of the
form $-(\psi')^4$. On the other hand, the algebraic construction of
\Cref{constr} shows that, for every $N\geq0$ and every
$\phi\in\mathcal V_{N+1}$ (cf. \eqref{saturationspace}), there exist $d \geq 1$ and
$\phi_0,\phi_1,\ldots,\phi_d\in\mathcal V_N$ such that
\begin{equation}\label{asymp_h}
	\phi=\phi_0-\sum_{k=1}^{d}(\phi_k')^4.
\end{equation}
This representation is particularly useful since each $\mathcal V_N$ is
a vector space containing the constants. In particular, if
$\psi\in\mathcal V_N$, then $\psi+c\in\mathcal V_N$ for every
$c\in\mathbb R$, and any real scalar multiple of $\psi$ also belongs to
$\mathcal V_N$. These properties allow us to combine
Proposition~\ref{conjugate_limit} with an induction argument to
approximately generate the factor $e^{-(\psi')^4}$.
More precisely, starting from the three-dimensional control space $\mathcal V_0=\operatorname{span}\{1,\cos x,\sin x\},$
the available control directions allow us to approximately reach
$e^\phi u_0$ for every $\phi\in\mathcal V_0$ in arbitrarily small time.
The representation \eqref{asymp_h} then allows us to proceed
inductively. Assuming that $e^\phi u_0$ is approximately reachable in
arbitrarily small time for every $\phi\in\mathcal V_N$, we show that
the same property holds for every $\phi\in\mathcal V_{N+1}$. Indeed,
each $\phi\in\mathcal V_{N+1}$ admits the representation
\eqref{asymp_h}, where all the functions
$\phi_0,\phi_1,\ldots,\phi_d$ belong to $\mathcal V_N$, so that the
induction hypothesis can be applied successively to generate the
corresponding factors. Consequently, for every $N\geq0$ and every
$\phi\in\mathcal V_N$, the state $e^\phi u_0$ is approximately
reachable from $u_0$ in arbitrarily small time by means of controls
taking values only in $\mathcal V_0$.
Finally, since $\bigcup_{N=0}^{\infty}\mathcal V_N$
is the space of real trigonometric polynomials and is dense in
$H^s(\mathbb T)$, a density argument yields small-time approximate
controllability to every state of the form $e^\phi u_0$, with
$\phi\in H^s(\mathbb T)$; see
\Cref{prpn_exp_pertubed_initial}. This property is then used, together
with the assumptions on the initial datum $u_0$ and the target $u_1$,
to prove the small-time $L^2$-approximate controllability stated in
\Cref{main_thm_approx}--\Cref{point_11}. Finally, the
$H^s$-approximate controllability result for an arbitrary prescribed
time $T>0$ is obtained by first steering the system, in small time, to
a neighbourhood of a suitable stationary state, then maintaining the
trajectory near this state for the required intermediate time, and
finally steering the solution, again in small time, to a neighbourhood
of the target $u_1$; see
\Cref{main_thm_approx}--\Cref{point_12}.}


Motivated by the above discussion, we introduce and study the so-called saturation property, which plays a fundamental role in establishing global approximate controllability of \eqref{ctrl_prblm}.
\subsection{Saturating subspaces}\label{sec_sat} 
{
Let us consider the following subspaces: 	\begin{equation}\label{saturationspace}
	\mathcal V_m:=\operatorname{span}\{1,\sin(kx),\cos(kx):1\leq k\leq m+1\}, m \geq 0.
\end{equation}
We also denote $\mathcal{V}_\infty := \bigcup_{j=0}^{\infty} \mathcal{V}_j.$
As $\mathcal{V}_\infty$ contains all trigonometric modes, it is clear that
	for every $s \ge 0$, the space $\mathcal{V}_\infty$ is dense in $H^{s}(\mathbb{T})$.
Next, we prove the following representation result for the elements of these vector spaces.
\begin{proposition}\label{constr}
	For every $m\geq0$ and every $\phi\in\mathcal V_{m+1}$, there exist
	$d \geq 1$ and functions
	$\phi_0,\phi_1,\ldots,\phi_d\in\mathcal V_{m}$ such that
	\begin{equation*}
		\phi=\phi_0-\sum_{k=1}^{d}(\phi_k')^4.
	\end{equation*}
\end{proposition}
\begin{proof}
We split the proof into cases. We first consider $m=0$, and then treat the remaining cases $m \geq 1$.

\noindent
\textit{\underline{Case $m=0$}.}	Let
	$\phi=a_0+a_1\cos x+b_1\sin x+a_2\cos(2x)+b_2\sin(2x)\in \mathcal V_1$ be arbitrary.
	For $\rho \in (-1,1)$, define
	\begin{equation*}
		Q_\rho(x):=(1+\rho)\cos^4x
		+(1-\rho)\cos^4\left(x+\frac{\pi}{2}\right)
		+\cos^4\left(x+\frac{\pi}{4}\right)
		+\cos^4\left(x-\frac{\pi}{4}\right).
	\end{equation*}
	Using the identity
	$\cos^4y=\frac{3}{8}+\frac{1}{2}\cos(2y)+\frac{1}{8}\cos(4y)$, we obtain
	\begin{equation}\label{eq:Qrho-cos}
		Q_\rho(x)=\frac{3}{2}+\rho\cos(2x).
	\end{equation}
	Choose $K_c>|a_2|$ and set $\rho_c=-a_2/K_c$. Then $|\rho_c|<1$, and
	\begin{equation*}
		a_2\cos(2x)=\frac{3K_c}{2}-K_cQ_{\rho_c}(x).
	\end{equation*}
	Notice that all the coefficients occurring in $K_cQ_{\rho_c}$ are positive.
	Hence every term therein can be written as the fourth power of the derivative
	of an element of $\mathcal V_0$. For instance,
	\begin{equation*}
		K_c(1+\rho_c)\cos^4x
		=
		\left[\left(K_c(1+\rho_c)\right)^{1/4}\sin x\right]'^{\,4},
	\end{equation*}
	and the same holds for the remaining terms of $K_cQ_{\rho_c}$, since
	$\sin(x+\theta)\in\mathcal V_0$ for every $\theta\in\mathbb R$.
	Similarly, replacing $x$ by $x-\pi/4$ in \eqref{eq:Qrho-cos} and using
	$\cos(2x-\pi/2)=\sin(2x)$, we obtain
	\begin{equation*}
		Q_\rho\left(x-\frac{\pi}{4}\right)
		=\frac{3}{2}+\rho\sin(2x).
	\end{equation*}
	Choose $K_s>|b_2|$ and set $\rho_s=-b_2/K_s$. Then
	\begin{equation*}
		b_2\sin(2x)=\frac{3K_s}{2}
		-K_sQ_{\rho_s}\left(x-\frac{\pi}{4}\right).
	\end{equation*}
	Consequently,
	\begin{align*}
		g(x)
		={}&a_0+a_1\cos x+b_1\sin x+\frac{3}{2}(K_c+K_s)
		-K_cQ_{\rho_c}(x)
		-K_sQ_{\rho_s}\left(x-\frac{\pi}{4}\right).
	\end{align*}
	Setting
	$\phi_0=a_0+a_1\cos x+b_1\sin x+\frac{3}{2}(K_c+K_s)\in\mathcal V_0$,
	the last two terms of $\phi$ can be written as a finite sum
	$\sum_{\ell=1}^{8}(\phi_\ell')^4$, with $\phi_\ell\in\mathcal V_0$.
	Hence
	\begin{equation*}
		\phi=\phi_0-\sum_{\ell=1}^{8}(\phi_\ell')^4.
	\end{equation*}

	\noindent
\textit{\underline{Case $m\geq 1$}.}
 By definition,
	$\mathcal V_{m+1}
	=\mathcal V_{m}
	+\operatorname{span}\{\sin((m+2)x),\cos((m+2)x)\}.$
Thus, the only new Fourier modes appearing in $\mathcal V_{m+1}$ are
those of frequency $m+2$. 
Let us first consider
\begin{equation}\label{eq:psi}
\psi(x):=\sin x+\frac{1}{m+1}\sin((m+1)x)\in\mathcal V_m.\end{equation}
Then $\psi'(x)=\cos x+\cos((m+1)x)$. Set
\begin{equation*}
	F(x):=(\psi'(x))^4
	=A_0+\sum_{r=1}^{4(m+1)}A_r\cos(rx).
\end{equation*} 
We first show that $
	A_{m+2}\neq0.$
For $m=1$, a direct expansion gives
	$(\cos x+\cos(2x))^4
	=A_0+\sum_{r=1}^{8}A_r\cos(rx),$
with $A_3=3$. Hence $A_3\neq0$.
For $m=2$, similarly,
$(\cos x+\cos(3x))^4
	=A_0+\sum_{r=1}^{12}A_r\cos(rx)$,
and the coefficient of $\cos(4x)$ is
	$A_4=\frac{31}{8}\neq0.$
Let now $m\geq3$. Expanding
$F(x)$
we obtain
\begin{align*}
	F(x)
	&={}\frac94
	+2\cos(2x)+\frac18\cos(4x)
+2\cos(2(m+1)x)+\frac18\cos(4(m+1)x)+3\left[\cos((m+2)x)+\cos(mx)\right]
	\\
    &+\frac12\left[\cos((m+4)x)+\cos((m-2)x)\right]
	+\frac34\left[
	\cos(2(m+2)x)+\cos(2mx)\right]+\frac12\left[
	\cos((3m+4)x)+\cos((3m+2)x)\right].
\end{align*}
For $m\geq3$, none of the frequencies
$
2, 4, 2(m+1), 4(m+1), m, m+4, m-2, 2m, 2(m+2), 3m+2, 3m+4
$
coincides with $m+2$. Therefore, the coefficient $A_{m+2}$ of
$\cos((m+2)x)$ is exactly 3.
Hence $A_{m+2}\neq 0$ holds for every $m\geq 1$.

We next isolate the $(m+2)$-th Fourier mode by means of a finite family of
translations. Choose an integer $L>5m+6$ and set
$\theta_\ell=2\pi\ell/L$, $\ell=0,\ldots,L-1$. We shall use the following identity
\begin{equation}\label{eq:discrete-orthogonality}
	\sum_{\ell=0}^{L-1}e^{ik\theta_\ell}
	=
	\begin{cases}
		L, & k\in L\mathbb Z,\\
		0, & k\notin L\mathbb Z.
	\end{cases}
\end{equation}
Given $K_c>0$ and $\rho_c\in(-1,1)$, define
\begin{equation*}
	\alpha_\ell
	:=\frac{K_c}{L}\left(1+\rho_c\cos((m+2)\theta_\ell)\right).
\end{equation*}
Since $|\rho_c|<1$, we have
$1+\rho_c\cos((m+2)\theta_\ell)\geq1-|\rho_c|>0$, and hence
$\alpha_\ell>0$ for every $\ell=0,\ldots,L-1$. We claim that
\begin{equation}\label{eq:cos-filter}
	\sum_{\ell=0}^{L-1}\alpha_\ell F(x+\theta_\ell)
	=
	K_cA_0+\frac{K_c\rho_cA_{m+2}}{2}\cos((m+2)x).
\end{equation}
Indeed, recalling that
$F(x)=A_0+\sum_{r=1}^{4(m+1)}A_r\cos(rx)$, we have
\begin{align*}
	\sum_{\ell=0}^{L-1}\alpha_\ell F(x+\theta_\ell)
	={}&\frac{K_c}{L}\sum_{\ell=0}^{L-1}F(x+\theta_\ell)
	+\frac{K_c\rho_c}{L}
	\sum_{\ell=0}^{L-1}\cos((m+2)\theta_\ell)F(x+\theta_\ell).
\end{align*}
Since $1\leq r\leq 4(m+1)<L$, $r\notin L\mathbb Z$. Thus, \eqref{eq:discrete-orthogonality} yields 
	$\displaystyle \sum_{\ell=0}^{L-1}\cos(r\theta_\ell)
	=
	\sum_{\ell=0}^{L-1}\sin(r\theta_\ell)=0.$
Consequently,
\begin{align}\label{eq:average-F}
	\sum_{\ell=0}^{L-1}F(x+\theta_\ell)
	&=LA_0+
	\sum_{r=1}^{4(m+1)}A_r
	\sum_{\ell=0}^{L-1}\cos(r(x+\theta_\ell))
	=LA_0,
\end{align}
Note that, since $0<m+2<L$, we have $m+2\notin L\mathbb Z$. Therefore it follows that $\sum_{\ell=0}^{L-1}\cos((m+2)\theta_\ell)=0$ and we obtain
\begin{align*}
	&\sum_{\ell=0}^{L-1}
	\cos((m+2)\theta_\ell)F(x+\theta_\ell)
	=
	\sum_{r=1}^{4(m+1)}A_r
	\sum_{\ell=0}^{L-1}
	\cos((m+2)\theta_\ell)\cos(r(x+\theta_\ell)).
\end{align*}
 Using
$\cos(r(x+\theta_\ell))
	=\cos(rx)\cos(r\theta_\ell)-\sin(rx)\sin(r\theta_\ell),$ we obtain
\begin{align*}
	\sum_{\ell=0}^{L-1}
	\cos((m+2)\theta_\ell)\cos(r\theta_\ell)
	&=
	\frac12\sum_{\ell=0}^{L-1}
	\left[
	\cos((r-m-2)\theta_\ell)+\cos((r+m+2)\theta_\ell)
	\right],\\
	\sum_{\ell=0}^{L-1}
	\cos((m+2)\theta_\ell)\sin(r\theta_\ell)
	&=
	\frac12\sum_{\ell=0}^{L-1}
	\left[
	\sin((r+m+2)\theta_\ell)+\sin((r-m-2)\theta_\ell)
	\right].
\end{align*}
Since $1\leq r\leq4(m+1)$, we have $
	r+m+2\leq4(m+1)+(m+2)=5m+6<L.$
Moreover, if $r\neq m+2$, then $0<|r-m-2|<L$. Thus, by
\eqref{eq:discrete-orthogonality}, all the preceding sums vanish for
$r\neq m+2$. For $r=m+2$, we have
\begin{equation*}
	\sum_{\ell=0}^{L-1}\cos^2((m+2)\theta_\ell)=\frac{L}{2},
	\qquad
	\sum_{\ell=0}^{L-1}
	\cos((m+2)\theta_\ell)\sin((m+2)\theta_\ell)=0.
\end{equation*}
Consequently,
\begin{equation}\label{eq:weighted-cos-sum}
	\sum_{\ell=0}^{L-1}
	\cos((m+2)\theta_\ell)F(x+\theta_\ell)
	=
	\frac{LA_{m+2}}{2}\cos((m+2)x).
\end{equation}
Combining \eqref{eq:average-F} and \eqref{eq:weighted-cos-sum}, one gets
\eqref{eq:cos-filter}.

Similarly, given $K_s>0$ and $\rho_s\in(-1,1)$, define
\begin{equation*}
	\beta_\ell
	:=\frac{K_s}{L}\left(1+\rho_s\sin((m+2)\theta_\ell)\right).
\end{equation*}
Since $|\rho_s|<1$, we have
$\beta_\ell\geq(K_s/L)(1-|\rho_s|)>0$ for every
$\ell=0,\ldots,L-1$. Proceeding as above and using
\begin{align*}
	\sum_{\ell=0}^{L-1}
	\sin((m+2)\theta_\ell)\cos(r\theta_\ell)&=0, \qquad
	\sum_{\ell=0}^{L-1}
	\sin((m+2)\theta_\ell)\sin(r\theta_\ell)
	=
	\begin{cases}
		L/2, & r=m+2,\\
		0, & r\neq m+2,
	\end{cases}
\end{align*}
we obtain$
	\sum_{\ell=0}^{L-1}
	\sin((m+2)\theta_\ell)F(x+\theta_\ell)
	=
	-\frac{LA_{m+2}}{2}\sin((m+2)x).$
Therefore the following identity holds
\begin{equation}\label{eq:sin-filter}
	\sum_{\ell=0}^{L-1}\beta_\ell F(x+\theta_\ell)
	=
	K_sA_0-\frac{K_s\rho_sA_{m+2}}{2}\sin((m+2)x).
\end{equation}
Let $\phi\in\mathcal V_{m+1}$ be arbitrary. We can write
\begin{equation*}
	\phi=g_0+a\cos((m+2)x)+b\sin((m+2)x),
	\qquad
	g_0\in\mathcal V_m,\quad a,b\in\mathbb R.
\end{equation*}
Since $A_{m+2}\neq0$, choose $
	K_c>\frac{2|a|}{|A_{m+2}|},
	K_s>\frac{2|b|}{|A_{m+2}|},$
and set $
	\rho_c:=-\frac{2a}{K_cA_{m+2}},
	\,
	\rho_s:=\frac{2b}{K_sA_{m+2}}.$
Then $\rho_c,\rho_s\in(-1,1)$ and
\begin{equation}\label{eq:choice-rho-K}
	-\frac{K_c\rho_cA_{m+2}}{2}=a,
	\qquad
	\frac{K_s\rho_sA_{m+2}}{2}=b.
\end{equation}
For $\ell=0,\ldots,L-1$, define
\begin{equation*}
	\phi_\ell^c(x)
	:=\alpha_\ell^{1/4}\psi(x+\theta_\ell),
	\qquad
	\phi_\ell^s(x)
	:=\beta_\ell^{1/4}\psi(x+\theta_\ell).
\end{equation*}
Let us recall the function $\psi\in\mathcal V_m$ defined in \eqref{eq:psi}. Since $\mathcal V_m$ is a vector space, definition of $\psi$ implies
$\phi_\ell^c,\phi_\ell^s\in\mathcal V_m,
	\, \ell=0,\ldots,L-1.$
Moreover, by the definition of $F$, we have
\begin{equation*}
	\left((\phi_\ell^c)'\right)^4
	=\alpha_\ell F(x+\theta_\ell),
	\qquad
	\left((\phi_\ell^s)'\right)^4
	=\beta_\ell F(x+\theta_\ell).
\end{equation*}
Setting
	$\phi_0:=g_0+(K_c+K_s)A_0,$
we also have $\phi_0\in\mathcal V_m$. Combining
\eqref{eq:cos-filter}, \eqref{eq:sin-filter}, and
\eqref{eq:choice-rho-K}, we obtain
\begin{align*}
	&\phi_0
	-\sum_{\ell=0}^{L-1}\left((\phi_\ell^c)'\right)^4
	-\sum_{\ell=0}^{L-1}\left((\phi_\ell^s)'\right)^4\\
	&\qquad
	=g_0-\frac{K_c\rho_cA_{m+2}}{2}\cos((m+2)x)
	+\frac{K_s\rho_sA_{m+2}}{2}\sin((m+2)x)\\
	&\qquad
	=g_0+a\cos((m+2)x)+b\sin((m+2)x)=\phi.
\end{align*}
Thus, there exist $d=2L$ and
$\phi_0,\phi_1,\ldots,\phi_d\in\mathcal V_m$ such that
\begin{equation*}
	\phi=\phi_0-\sum_{k=1}^{d}(\phi_k')^4.
\end{equation*}
This completes the proof.
	\end{proof}}

\subsection{Small-time global approximate null controllability}
The conjugated dynamics limit \Cref{conjugate_limit}, together with the saturation property \Cref{constr}, plays a crucial role in establishing approximate controllability. In particular, \Cref{conjugate_limit} immediately implies that the trajectory can be driven arbitrarily close to the origin. More precisely, a straightforward observation is the following: for any $u_0 \in H^s(\T)$ and $\varepsilon > 0$, one can choose a constant $r > 0$ such that
\[
e^{r}\varepsilon > 2\left\|u_0\right\|_s \quad \Longrightarrow \quad \left\|e^{-r}u_0\right\|_s < \frac{\varepsilon}{2}.
\]
Since $-r \in \mathcal{V}_0$, one can  write $-r = \sum_{i=1}^{3} p_i \mu_i$ for some vector $\widehat{p} := (p_1,p_2,p_3) \in \R^3$. Applying the conjugated dynamics limit (\Cref{conjugate_limit}) with $\varphi = 0$ and this particular choice of $p = \widehat{p}$, we obtain a time $\delta > 0$ such that the solution of \eqref{ctrl_prblm}--\eqref{fourier_mods} is well-defined on $[0,\delta]$ and satisfies
\[
\big\| \mathcal{R}_{\delta}\left(u_0, \delta^{-1}\widehat{p}\right) - e^{-r}u_0 \big\|_s < \frac{\varepsilon}{2}.
\]
Using the triangle inequality, and setting $q := \widehat{p}/\delta$, we deduce
\[
\left\|\mathcal{R}_{\delta}\left(u_0, q\right)\right\|_s 
\le 
\big\| \mathcal{R}_{\delta}\left(u_0, \delta^{-1}\widehat{p}\right) - e^{-r}u_0 \big\|_s
+ \left\|e^{-r}u_0\right\|_s
< \varepsilon.
\]
This shows that any initial state $u_0$ can be driven arbitrarily close to zero in an arbitrarily small time. In the literature, this property is referred to as small-time global approximate null controllability. A natural question then arises: Can one design controls that steer the system from a given initial state to an arbitrarily close, preassigned target state, possibly under certain conditions on the nature of the data? Furthermore, can such controllability be achieved at a prescribed time? \Cref{sec_smalltime} addresses these questions and is devoted to establishing the corresponding controllability results below.
\subsection{Concatenation property}
We end this section by discussing the concatenation of two scalar controls. Let us recall that the concatenation $p * q$ of two scalar control laws $p: [0, T_1] \to \mathbb{R}, \ q: [0, T_2] \to \mathbb{R}$ is the control law defined on $[0, T_1 + T_2]$ as follows 
\begin{align}\label{con}
(p * q)(t) = \begin{cases}
p(t), & t \in [0, T_1]\\
q(t - T_1) , & t \in (T_1, T_1 + T_2].
\end{cases}
\end{align}
Such a definition naturally extends componentwise to controls taking values in $\mathbb{R}^3$. Assume that for the control $p * q,$ the solution of \eqref{ctrl_prblm} with initial data $u_0$ exists for the time interval $[0, {T}]$ where ${T} \in (T_1, T_1 + T_2]$ then the associated flow satisfies the concatenation property 
\begin{align}\label{flow_prpty}
\mathcal{R}_{T_1 + t}\left(u_0, p * q\right) = \mathcal{R}_t\left( \mathcal{R}_{T_1}\left(u_0, p\right), q\right), \ 0 < t < {T} - T_1.
\end{align}

\section{Small-time approximate controllability}\label{sec_smalltime}
This section is devoted to proving the small-time approximate controllability result stated in \Cref{main_thm_approx}. We begin by discussing the following property of the dynamics of \eqref{ctrl_prblm}--\eqref{fourier_mods} in small time. 
\begin{proposition}\label{prpn_exp_pertubed_initial}
	Let $s>\frac{1}{2}$ and $u_0,\phi\in H^s(\mathbb T)$. For any
	$\varepsilon,T>0$, there exist $\tau\in(0,T]$ and
	$p\in L^2((0,\tau);\mathbb R^3)$ such that the solution
	$\mathcal R(u_0,p)$ of \eqref{ctrl_prblm}--\eqref{fourier_mods}
	is well-posed in $[0,\tau]$ and
	\begin{equation*}
		\left\|\mathcal R_\tau(u_0,p)-e^\phi u_0\right\|_s<\varepsilon.
	\end{equation*}
\end{proposition}

\begin{proof}
{	Recall that $
		\mathcal V_N=\operatorname{span}
		\{1,\sin(kx),\cos(kx):1\leq k\leq N+1\}.$
	We first prove the following property for every $N\geq 0$:
	\begin{itemize}
		\item[($P_N$)] For any $u_0\in H^s(\mathbb T)$,
		$\phi\in \mathcal V_N$, and any $\varepsilon,T>0$, there exist
		$\tau\in(0,T)$ and a piecewise constant control
		$p:[0,\tau]\to\mathbb R^3$, with
		$p\in L^2((0,\tau);\mathbb R^3)$, such that the corresponding
		solution of \eqref{ctrl_prblm}--\eqref{fourier_mods} is
		well-posed in $[0,\tau]$ and
		\begin{equation*}
			\left\|\mathcal R_\tau(u_0,p)-e^\phi u_0\right\|_s
			<\varepsilon.
		\end{equation*}
	\end{itemize}
	We argue by induction on $N$.
	
	\paragraph{$\bullet$ \textit{Base case}: $N=0$.}
	Let $\phi\in \mathcal V_0=\operatorname{span}\{1,\cos x,\sin x\}$.
	By the definition of the functions $\mu_i$ (see \eqref{fourier_mods}), there exists
	$\lambda=(\lambda_1,\lambda_2,\lambda_3)\in\mathbb R^3$ such that
	$\phi=\langle\lambda,\mu\rangle$. Proposition~\ref{conjugate_limit},
	applied with $\varphi=0$ and $p=\lambda$, gives
	\begin{equation*}
		\mathcal R_\delta
		\left(u_0,\delta^{-1}\lambda\right)
		\longrightarrow e^{\langle\lambda,\mu\rangle}u_0
		=e^\phi u_0
		\quad\text{in }H^s,
		\qquad \text{ as }\delta\to0^+.
	\end{equation*}
	Hence there exists $\tau\in(0,T)$ such that, for the constant
	control $p^\tau=\lambda/\tau$,
	\begin{equation*}
		\left\|\mathcal R_\tau(u_0,p^\tau)-e^\phi u_0\right\|_s
		<\varepsilon.
	\end{equation*}
	Thus $(P_0)$ holds.
	
	\paragraph{$\bullet$ \textit{Inductive step}: $N\Longrightarrow N+1$.}
	Assume that $(P_N)$ holds. 
    We shall prove
	$(P_{N+1})$ holds as well.   
	Let $\phi\in\mathcal V_{N+1}$. By \Cref{constr}, there exist
	$d\geq 1$ and $\phi_0,\phi_1,\ldots,\phi_d\in\mathcal V_N$
	such that
	\begin{equation}\label{eq:phi-VN-representation}
		\phi=\phi_0-\sum_{k=1}^{d}(\phi_k')^4.
	\end{equation}
	We first prove that, for every $\psi\in \mathcal V_N$, the state
	$e^{-(\psi')^4}u_0$ can be reached approximately in arbitrarily
	small time.
	Let $\psi\in \mathcal V_N$. Choose a constant $c>0$ sufficiently large so
	that
	$\widetilde\psi:=\psi+c>0$. Since $\mathcal V_N$ is a vector space, $\widetilde\psi\in \mathcal V_N$, and consequently
	\begin{equation}\label{eq:plus-minus-VN}
		\pm a\widetilde\psi\in \mathcal V_N
		\qquad\text{for every }a>0.
	\end{equation}
	Moreover, $(\widetilde\psi')^4=(\psi')^4$. Proposition
	~\ref{conjugate_limit}, applied with
	$\varphi=\widetilde\psi$ and $p=0$, yields
	\begin{equation*}
		e^{\delta^{-1/4}\widetilde\psi}
		\mathcal R_\delta
		\left(e^{-\delta^{-1/4}\widetilde\psi}u_0,0\right)
		\longrightarrow e^{-(\psi')^4}u_0
		\quad\text{in }H^s,
		\qquad\delta\to0^+.
	\end{equation*}
	Thus, there exists $\tau_2\in(0,T/3)$ such that
	\begin{equation}\label{ineq_1}
		\left\|
		e^{\tau_2^{-1/4}\widetilde\psi}
		\mathcal R_{\tau_2}
		\left(e^{-\tau_2^{-1/4}\widetilde\psi}u_0,0\right)
		-e^{-(\psi')^4}u_0
		\right\|_s
		<\frac{\varepsilon}{2}.
	\end{equation}
	By \eqref{eq:plus-minus-VN},
	$-\tau_2^{-1/4}\widetilde\psi\in \mathcal V_N$. Hence, by the induction
	hypothesis $(P_N)$, for any $\varepsilon_1>0$ there exist
	$\tau_1\in(0,T/3)$ and a piecewise constant control
	$p^{\tau_1}:[0,\tau_1]\to\mathbb R^3$ such that
	\begin{equation*}
		\left\|
		\mathcal R_{\tau_1}(u_0,p^{\tau_1})
		-e^{-\tau_2^{-1/4}\widetilde\psi}u_0
		\right\|_s
		<\varepsilon_1.
	\end{equation*}
	Since
	$\mathcal R(e^{-\tau_2^{-1/4}\widetilde\psi}u_0,0)$ is
	well-defined in $[0,\tau_2]$, by choosing $\varepsilon_1$
	sufficiently small and using
	\Cref{existence}--\Cref{point_2}, the solution
		$\mathcal R
		\left(\mathcal R_{\tau_1}(u_0,p^{\tau_1}),0\right)$
	is also well-defined in $[0,\tau_2]$. Hence, by
	\eqref{con}--\eqref{flow_prpty},
	$\mathcal R(u_0,p^{\tau_1}*0|_{[0,\tau_2]})$ is well-defined
	in $[0,\tau_1+\tau_2]$. Moreover, by the stability property 	\Cref{existence}--\Cref{point_1},
	there exists $C_1=C_1(\tau_2)>0$ such that 
\begin{align}\label{ineq_6}
		\notag\left\|
		\mathcal R_{\tau_1+\tau_2}
		\left(u_0,p^{\tau_1}*0|_{[0,\tau_2]}\right)
		-\mathcal R_{\tau_2}
		\left(e^{-\tau_2^{-1/4}\widetilde\psi}u_0,0\right)
		\right\|_s 
        &=\left\| \mathcal{R}_{\tau_2}\left( \mathcal{R}_{\tau_1}\left(u_0, p^{\tau_1}\right),0\right) -  \mathcal{R}_{\tau_2}\left( e^{-\tau_2^{-1/4} \widetilde{\psi}} u_0, 0\right)\right\|_s\\
		&\qquad\leq C_1
		\left\|
		\mathcal R_{\tau_1}(u_0,p^{\tau_1})
		-e^{-\tau_2^{-1/4}\widetilde\psi}u_0
		\right\|_s
		<C_1\varepsilon_1.
	\end{align}
	Set
		$\widehat u_0:=
		\mathcal R_{\tau_2}
		\left(e^{-\tau_2^{-1/4}\widetilde\psi}u_0,0\right)
		\in H^s(\mathbb T).$
	Since $\tau_2^{-1/4}\widetilde\psi\in \mathcal V_N$, the induction
	hypothesis $(P_N)$, now applied with initial datum $\widehat u_0$,
	gives $\tau_3\in(0,T/3)$ and a piecewise constant control
	$p^{\tau_3}:[0,\tau_3]\to\mathbb R^3$ such that
	\begin{equation}\label{ineq_3}
		\left\|
		\mathcal R_{\tau_3}(\widehat u_0,p^{\tau_3})
		-e^{\tau_2^{-1/4}\widetilde\psi}\widehat u_0
		\right\|_s
		<\varepsilon_1.
	\end{equation}
A similar argument as above leads to the existence of the solution $\mathcal R\left( \mathcal{R}_{\tau_1 + \tau_2}\left(u_0, p^{\tau_1} * 0|_{[0, \tau_2]}\right), p^{\tau_3}\right)$ of \eqref{ctrl_prblm} is well-defined in $[0,\tau_3]$. Which further simplifies that the solution $\mathcal R\left( u_0, p^{\tau_1} * 0|_{[0, \tau_2]}*p^{\tau_3}\right)$ of \eqref{ctrl_prblm} is well-defined in $[0,\tau_1+\tau_2+\tau_3].$
 Using the
	stability property \eqref{stability}, the flow property \eqref{flow_prpty}, \eqref{ineq_1},
	\eqref{ineq_3}, and \eqref{ineq_6}, we obtain a constant
	$C_2=C_2(\tau_3,\|p^{\tau_3}\|)>0$ such that
	\begin{align*}
&\left\|  \mathcal{R}_{\tau_1 + \tau_2 + \tau_3}\left( u_0, p^{\tau_1} * 0|_{[0, \tau_2]} * p^{\tau_3} \right) - e^{ - \left(\psi'\right)^4} u_0 \right\|_s \\
&\le \left\|\mathcal{R}_{\tau_3} \left(\mathcal{R}_{\tau_1 + \tau_2}\left(u_0, p^{\tau_1} * 0|_{[0, \tau_2]}\right) , p^{\tau_3}\right) - \mathcal{R}_{\tau_3}( \mathcal{R}_{\tau_2}\left( e^{-\tau_2^{-1/4} \widetilde{\psi}} u_0, 0\right) , p^{\tau_3})\right\|_s\\
& \hspace{2cm}+ \left\| \mathcal{R}_{\tau_3}\left( \mathcal{R}_{\tau_2}\left( e^{-\tau_2^{-1/4} \widetilde{\psi}} u_0, 0\right) , p^{\tau_3}\right) - e^{\tau_2^{- 1/4} \widetilde{\psi}} \mathcal{R}_{\tau_2}\left( e^{-\tau_2^{-1/4} \widetilde{\psi}} u_0, 0\right)\right\|_s \\
& \hspace{4cm}+ \left\| e^{\tau_2^{- 1/4} \widetilde{\psi}} \mathcal{R}_{\tau_2}\left( e^{-\tau_2^{-1/4} \widetilde{\psi}} u_0, 0\right) - e^{ - \left(\psi'\right)^4} u_0 \right\|_s\\
& \le C_2 \left\| \mathcal{R}_{\tau_1 + \tau_2}\left(u_0, p^{\tau_1} * 0|_{[0, \tau_2]}\right) - \mathcal{R}_{\tau_2}\left( e^{-\tau_2^{-1/4} \widetilde{\psi}} u_0, 0\right)\right\|_s + \varepsilon_1 + \frac{\varepsilon}{2}\\
&\le C_1 C_2 \varepsilon_1 + \varepsilon_1 + \frac{\varepsilon}{2}.
\end{align*}	
We can choose $\varepsilon_1 > 0,$ small enough such that $C_1 C_2 \varepsilon_1 + \varepsilon_1 < \varepsilon/2.$	Therefore, we have proved that for any $\varepsilon, T > 0,$ there exists a time $\tau := \tau_1 + \tau_2 + \tau_3 \in (0, T)$ and a piecewise constant control $\overline{p} :=  p^{\tau_1} * 0|_{[0, \tau_2]} * p^{\tau_3} : [0, \tau] \to \mathbb{R}^3,$ such that 
	\begin{equation}\label{eq:one-quartic}
		\left\|
		\mathcal R_\tau(u_0,\overline p)
		-e^{-(\psi')^4}u_0
		\right\|_s
		<\varepsilon.
	\end{equation}
	 Thus, for every $\psi\in \mathcal V_N$, the state
	$e^{-(\psi')^4}u_0$ can be reached approximately in arbitrarily
	small time.
	
	We now use \eqref{eq:one-quartic} successively for the functions
	$\phi_1,\ldots,\phi_d$ appearing in
	\eqref{eq:phi-VN-representation}. We claim that, for every
	$d\geq1$, the state
	\begin{equation*}
		e^{-\sum_{k=1}^d(\phi_k')^4}u_0
	\end{equation*}
	can be reached approximately in arbitrarily small time. We argue
	by induction on $d$. The case $d=1$ is exactly
	\eqref{eq:one-quartic}. Assume the assertion holds for $d-1$ and
	set
	\begin{equation*}
		\overline\phi:=
		-\sum_{k=1}^{d-1}(\phi_k')^4,
		\qquad
		\overline u_0:=e^{\overline\phi}u_0.
	\end{equation*}
	For any $\varepsilon_2>0$, by the induction hypothesis on $d$,
	there exist $\overline\tau_1\in(0,T/3)$ and a piecewise constant
	control $\overline p_1$ such that
	\begin{equation}\label{ineq_4}
		\left\|
		\mathcal R_{\overline\tau_1}(u_0,\overline p_1)
		-\overline u_0
		\right\|_s
		<\varepsilon_2.
	\end{equation}
	Applying \eqref{eq:one-quartic} with initial datum
	$\overline u_0$ and $\psi=\phi_d\in \mathcal V_N$, there exist
	$\overline\tau_2\in(0,T/3)$ and a piecewise constant control
	$\overline p_2$ such that
	\begin{equation}\label{ineq_5}
		\left\|
		\mathcal R_{\overline\tau_2}
		(\overline u_0,\overline p_2)
		-e^{-(\phi_d')^4}\overline u_0
		\right\|_s
		<\frac{\varepsilon}{2}.
	\end{equation}
	By stability \eqref{stability} and the flow property \eqref{flow_prpty}, combining \eqref{ineq_4}-\eqref{ineq_5}, there exists
	$C_3=C_3(\overline\tau_2,\|\overline p_2\|)>0$ such that
	\begin{align*}
		&\left\|
		\mathcal R_{\overline\tau_1+\overline\tau_2}
		(u_0,\overline p_1*\overline p_2)
		-e^{-(\phi_d')^4}\overline u_0
		\right\|_s\\
		&\qquad\leq
		C_3
		\left\|
		\mathcal R_{\overline\tau_1}(u_0,\overline p_1)
		-\overline u_0
		\right\|_s
		+\frac{\varepsilon}{2}
	<C_3\varepsilon_2+\frac{\varepsilon}{2}.
	\end{align*}
	Choosing $\varepsilon_2>0$ sufficiently small gives
	\begin{equation}\label{eq_est22}
		\left\|
		\mathcal R_{\overline\tau}
		(u_0,\widehat p)
		-e^{-\sum_{k=1}^d(\phi_k')^4}u_0
		\right\|_s
		<\varepsilon
	\end{equation}
	for some $\overline\tau\in(0,2T/3)$ and some piecewise constant
	control $\widehat p:=\overline p_1 *\overline p_2:[0,\overline \tau]\mapsto \mathbb{R}^3$.
	
	It remains to incorporate the term $\phi_0$. Set $\widetilde u_0:=
		e^{-\sum_{k=1}^d(\phi_k')^4}u_0.$
	Since $\phi_0\in \mathcal V_N$, the induction hypothesis $(P_N)$, applied
	with initial datum $\widetilde u_0$, yields
	$\widetilde\tau\in(0,T/3)$ and a piecewise constant control
	$\widetilde p:[0,\widetilde \tau]\mapsto \mathbb{R}^3$ such that
	\begin{equation}\label{eq_est21}
		\left\|
		\mathcal R_{\widetilde\tau}
		(\widetilde u_0,\widetilde p)
		-e^{\phi_0}\widetilde u_0
		\right\|_s
		<\frac{\varepsilon}{2}.
	\end{equation}
	Combining \eqref{eq_est22} and \eqref{eq_est21}, and using once
	more the stability and flow properties, we conclude, after taking
	the approximation error in \eqref{eq_est22} sufficiently small,
	that there exist $\tau:=\overline \tau+\widetilde \tau\in(0,T)$ and a piecewise constant control
	$p:=\widehat p* \widetilde p$ such that
	\begin{equation*}
		\left\|
		\mathcal R_\tau(u_0,p)
		-e^{\phi_0-\sum_{k=1}^d(\phi_k')^4}u_0
		\right\|_s
		<\varepsilon.
	\end{equation*}
	In view of \eqref{eq:phi-VN-representation}, the last expression
	is precisely
	\begin{equation*}
		\left\|\mathcal R_\tau(u_0,p)-e^\phi u_0\right\|_s
		<\varepsilon.
	\end{equation*}
	Therefore $(P_{N+1})$ holds. By induction, $(P_N)$ is valid for
	every $N\geq 0$.
	
	Finally, let $\phi\in H^s(\mathbb T)$ be arbitrary. Since the
	real trigonometric polynomials are dense in $H^s(\mathbb T)$,
	there exist $N\geq 0$ and $\phi_N\in \mathcal V_N$ such that
	$\|\phi_N-\phi\|_s$ is arbitrarily small. Since $s>1/2$,
	$H^s(\mathbb T)$ is a Banach algebra, and the map
	$v\mapsto e^v$ is continuous from $H^s(\mathbb T)$ into itself.
	Hence $\phi_N$ may be chosen so that
	\begin{equation*}
		\left\|e^{\phi_N}u_0-e^\phi u_0\right\|_s
		<\frac{\varepsilon}{2}.
	\end{equation*}
	Using $(P_N)$, there exist $\tau\in(0,T)$ and a piecewise
	constant control $p$ such that
	\begin{equation*}
		\left\|\mathcal R_\tau(u_0,p)-e^{\phi_N}u_0\right\|_s
		<\frac{\varepsilon}{2}.
	\end{equation*}
	Therefore, combining the last two inequalities, we obtain
	\begin{equation*}
		\left\|\mathcal R_\tau(u_0,p)-e^\phi u_0\right\|_s
		<\varepsilon,
	\end{equation*}
	which completes the proof.}
\end{proof}
We are now in a position to apply \Cref{prpn_exp_pertubed_initial} to conclude \Cref{main_thm_approx}.
\subsection{Proof of \Cref{main_thm_approx}}
\begin{proof}
{\Cref{point_11}.}
Assume that $u_0, u_1 \in H^{s}(\mathbb{T})$ and $\operatorname{sign}(u_0) = \operatorname{sign}(u_1)$. 
 We define $\mathcal Z$ as the closed set in which both $u_0$ and $u_1$ vanish:
\begin{align*}
    \mc Z := u_0^{-1}(\{0\}) = \ u_1^{-1}(\{0\}) .
\end{align*}
Consider for $\theta > 0$ the set
\begin{align*}
    \mc Z_{\theta} := \{ x \in \mathbb{T} : \operatorname{dist}(x, \mc Z) < \theta \},
\end{align*}
and its complement in $\mathbb{T}$, denoted by $\mc Z_{\theta}^c$.  
For $\theta > 0$, we define
\begin{align*}
    \phi_{\theta} = \chi_{\mc{Z}_{\theta}^c} \log\!\left(\frac{u_1}{u_0}\right),
\end{align*}
where $\chi_{\mc{Z}_{\theta}^c}$ is the indicator function of the set $\mc Z_{\theta}^c$.
The function $\phi_{\theta}$ is well defined because $u_1/u_0 > 0$ on $\mc Z_{\theta}^c$.
Furthermore, $\phi_{\theta} \in L^{\infty}(\mathbb{T})$. Notice that
\begin{align}\label{eq_rmk}
    \left\|e^{\phi_{\theta}} u_0 - u_1\right\|_{L^2(\mathbb{T})}
    \le 
    \left\|e^{\phi_{\theta}} u_0 - u_1\right\|_{L^2(\mc Z_{\theta}^c)}
    + \left\|u_0 - u_1\right\|_{L^2(\mc Z_{\theta}\setminus \mc Z)}.
\end{align}
Fix any $\varepsilon, T > 0$.  
We can choose $\theta > 0$ small enough so that
\begin{align*}
    \left\|e^{\phi_{\theta}} u_0 - u_1\right\|_{L^2(\mathbb{T})} < \frac{\varepsilon}{3}.
\end{align*}
Using density, there exists a $\widetilde\phi_{\theta}\in H^s(\T)$ such that
\begin{align}\label{est21}
    \left\|e^{\widetilde\phi_{\theta}} u_0 - u_1\right\|_{L^2(\mathbb{T})} \leq \left\|e^{\widetilde\phi_{\theta}} u_0 -e^{\phi_{\theta}} u_0\right\|_{L^2(\mathbb{T})}+\left\| e^{\phi_{\theta}} u_0-u_1\right\|_{L^2(\mathbb{T})}<  \frac{2\varepsilon}{3}.
\end{align}
We then apply \Cref{prpn_exp_pertubed_initial} with $\phi = \widetilde\phi_{\theta}$ and deduce that
there exist a time $\tau \in [0, T)$ and a control $p\in L^2(0,\tau;\mathbb{R}^3)$
such that the solution $\mc R( u_0, p)$ of \eqref{ctrl_prblm} is well defined in $[0,\tau]$
and satisfies
\begin{align}\label{est12}
    \left\|\mc R_{\tau}(u_0, p) - e^{\widetilde\phi_{\theta}} u_0\right\|_{L^2(\T)}
    < \frac{\varepsilon}{3}.
\end{align}
Applying the triangle inequality, from \eqref{est21} and \eqref{est12} we conclude that
\begin{align*}
    \left\|\mc R_{\tau}(u_0, p) - u_1\right\|_{L^2(\T)}
    \le 
    \left\|\mc R_{\tau}(u_0, p) - e^{\widetilde\phi_{\theta}} u_0\right\|_{L^2(\T)}
    + \left\|e^{\widetilde\phi_{\theta}} u_0 - u_1\right\|_{L^2(\T)}
    < \varepsilon.
\end{align*}

\medskip
\paragraph{\Cref{point_12}}
In this case, our aim is to show approximate controllability in $H^s(\T)$ norm. Fix $\varepsilon, \ T > 0.$ Here we replace $\phi_{\theta}$ with two choices, $$\phi_1 = \log \left( \frac{\operatorname{sign} (u_0)}{u_0} \right) \ \text{ and } \phi_2 = \log \left( \operatorname{sign} (u_1)u_1 \right),$$ which are well-defined everywhere in $\T.$ Since $u_0, \ u_1$ are in $H^s(\T),$ then both $\phi_1, \ \phi_2 \in H^s(\T).$ Without loss of generality, we assume that $u_0, u_1>0.$  Applying Proposition~\ref{prpn_exp_pertubed_initial} with $\phi = \phi_1$, for any $\varepsilon'>0,$ we obtain a time 
$\tau_1 \in (0, T/2]$ and a control $p^1 : [0, \tau_1) \to \R^3$ such that 
\begin{equation}\label{eq_1st}
\bigl\| \mathcal{R}_{\tau_1}\left(u_0, p^1\right) - 1 \bigr\|_{s} < \frac{\varepsilon'}{3}.
\end{equation}
Similarly, applying Proposition~\ref{prpn_exp_pertubed_initial} with $\phi = \phi_2$, we find a time 
$\tau_2 \in (0, T/2]$ and a control $p^2 : [0, \tau_2) \to \R^3$ satisfying 
\begin{equation}\label{eq_3rd}
\bigl\| \mathcal{R}_{\tau_2}\left(1, p^2\right) - u_1 \bigr\|_{s} < \frac{\varepsilon}{3}.
\end{equation}
Next, note that $1$ is a stationary solution of \eqref{ctrl_prblm} under the control 
$p^{0} : [0,\, T - \tau_1 - \tau_2] \to \mathbb{R}^3$ defined below.  
For the nonlinearities $\mathcal{N}_{KS}$ or $\mathcal{N}_{CH}$ in \eqref{ctrl_prblm}, 
one can consider 
\[
p^{0}(t) = (0,0,0), \qquad \forall\, t \in [0,\, T - \tau_1 - \tau_2].
\]
Let us define the control $p^1 * p^{0} * p^2$, which steers the solution of \eqref{ctrl_prblm} from $u_0$ to a state arbitrarily close to $u_1$ in the $H^s$–norm at time $T$.
Indeed, thanks to \eqref{con} and \eqref{flow_prpty} together with \eqref{eq_1st} and \eqref{eq_3rd}, we deduce
\begin{align*}
 \norm{\mathcal{R}_T\left(u_0, p^1 * p^{0} * p^2\right)-u_1}_{s} &\leq \norm{\mathcal{R}_{\tau_2}\left(\mathcal{R}_{T-\tau_2}(u_0, p^1 * p^{0}), p^2\right)-\mathcal{R}_{\tau_2}\left(1, p^{2}\right)}_{s}\\
&\qquad\qquad+\norm{\mathcal{R}_{\tau_2}\left(1,  p^{2}\right)-u_1}_{s}\\
&\leq C \norm{\mathcal{R}_{T-\tau_2}(u_0, p^1 * p^{0})-1}_{s}+\frac{\varepsilon}{3}\\
&\leq C\norm{\mathcal{R}_{T-\tau_1-\tau_2}\left(\mathcal{R}_{\tau_1}(u_0,p^1), p^{0}\right)-\mathcal{R}_{T-\tau_1-\tau_2}(1,p^0)}_{s}\\
&\qquad\qquad+C\norm{\mathcal{R}_{T-\tau_1-\tau_2}(1,p^0)-1}_{s}+\frac{\varepsilon}{3}\\
&\leq CC' \varepsilon'+\frac{\varepsilon}{3},
\end{align*}
where the existence of $C\left(\tau_2,\norm{p^2}\right), C'\left(T-\tau_1-\tau_2, \norm{p^0}\right)>0$ are given by \eqref{stability}. 
Choosing $\varepsilon'>0$ sufficiently small so that $CC' \varepsilon'+\frac{\varepsilon}{3}<\varepsilon$, we complete the proof.
\end{proof}
\begin{remark}
   The fact that the approximate controllability result in \Cref{main_thm_approx}--\Cref{point_11} is stated only in the $L^2$ setting, while \Cref{main_thm_approx}--\Cref{point_12} is formulated in the stronger $H^s$-topology, is a consequence of the approximation procedure used in the proof. Indeed, the quantity
$
\|e^{\phi_\theta}u_0-u_1\|_{H^s(\mc Z^c_\theta)},
$
which arises in the above argument (see inequality \ref{eq_rmk}), cannot be made arbitrarily small as $\theta \to 0$ whenever $s>0$ and $\mc Z \neq \emptyset$.  Moreover, in this case, extending the small-time approximate controllability result to arbitrary times is not addressed by the present approach. This is due to the fact that our analysis relies on sign conditions for the initial and terminal data. Consequently, the usual strategy of steering the system sufficiently close to a desired target in a short time and then maintaining the trajectory in a neighbourhood of that target for a sufficiently long time by means of a suitable control cannot be implemented here. As a result, the issue of $H^s$-approximate controllability in arbitrary time
$T>0$ under the assumptions of \Cref{point_11} remains an open and interesting
problem.
\end{remark}
\section{Proof of the conjugated dynamics limit and semi-global stability}\label{section_propn_proofs}
In this section, we prove Proposition~\ref{conjugate_limit} and establish the semi-global stability property \eqref{stability}. To this end, we collect several inequalities that will be used throughout this section.

\begin{lemma}(See \cite{Adams_Fournier_2003})\label{lemma_interpolation}
	The Sobolev space $H^{s}(\mathbb{T})$ satisfies the following properties:
	\begin{enumerate}[(i)]
		\item For any $s>\tfrac12$ and $u\in H^{s}(\mathbb{T})$, we have a constant $C>0$ such that
		\begin{equation}\label{L_infty_Hs}
			\|u\|_{L^\infty} \le C\,\|u\|_{s}.
		\end{equation}
		
		\item For $0\le s_{1}\le s_{2}$ and any $\varepsilon>0$, there exists $C(\varepsilon)>0$ such that for every $u\in H^{s_{2}}(\mathbb{T})$,
		\begin{equation}\label{interpolation_inequality}
			\|u\|_{s_{1}}
			\le \varepsilon\,\|\partial_x^{s_{2}}u\| + C(\varepsilon)\,\|u\|.
		\end{equation}
	\end{enumerate}
\end{lemma}
\begin{lemma}\label{lemma_Young}(Young's inequality)
Let $a, b \in [0, \infty),$ and $\varepsilon > 0, $ then we have
\begin{align}\label{Young_ineq}
ab \le \varepsilon^{-p} \frac{a^p}{p} + \varepsilon^{q} \frac{b^q}{q},
\end{align}
where $1 < p < \infty, \ \frac{1}{p} + \frac{1}{q} = 1.$
\end{lemma}

\subsection{Proof of conjugated dynamics limit }\label{sec_asymp}
In this section, we prove \Cref{conjugate_limit} for the Kuramoto--Sivashinsky equation \eqref{ctrl_prblm}, \eqref{ks}. 
The corresponding modifications for the Cahn--Hilliard equation are discussed in \Cref{rmk_ch}.
\begin{proof} [\bf Proof of \Cref{conjugate_limit}]	For ease of reading, we split the proof into several steps. 

	\smallskip
	\textbf{Step 1. Formulation.} 
To simplify the presentation, we assume throughout the proof that $\delta\in(0,1)$. By definition, $$ u := \mathcal{R}\left( e^{-\delta^{-\frac{1}{4}} \varphi} u_0, \delta^{-1} p\right)$$ is the solution of 
\begin{align*}
\begin{cases}
\partial_t u + \partial_{x}^4 u +  \partial_{x}^2 u + \mc{N}(u) = \delta^{-1} \langle p, \mu \rangle u,& t>0 \, , \, x \in  \mathbb{T},\\
u(0, x) = e^{-\delta^{-\frac{1}{4}} \varphi} u_0(x), & x \in \mathbb{T}.
\end{cases}
\end{align*}
Let us denote $$\Psi(t) := e^{\delta^{- \frac{1}{4}} \varphi} u(t, x).$$ Then according to Proposition \ref{existence}, $\Psi(t)$ is well-defined up to maximal time $T_*^{\delta} = T_*(e^{-\delta^{-\frac{1}{4}} \varphi} u_0, \delta^{-1} p) > 0.$
Next, we consider the operator:
$$( - \partial_{x}^4) : H^{s + 4}(\mathbb{T}) \to H^s (\mathbb{T}), \quad u \mapsto  - \partial_x^4 u .$$
It is easy to check that $- \partial_{x}^4$ is the infinitesimal generator of the strongly continuous semigroup $\{e^{t (-\partial_x^4)}\}_{t \ge 0}$. Moreover, it has the following expression
\begin{align}\label{expr}
  e^{t (-\partial_x^4)} u_0=\sum_{k\in \mathbb{Z}}u_{0,k} e^{-k^4 t} e^{ikx}, 
\end{align}
where $u_{0,k}$ are the Fourier coefficient for $u_0.$
We introduce the following functions
\begin{align}\label{defn_w_v}
w(t) = e^{\left(-(\varphi')^4 + \langle p, \mu \rangle\right) t} u_0^{\delta}, \hspace{2cm} v(t) = \Psi(\delta t) - w(t),
\end{align}
where $u_0^{\delta} := e^{ \delta^{1/8} (-\partial_x^4)} u_0 \in H^{s+4} (\mathbb{T}),$ such that
\begin{align}\label{u_0_delta_goesto_u_0}
\| u_0 - u_0^{\delta} \|_s \to 0, \quad \text{ as } \delta \to 0^+.
\end{align}
Thanks to \eqref{expr}, let us calculate the $H^s$ norm of $u_0^{\delta}$ as follows
\begin{align*}
   & \norm{u_0^\delta}^2_{s}=\sum_{k\in \mathbb{Z}}(1+|k|^2)^s |u_{0,k}|^2 e^{-2 k^4 \delta^{1/8}}\leq \norm{u_0}^2_s,\\
   & \norm{u_0^\delta}^2_{s+4}=\sum_{k\in \mathbb{Z}}(1+|k|^2)^{s+4}  |u_{0,k}|^2 e^{-2 k^4 \delta^{1/8}}\leq  \frac{C}{\delta^{1/4}}\norm{u_0}^2_s,
\end{align*}
which further simplifies that, there exists $C > 0$ independent of $\delta > 0,$ such that
\begin{align}\label{Hs_and_Hr_nrm_u_0_delta}
&\| u_0 ^{\delta} \|_s \le C, \quad \|u_0 ^{\delta} \|_{s+4} \le C \delta^{-1/8}.
\end{align}
Our aim is to show $\Psi(\delta) \xrightarrow{\delta \to 0^+} e^{(-(\varphi')^4 + \langle p, \mu \rangle) } u_0$ in $H^s(\T).$ Thanks to the definition \eqref{defn_w_v}, it is sufficient to prove that
\begin{align*}
\| v(1) \|_{s}\xrightarrow{\delta \to 0^+}  0.
\end{align*}
However, before proving that, we need to ensure the existence of $\delta_0 > 0$ small enough such that, for every $0 < \delta < \delta_0,$ $v(t)$ is well-defined in $[0, 1],$ that is  
\begin{align}\label{T_delta_ge1}
\delta^{-1} T^{\delta}_* \ge 1.
\end{align}
Let us take $t<\min\{2, \delta^{-1} T^{\delta}_*\}$. Observe that $v$ satisfies the following equation
\begin{equation}\label{equn_v_t}
\begin{cases}
\partial_t v  + \delta \partial_{x}^4 v & =  - \delta\partial_{x}^4w- \delta  \partial_{x}^2 (v + w) + 4 \delta^{3/4} \varphi' \partial_{x}^3 (v + w)  -  \delta e^{\delta^{-\frac{1}{4}} \varphi} \mathcal{N}\Big(e^{-\delta^{-\frac{1}{4}} \varphi} (v + w)\Big)   \\
&\qquad \qquad + F_1 \partial_{x}^2 (v + w)  + F_2 \partial_{x}(v + w) + F_3 (v + w) + \langle p, \mu \rangle v - (\varphi')^4 v ,
\end{cases}
\end{equation}
	with initial condition 
\begin{align}\label{initial_v}
		v(0) = u_0 - u_0^{\delta},
\end{align}
where 
\begin{align}
&F_1 := \Big( 6 \delta^{3/4} \varphi'' - 6 \delta^{1/2} (\varphi')^2 \Big)  \label{exp_F1},\\
&F_2 := \Big( 4\delta^{3/4} \varphi''' - 12 \delta^{1/2} \varphi' \varphi'' + 4 \delta^{1/4} (\varphi')^3 + 2 \delta^{3/4}  \varphi' \Big), \label{exp_F2} \\
\label{exp_f3} &F_3 := \Big( \delta^{3/4} \varphi'''' - 3 \delta^{1/2} (\varphi'')^2 -4\delta^{1/2}\varphi'\varphi''' + 6 \delta^{1/4} (\varphi')^2 (\varphi'') + \delta^{3/4}    \varphi'' - \delta^{1/2}    (\varphi')^2  \Big), 
\end{align}
and
\begin{align}
& \mathcal{N}_{KS} \Big(e^{-\delta^{-\frac{1}{4}} \varphi} (v + w)\Big) = e^{-2\delta^{-\frac{1}{4}} \varphi} \Big( (v + w) \partial_{x}(v + w) - \delta^{-1/4} \varphi' (v + w)^2 \Big). \label{ks_non_in_vw}
\end{align}
Thanks to \eqref{defn_w_v} and \eqref{Hs_and_Hr_nrm_u_0_delta}, there exists a constant $C >0$ such that, for all $t \in [0, 2],$
\begin{align}\label{Hs_and_Hr_nrm_w(t)}
	\| w(t) \|_s \le C, \quad \| w(t) \|_{s+4} \le C \delta^{-1/8}.
\end{align}
The regularity of $\varphi$, together with the assumption $\delta\in(0,1)$ and the above definitions \eqref{exp_F1}--\eqref{exp_f3}, yields that
\begin{align}\label{Fi}
\| F_1\|_{s} \le C \delta^{1/2}, \quad \| F_2\|_{s + 1} \le C \delta^{1/4}, \quad \| F_3\|_{s + 1} \le C \delta^{1/4}.
\end{align}

\smallskip
\textbf{Step 2. $L^2$-energy type estimate.}
Let us assume that $u_0 \in H^{2s + 2}(\mathbb{T})$ which implies $u(t) \in H^{2s + 2}(\mathbb{T})$ and therefore $v(t)\in H^{2s + 2}(\mathbb{T})$ for every $t\in (0, \delta^{-1} T^{\delta}_*).$ Taking the $L^2$-inner product of equation \eqref{equn_v_t} with $v$, and applying Young’s inequality together with \eqref{Hs_and_Hr_nrm_w(t)} and \eqref{Fi}, for sufficiently small $\varepsilon>0$, we obtain a constant $C>0$ independent of $\delta$ such that
\begin{align}\label{L^2_energy_v}
\frac{1}{2}\frac{d}{dt} \| v\|^2 + \delta \| \partial_x^2 v\|^2 & \le \delta \| w\|_2 \| \partial_{x}^2 v\| + \delta    \| v\| \| \partial_{x}^2 v\| + \delta    \|w\|_2 \|v\| + 4 \delta^{3/4} \| \varphi'\|_{L^{\infty}} \| w\|_3 \| v\| + \| F_1\|_{L^{\infty}} \| v\| \| \partial_x^{2} v\| \notag \\
& \qquad \qquad+ \| F_1\|_{L^{\infty}} \| w\|_2 \|v\| + C \| \partial_x F_2 \|_{L^{\infty}} \| v\|^2 + \|F_2\|_{L^{\infty}} \|w\|_1 \|v\| +  \|F_3\|_{L^{\infty}} \|v\|^2   \notag \\
& \notag\qquad \qquad \qquad +  \|F_3\|_{L^{\infty}} \| w\| \|v\| + \mc I (\varphi, v, w, \mu, p)\notag \\
& \notag \le C \delta^{7/8}  \| \partial_{x}^2 v\| + \delta    \| v\| \| \partial_{x}^2 v\| + C \delta^{7/8}  \| v\| + C \delta^{5/8} \| v\| + C \delta^{1/2} \| v\| \| \partial_x^{2} v\| + C \delta^{1/8} \| v\| 
 \\
 & \notag   \qquad \qquad \qquad+C \delta^{1/4} \| v\|^2+ \mathcal I (\varphi, v, w, \mu, p)\\
& \leq \varepsilon \delta \| \partial_{x}^2 v\| + C \delta^{1/8} + C (1 + \delta^{1/8}) \| v\|^2 + \mc I (\varphi, v, w, \mu, p),
\end{align}
where 
\begin{align}\label{exp_N_for_s0}
\mc I (\varphi, v, w, \mu, p) := & \ 4 \delta^{3/4} \left\langle \varphi' \partial_x^3v,\ v  \right\rangle_{L^2} - \left\langle e^{-\delta^{-\frac{1}{4}} \varphi} \Big( \delta (v + w) \partial_x(v + w) - \delta^{3/4} \varphi' (v + w)^2 \Big), v \right\rangle_{L^2} \notag \\
&  + \left\langle \Big( \langle p, \mu \rangle v - (\varphi')^4v \Big), \ v \right\rangle_{L^2} .
\end{align}
Considering the first term of $\mc I (\varphi, v, w, \mu, p)$ and using interpolation of Lemma~\ref{lemma_interpolation}, we have a constant $C>0$ such that
\begin{align}\label{1st_trm_N0}
| 4 \delta^{3/4} \left\langle \varphi' \partial_x^3v,\ v  \right\rangle_{L^2}  | &\le C \delta^{3/4} \| \varphi''\|_{L^{\infty}} \| v\| \| \partial_x^2 v\| + C \delta^{3/4} \| \varphi''\|_{L^{\infty}} \| \partial_x v\|^2 \notag \\
& \le C \delta^{3/4}  \| v\| \| \partial_x^2 v\| + C \delta^{3/4} \| \partial_x v\|^2 \notag \\
& \le \frac{\varepsilon}{2} {\delta} \| \partial_x^2 v \|^2 + C \delta^{1/8} \| v\|^2 + C \delta^{3/4} \left(\frac{\delta^{1/4}\varepsilon}{ 2C} \| \partial_x^2 v \|^2 + C_1 \delta^{-1/4} \|v \|^2\right) \notag \\
& \le \varepsilon \delta \| \partial_x^2 v \|^2 + C \delta^{1/8} \| v\|^2.
\end{align}
We focus on the remaining terms in $\mc I(\varphi, v, w, \mu, p)$.  If $\varphi>0$, we have $\| e^{-\delta^{-1/4}\varphi} \|_{L^{\infty}} \xrightarrow{\delta \to 0^{+}} 0,$
and hence $$\| \varphi' e^{-\delta^{-1/4}\varphi} \|_{L^{\infty}} \xrightarrow{\delta \to 0^{+}} 0.$$ Thus, it follows that $\|e^{-\delta^{-\frac{1}{4}} \varphi}\|_{L^\infty}+\| \varphi' e^{-\delta^{-1/4}\varphi} \|_{L^{\infty}}<1$ for some small value of $\delta.$ Using these limits with \eqref{Hs_and_Hr_nrm_w(t)}, we deduce a constant $C>0$ independent of $\delta$ such that
\begin{align}\label{2nd_3rd_trm_N0}
&\left| \left\langle e^{-\delta^{-\frac{1}{4}} \varphi} \Big( \delta (v + w) \partial_x(v + w) - \delta^{3/4} \varphi' (v + w)^2 \Big), v \right\rangle_{L^2}\right| + \left|  \left\langle \Big( \langle p, \mu \rangle v - (\varphi')^4 \Big)v, \ v \right\rangle_{L^2}\right| \notag \\
& \le \delta\left| \left\langle e^{-\delta^{-\frac{1}{4}} \varphi},\ v^2 \partial_xw + v w \partial_x w\right\rangle_{L^2}\right| + \frac{\delta}{3} \left| \left\langle e^{-\delta^{-\frac{1}{4}} \varphi}, \ \partial_x(v^3) \right\rangle_{L^2}\right| + \frac{\delta}{2}\left| \left\langle e^{-\delta^{-\frac{1}{4}} \varphi}, \ w \partial_x(v^2) \right\rangle_{L^2}\right| \notag \\
& \qquad+ \delta^{3/4} \left|\left\langle \varphi' e^{-\delta^{-\frac{1}{4}} \varphi}, \ v^3 + 2 v^2 w + v w \right\rangle_{L^2}\right| + C \| v\|^2 \notag \\
& \le \delta \left\| e^{-\delta^{-\frac{1}{4}} \varphi}\right\|_{L^{\infty}} \left(\left\| v\right\|^2 {\left\| w\right\|_2} + \left\|v\right\| \|w\|_1^2\right) + C \delta^{3/4} \left\| \varphi' e^{-\delta^{-\frac{1}{4}} \varphi}\right\|_{L^{\infty}} \left\| v\right\|^3_{L^3} + C \delta^{3/4} \left\| \varphi' e^{-\delta^{-\frac{1}{4}} \varphi}\right\|_{L^{\infty}} {\left\| w\right\|_1 }\left\| v\right\|^2 \notag \\
& \quad + C \delta \left\| e^{-\delta^{-\frac{1}{4}} \varphi}\right\|_{L^{\infty}}  \left\| w\right\|_2 \left\| v \right\|^2 + \delta^{3/4} \left\| \varphi' e^{-\delta^{-\frac{1}{4}} \varphi}\right\|_{L^{\infty}} \left(\left\| v\right\|^3_{L^3} + 2 \left\| v\right\|^2 {\| w\|_1} + \left\|v \right\| \left\|w\right\|\right) + C \left\| v\right\|^2 \notag \\
& \le C \delta^{1/8} + C (1 + \delta^{1/8}) \| v\|^2 + C \delta^{3/4} \| v\|^3_{L^3}.
\end{align} 
If $\varphi=0$, then the above estimate readily follows.
Putting together \eqref{L^2_energy_v}, \eqref{1st_trm_N0} and \eqref{2nd_3rd_trm_N0}, we have
\begin{align}\label{L^2_energy_v_fnl}
\frac{d}{dt} \| v\|^2 + \delta \| \partial_x^2 v\|^2 \le   C \delta^{1/8} + C (1 + \delta^{1/8}) \| v\|^2 + C \delta^{3/4} \| v\|^3_{L^3}.
\end{align}

\smallskip
\textbf{Step 3. $H^s$-energy type estimate.}
Let us take the $L^2$–inner product of equation \eqref{equn_v_t} with $\partial_x^{2s} v$, and using Young’s inequality together with the fact that $H^s(\mathbb{T})$ is an algebra for $s>\tfrac12$, we obtain
\begin{align}\label{del_s_L^2_enery_intermediate}
\frac{1}{2}\frac{d}{dt} \| \partial_x^s v\|^2 + \delta \| \partial_x^{s + 2} v\|^2 
& \le \delta \| w\|_{ s + 2} \| \partial_{x}^{ s + 2} v\| + \delta    \| \partial_x^s v\| \| \partial_{x}^{s + 2} v\| + \delta \|w\|_{s} \|\partial_x^{s+2} v\| + 4 \delta^{3/4} {\| \varphi\|_{{s+1}} \| w\|_{s + 3} \| \partial_x^s v\| }\notag  \\
& \quad + \langle F_1 \partial_{x}^2 (v + w), \partial_{x}^{2s} v \rangle_{L^2}  + \langle F_2 \partial_{x}(v + w) , \partial_{x}^{2s} v \rangle_{L^2} + \langle F_3 (v + w) , \partial_{x}^{2s} v \rangle_{L^2} \notag \\
&  \hspace{4cm} + \mc J (\varphi, v, w, \mu, p)
\end{align}
where 
\begin{align}\label{exp_N_for_s1}
\mc J (\varphi, v, w, \mu, p) := & \ 4 \delta^{3/4} \left\langle \varphi' \partial_{x}^3v, \ \partial_{x}^{2s}v \right\rangle_{L^2} - \left\langle e^{-\delta^{-\frac{1}{4}} \varphi} \Big( \delta (v + w) \partial_x(v + w) - \delta^{3/4} \varphi' (v + w)^2 \Big), \partial_{x}^{2s}v \right\rangle_{L^2} \notag \\
&  + \left\langle \Big( \langle p, \mu \rangle v - (\varphi')^4v \Big), \ \partial_{x}^{2s}v \right\rangle_{L^2}=: \mc J_1+\mc J_2+\mc J_3.
\end{align}
We now estimate the remaining term in \eqref{del_s_L^2_enery_intermediate} as follows.
\begin{align}\label{est1}
\left| \left\langle F_1 \partial_{x}^2 (v + w), \partial_{x}^{2s} v \right\rangle_{L^2} \right| & = \left| \left\langle \partial_{x}^{s}(F_1 \partial_{x}^2 (v + w)), \partial_{x}^{s} v \right\rangle_{L^2} \right| \notag \\
& \le C \| F_1\|_s {\left(\norm{v}+\| \partial_{x}^{s + 2} v \|\right)} \| \partial_{x}^{s} v\| + C \| F_1\|_s \| w\|_{s + 2} \| \partial_{x}^{s} v\| \notag\\
& \le C \delta^{1/2} \| \partial_{x}^{s } v \| \|  v\| +C \delta^{1/2} \| \partial_{x}^{s + 2} v \| \| \partial_{x}^{s} v\| + C {\delta^{3/8}}  \| \partial_{x}^{s} v\|.
\end{align}
Here we have used the fact that \begin{align*}\norm{\pa_x^s(F_1\pa_x^2 v)}_{L^2}\leq \norm{F_1\pa_x^2 v}_s\leq C \norm{F_1}_s\norm{\pa_x^2 v}_s\leq  C\delta^{1/2}\norm{v}_{s+2}\leq C \delta^{1/2}\left(\norm{v}+\norm{\pa_x^{s+2}v}\right).\end{align*}
Next, we estimate $ \left\langle F_2 \partial_{x}(v + w) , \partial_{x}^{2s} v \right\rangle_{L^2}.$ Using the algebra of $H^s(\mathbb{T})\ (s > 1/2)$ and the interpolation inequality in Lemma~\ref{lemma_interpolation} we have
\begin{align}\label{est2}
\left| \langle F_2 \partial_{x}(v + w) , \partial_{x}^{2s} v \rangle_{L^2} \right| & = \left|\langle \partial_x^s (F_2 \partial_{x}v) , \partial_{x}^{s} v \rangle_{L^2}\right| + \left|\langle \partial_x^s (F_2 \partial_{x}w) , \partial_{x}^{s} v \rangle_{L^2}\right| \notag \\
& \le C \| F_2\|_{s + 1} \left(\norm{v}+\| \partial_{x}^{s + 1} v\| \right)\| \partial_{x}^s v\| + C \| F_2\|_{s + 1} \|w \|_{s + 1} \| \partial_{x}^s v\| \notag \\
& \le  C \delta^{1/4}  \norm{v}\| \partial_{x}^{s} v\| +C \delta^{1/4} \Big( \delta^{1/4} \| \partial_{x}^{s + 2} v\| + C \delta^{- 1/4} \| \partial_{x}^s v \| \Big) \| \partial_{x}^s v\| + C \delta^{1/8}  \| \partial_{x}^{s} v\| \notag \\
& \le C \delta^{1/2} \| \partial_{x}^{s + 2} v \| \| \partial_{x}^{s} v\| + C \| \partial_{x}^s v \|^2 + C \delta^{1/8}  \| \partial_{x}^{s} v\|+ C \delta^{1/4}  \norm{v}\| \partial_{x}^{s} v\|.
\end{align}
Similarly, for any $s > 1/2,$
\begin{align}\label{est3}
\left| \left\langle F_3 (v + w) , \partial_{x}^{2s} v \right\rangle_{L^2} \right| \le C \delta^{1/4} \| \partial_x^s v \|^2 + C \delta^{1/8} \| \partial_x^s v\| {+C \delta^{1/4}  \norm{v}\| \partial_{x}^{s}v\|}.
\end{align}
Putting together \eqref{del_s_L^2_enery_intermediate} and \eqref{est1}--\eqref{est3} and again using Young's inequality, we deduce
\begin{align}\label{del_s_L^2_enery_fnl}
\frac{1}{2}\frac{d}{dt} \| \partial_x^s v\|^2 + \delta \| \partial_x^{s + 2} v\|^2 
& \le C \delta^{7/8} \| \partial_{x}^{ s + 2} v\| + \delta    \| \partial_x^s v\| \| \partial_{x}^{s + 2} v\| + C \delta^{7/8} \|\partial_x^s v\| + C \delta^{5/8} \| \partial_x^s v\| \notag  \\
&\quad + C \delta^{1/2} \| \partial_x^{s + 2} v\| \| \partial_x^s v\| + C \delta^{1/8} \| \partial_x^s v\| + C \| \partial_x^s v \|^2 + C \delta^{1/4} \| \partial_x^s v\|^2 \notag \\
& \hspace{4cm}+ \mc J (\varphi, v, w, \mu, p) \notag \\
& \le \varepsilon \delta \| \partial_{x}^{ s + 2} v\|^2 + C \delta^{1/8} + C (1 + \delta^{1/8}) \| \partial_x^s v\|^2 + \mc J (\varphi, v, w, \mu, p).
\end{align}
Finally adding \eqref{L^2_energy_v_fnl} and \eqref{del_s_L^2_enery_fnl}, we obtain
\begin{align}\label{intermediate_Hs_energy_equn_v}
\frac{d}{dt} \| v\|_s^2  + \delta \| \partial_x^{s + 2} v\|^2
& \le C \delta^{1/8} + C (1 + \delta^{1/8}) \|  v\|_s^2 + + C \delta^{3/4} \| v\|^3_{L^3} + \mc J (\varphi, v, w, \mu, p).
\end{align}
We will estimate the terms in $\mc J$ by successive application of Young's inequality.
Performing integration by parts, for sufficiently small $\varepsilon>0$, we have a constant $C>0$ independent of $\delta$ such that 
\begin{align}\label{delta11}
\notag \left|\mc J_1\right|=\left| 4 \delta^{3/4} \left\langle \varphi' \partial_{x}^3v, \ \partial_{x}^{2s}v \right\rangle_{L^2} \right| &= 4 \delta^{3/4}\bigg[ \left|\left\langle \varphi'' \partial_{x}^2v, \ \partial_{x}^{2s}v \right\rangle_{L^2} \right|+\left| \left\langle\varphi' \partial_{x}^2v, \ \partial_{x}^{2s+1}v \right\rangle_{L^2} \right|\bigg]\\
\notag&\leq C \delta^{3/4}\bigg[\norm{\pa_x^{s}\left(\varphi''\pa_x^2 v\right)}\norm{\pa_x^{s} v}+\norm{\pa_x^{s}\left(\varphi'\pa_x^2 v\right)}\norm{\pa_x^{s+1} v} \bigg]\\
\notag&\leq C \delta^{3/4}\bigg[\norm{v}_{s+2}\norm{\pa_x^{s} v}+\norm{v}_{s+2}\norm{\pa_x^{s+1} v} \bigg]\\
&\leq \varepsilon \delta \norm{\pa_x^{s+2} v}^2+C\delta^{1/2}\norm{v}^2+C \norm{\pa_x^s v}^2+C \norm{v}^2.
\end{align}
For the last two terms of $\mc J(\varphi, v, w, \mu, p)$, we will use the algebra property of
$H^s(\mathbb{T})$  for $s > 1/2$ and the interpolation inequality stated in
Lemma~\ref{lemma_interpolation}.  Thus, we obtain
\begin{align}\label{2nd_3rd_trm_N1}
\notag&\left|\mc J_2\right|+\left|\mc J_3\right|\\
&=\left|\langle e^{-\delta^{-\frac{1}{4}} \varphi} \Big( \delta (v + w) \partial_x(v + w) - \delta^{3/4} \varphi' (v + w)^2 \Big), \partial_{x}^{2s}v \rangle_{L^2}\right| +\left| \left\langle \Big( \left\langle p, \mu \right\rangle v - (\varphi')^4 v\Big), \ \partial_{x}^{2s}v \right\rangle_{L^2}  \right| \notag \\
& \le \delta \left| \left\langle\partial_x^s\left(e^{-\delta^{-\frac{1}{4}} \varphi} \partial_x(v+w)^2\right), \  \partial_x^s v\right\rangle_{L^2}\right|+ C \delta^{3/4} \left\| \partial_x^s (e^{-\delta^{- \frac{1}{4}}  \varphi} \varphi' (v + w)^2)\right\| \left\| \partial_x^s v\right\| + C \left\| v\right\|_s^2 .
\end{align}
We estimate the first term in the above inequality as follows.
\begin{align}\label{delta10}\notag\delta \left| \notag\left\langle\left(e^{-\delta^{-\frac{1}{4}} \varphi} \partial_x(v+w)^2\right), \  \partial_x^{2s} v\right\rangle_{L^2}\right|&=\delta^{\frac{3}{4}}\left| \left\langle\left(e^{-\delta^{-\frac{1}{4}} \varphi} \varphi'(v+w)^2\right), \  \partial_x^{2s} v\right\rangle_{L^2}\right|\\
\notag&\qquad+\delta\left| \left\langle\left(e^{-\delta^{-\frac{1}{4}} \varphi} (v+w)^2\right), \  \partial_x^{2s+1} v\right\rangle_{L^2}\right|\\
\notag&\leq C\delta^{\frac{3}{4}}\norm{e^{-\delta^{-\frac{1}{4}} \varphi}(v+w)^2}_{s}\norm{\pa_x^s v}+C\delta \norm{e^{-\delta^{-\frac{1}{4}} \varphi}(v+w)^2}_{s}\norm{\pa_x^{s+1} v}\\
&\leq \norm{e^{-\delta^{-\frac{1}{4}} \varphi}}_{s}\left(\varepsilon\delta \norm{\pa_x^{s+2} v}^2+ C\delta^{\frac{1}{2}} \left(\norm{v}_s^4+\norm{v}_s^2+ 1\right)\right).
\end{align}
The second term of \eqref{2nd_3rd_trm_N1} can be estimated similarly as
\begin{align}\label{delta34}
\notag C \delta^{3/4} \left\| \partial_x^s (e^{-\delta^{ -\frac{1}{4}}  \varphi} \varphi' (v + w)^2)\right\| \left\| \partial_x^s v\right\| & \le C \delta^{\frac{3}{4}} \norm{  e^{-\delta^{-\frac{1}{4}} \varphi}}_{s} \norm{\partial_x^s v}  \norm{v+w}^2_{s}\\
& \leq C \delta^{\frac{3}{4}} \norm{  e^{-\delta^{-\frac{1}{4}} \varphi}}_{s}\left(1+\norm{v}_s^2+\norm{v}^4_{s}\right). 
\end{align}
Let $\varphi>0.$ Observe that, if $s\in \mathbb{N^*}$, $\norm{e^{-\delta^{-\frac{1}{4}} \varphi}}_{s}\leq C \left(\norm{e^{-\delta^{-\frac{1}{4}} \varphi}}+\norm{ \pa_x^s (e^{-\delta^{-\frac{1}{4}} \varphi})}\right)\leq C \norm{e^{-\delta^{-\frac{1}{4}} \varphi}}_{L^{\infty}}\left(1+\delta^{-\frac{s}{4}}\norm{\pa_x^s \varphi}_{L^\infty} \right).$ If not, then $s$ will be replaced by $\lceil s \rceil$ (the smallest integer greater than or equal to $s$). Therefore as $\norm{e^{-\delta^{-\frac{1}{4}} \varphi}}_{L^{\infty}}(1+C\delta^{-\frac{s}{4}})\xrightarrow{\delta \to 0^{+}} 0 ,$ we have $\norm{e^{-\delta^{-\frac{1}{4}} \varphi}}_{s}<1$ for some small value of $\delta.$ 
Hence simplifying \eqref{2nd_3rd_trm_N1} together with \eqref{delta10}--\eqref{delta34}, we deduce
\begin{align}\label{2nd_3rd_trm_N2}
&\left|\mc J_2\right|+\left|\mc J_3\right|
\le \delta \varepsilon \norm{\pa_x^{s+2} v}^2 +C(1+\delta^{1/2})\norm{v}_s^2+ C\delta^{1/2} \norm{v}_s^4+C \delta^{1/2}.
\end{align}
Combining \eqref{intermediate_Hs_energy_equn_v}, \eqref{delta11} and \eqref{2nd_3rd_trm_N2} and using the fact for $s > 1/2, \ \|v\|_{L^3}^3 \le C \| v\|_{s}^3$ we obtain
\begin{align}\label{H^s_enrgy_v}
\frac{d}{dt} \| v\|_s^2  \le C \delta^{1/8} + C (1 + \delta^{1/8}) \|  v\|_s^2 +  C \delta^{1/8} \| v\|^4_s.
\end{align}
The above relation holds for any $ t < \min\{2, \delta^{-1} T^{\delta}_*\}.$ By the Gronwall Lemma and using \eqref{initial_v}, we have
\begin{align}\label{s_nrm_v(t)_after_grnwl}
\| v(t) \|^2_s \le e^{C(1 + \delta^{1/8})t} \Big( C \delta^{1/8}t  + \|u_0 - u_0^{\delta} \|_s^2 + C \delta^{1/8} \int_0^t \| v( \rho) \|_s^4 \ d\rho\Big),
\end{align}
for $ t < \min\{2, \delta^{-1} T^{\delta}_*\}$ and for $u_0 \in H^{2s + 2}(\mathbb{\T}).$ Finally, by the density of $ H^{2s + 2}(\T)$ in $H^s(\T)$ and using \eqref{stability}, we can have \eqref{s_nrm_v(t)_after_grnwl} for every $u_0 \in H^s(\T).$

\smallskip
\textbf{Step 4. Analysis of the maximal existence time.}
We are left to justify \eqref{T_delta_ge1}. Due to \eqref{u_0_delta_goesto_u_0}, we can choose $\delta_0 \in (0, 1)$ sufficiently small such that, for $0 < \delta < \delta_0,$ we have $\| u_0 - u_0^{\delta}\|_s^2 < 1/8$ and then $\| v(0) \|_s^2 < 1/8 .$ Denote $$\tau^{\delta} := \sup\{ t < \delta^{-1} T^{\delta}_* : \| v(t) \|_s < 1 \}.$$ The above inequality \eqref{s_nrm_v(t)_after_grnwl} ensures that $\tau^{\delta} > 0.$ If $\tau^{\delta} = + \infty,$ then \eqref{T_delta_ge1} is obvious. Thus consider the case when $\tau^{\delta}$ is finite. To prove \eqref{T_delta_ge1} we show that for sufficiently small $\delta_0 > 0$ and for all $ 0 < \delta < \delta_0$ we have $\tau^{\delta} \ge 1.$ We prove by contradiction. If not assume that for every $\delta_0 \in (0, 1),$ there exists a $\delta \in (0, \delta_0)$ such that $\tau^{\delta} < 1.$ Then from \eqref{s_nrm_v(t)_after_grnwl}, we have 
\begin{align}\label{1<grnwal_nrm}
1 = \| v(\tau^{\delta} ) \|_s^2 < e^{C(1 + \delta^{1/8})\tau^{\delta}} \Big( C \delta^{1/8}\tau^{\delta} + \|u_0 - u_0^{\delta} \|_s^2 + C \delta^{1/8} \int_0^{\tau^{\delta}} \| v( \rho) \|_s^4 \ d\rho\Big).
\end{align}
By the definition of $\tau^{\delta},$ for $t \in [0, \tau^{\delta}), \ \| v(t) \|_s < 1.$ For $\delta_0$ sufficiently small we have for $ 0 < \delta < \delta_0.$
\begin{align*}
e^{C(1 + \delta^{1/8})\tau^{\delta}} \Big( C \delta^{1/8}\tau^{\delta} + \|u_0 - u_0^{\delta} \|_s^2 \Big) < \frac{1}{2},
\end{align*}
 and hence
\begin{align*}
e^{C(1 + \delta^{1/8})\tau^{\delta}} \Big( C \delta^{1/8}\tau^{\delta} + \|u_0 - u_0^{\delta} \|_s^2 + C \delta^{1/8} \int_0^{\tau^{\delta}} \| v( \rho) \|_s^4 \ d\rho\Big) < 1,
\end{align*}
which contradicts \eqref{1<grnwal_nrm}. Hence, there exists a $\delta_0>0$ small enough, such that $\tau^{\delta} > 1$ for all $\delta \in (0, \delta_0).$ Thus we completes the proof of \eqref{T_delta_ge1}, and consequently \begin{align*}\| v(1) \|_{s}\xrightarrow{\delta \to 0^+}  0. \end{align*}
\end{proof}
\begin{remark}\label{rmk_ch} To prove \Cref{conjugate_limit} for the Cahn–Hilliard equation \eqref{ctrl_prblm}, \eqref{ch}, we indicate the modifications with respect to the proof given for the Kuramoto–Sivashinsky equation.
The only changes occur in the nonlinear terms of the equation for $v$ in \eqref{equn_v_t}.
To this end, we compute the corresponding nonlinear term:
\begin{align*}
\mathcal{N}_{CH}\big(e^{-\delta^{-1/4}\varphi}(v+w)\big)
&= -3 e^{-3\delta^{-1/4}\varphi} \Big( (v+w)^2 \partial_x^2 (v+w)
+ 2 (v+w)(\partial_x(v+w))^2 \notag \\
&\qquad
- 6 \delta^{-1/4} \varphi' (v+w)^2 \partial_x (v+w)   
+ \big(3 \delta^{-1/2} (\varphi')^2 - \delta^{-1/4} \varphi''\big) (v+w)^3 \Big).
\end{align*}
Performing similar estimates and simplifying as above, we obtain, for all $
t < \min\{2,\delta^{-1} T_*^\delta\}
$
\begin{align*}
\|v(t)\|_s^2
\le
e^{C(1+\delta^{1/8})t}
\Big(
C \delta^{1/8} t
+ \|u_0 - u_0^\delta\|_s^2
+ C \delta^{1/8} \int_0^t \|v(\rho)\|_s^6 \, d\rho
\Big).
\end{align*}
Applying arguments similar to those used above, we complete the proof.
\end{remark}
\subsection{Proof of semi-global stability} As mentioned in \Cref{sec_pre}, we prove \Cref{existence}--\Cref{point_1} here.
\begin{proof}
For any initial data $u_0, v_0 \in H^s(\T)$ and any control 
$p \in L^2_{\mathrm{loc}}(\mathbb{R}^+; \mathbb{R}^3)$, there exist positive times
$T_*^1 = T_*^1(u_0,p)$ and $T_*^2 = T_*^2(v_0,p)$ such that the corresponding
solutions of \eqref{ctrl_prblm}--\eqref{fourier_mods} satisfy
\begin{align*}
&u(t) := \mathcal R_t(u_0,p) \quad \text{is well-defined for all } t \in [0, T_*^1], \\
&v(t) := \mathcal R_t(v_0,p) \quad \text{is well-defined for all } t \in [0, T_*^2].
\end{align*}
Let us define 
\begin{align*}
\overline T <\min \{ T^1_*, T^2_*\}, \text{ and } \ \vartheta (t) := u(t) - v(t), \ \varrho(t) := u(t) + v(t), \, \text{ for all } t \in [0, \overline T],
\end{align*}
Then $\vartheta$ satisfies the following equation
\begin{align}\label{varheta_equn}
\begin{cases}
\partial_t \vartheta + \partial_{x}^4 \vartheta+  \partial_{x}^2 \vartheta + \Bigg(\mc{N}(u) - \mc{N}(v)\bigg)= \ip{p}{\mu} \vartheta, &  (t, x) \in (0, \overline T) \times  \mathbb{T},\\
\vartheta(0, x) = \vartheta_0(x) := u_0(x) - v_0(x), & x \in \mathbb{T},
\end{cases}
\end{align}
where, 
\begin{align}
 &\mc{N}_{KS}(u) - \mc{N}_{KS}(v)  = \frac{1}{2}\pa_x (\vartheta \varrho), \label{diff_KS}\\
 &\mc{N}_{CH}(u) - \mc{N}_{CH}(v) = - \pa_x^2 (\vartheta(u^2 + uv + v^2)).\label{diff_CH}
\end{align}
Multiplying first equation of \eqref{varheta_equn} by $\vartheta$ and integrating over $\T$, for sufficently small $\varepsilon>0$, we obtain a positive constant $C>0$ such that for $s > 1/2$
\begin{align}\label{est200}
\frac{d}{dt}\|\vartheta\|^2 + \|\pa_{x}^2 \vartheta\|^2
\le \varepsilon \|\pa_{x}^2 \vartheta\|^2 + C \|\vartheta\|^2 \bigl( 1 + \|u\|_s^4 + \| v \|_s^4 \bigr),
\end{align}
where the nonlinear terms are estimated as below.
\paragraph{\textit{Case 1. Kuramoto–Sivashinsky}}
\begin{align*}
\left| \int_{\mathbb{T}} \vartheta \,  \pa_x (\vartheta \varrho) \right| & \le C \left| \int_{\mathbb{T}}  \pa_x \vartheta^2 \, \varrho \right|  \le C \| \varrho\| \| \vartheta \|_1^2 \\
& \le \varepsilon \norm{\pa_x^2 \vartheta}^2+  C \left \|\varrho\|^2\right \|\vartheta\|^2.
\end{align*}
\paragraph{\textit{Case 2. Cahn–Hilliard}}
\begin{align*}
\left| \int_{\mathbb{T}} \vartheta \, \pa_x^2 \big(\vartheta(u^2 + uv + v^2)\big)\right| & = \left| \int_{\mathbb{T}} \pa_x^2 \vartheta \, \big(\vartheta(u^2 + uv + v^2)\big)\right|  \\
& \le \varepsilon \| \pa_{x}^2 \vartheta \|^2 + C \| \big(\vartheta(u^2 + uv + v^2)\big)\|^2 \\
& \le \varepsilon \| \pa_{x}^2 \vartheta \|^2 + C \| \vartheta\|^2 \|(u^2 + uv + v^2)\|_{L^{\infty}}^2 \\
& \le \varepsilon \| \pa_x^2 \vartheta \|^2 + C \| \vartheta\|^2 \| u\|_s^2 \|v\|_s^2.
\end{align*}
Next multiplying \eqref{varheta_equn} by $\pa_x^{2s} \vartheta$, we have the following
\begin{align}\label{est201}
\frac{d}{dt} \| \partial_x^s \vartheta\|^2 +\| \partial_x^{s + 2} \vartheta\|^2 \le \varepsilon \| \pa_x^{s + 2} \vartheta\|^2 + C \| \vartheta\|_s^2 \Big(1 + \| u\|^4_s + \|v\|_s^4 \Big). 
\end{align}
At this point, we have estimated the nonlinear terms in the following manner
\paragraph{\textit{Case 1. Kuramoto–Sivashinsky}}
\begin{align*}
\left| \int_{\mathbb{T}} \pa_x^{2s}\vartheta \, \pa_x \big(\varrho \vartheta\big)\right| & = \left| \int_{\T} \pa_x^{s + 1} \vartheta \, \pa_x^s (\varrho \vartheta) \right| \\
&\le \| \pa_x^{s + 1} \vartheta\| \| \varrho \vartheta\|_s \\
& \le \varepsilon \| \pa_x^{s + 2} \vartheta\|^2 + C (\| \varrho \|_s^2  + \| \varrho\|_s )\| \vartheta\|_s^2.
\end{align*}
\paragraph{\textit{Case 2. Cahn-Hilliard}}
\begin{align*}
\left| \int_{\mathbb{T}} \pa_x^{2s}\vartheta \, \pa_x^2 \big(\vartheta(u^2 + uv + v^2)\big)\right| =
& \left| \int_{\mathbb{T}} \pa_x^{s + 2}\vartheta \, \pa_x^s\big(\vartheta(u^2 + uv + v^2)\big)\right| \\
& \le  \varepsilon \| \pa_x^{s + 2}\vartheta\|^2 + C \| \pa_x^s\big(\vartheta(u^2 + uv + v^2)\big)\|^2 \\
& \le \varepsilon \| \pa_x^{s + 2}\vartheta\|^2 + C \| \vartheta\|^2_s \| u\|_s^2 \|v\|_s^2.
\end{align*}
Adding \eqref{est200} and \eqref{est201}, we have a positive constant $C$ such that
\begin{align*}
\frac{d}{dt}\|\vartheta\|_{s}^2 
\le C \|\vartheta\|_{s}^2
\bigl( 1 + \|u\|_{s}^4 + \| v\|_{s}^4 \bigr).
\end{align*}
Integrating the previous inequality over the interval $(0,\overline T)$ and using that
$u, v \in C([0,\overline T];H^s(\mathbb{T}))$, we deduce the existence of a constant $C>0$ such that
\[
\|\vartheta\|_{C([0,\overline T];H^s(\mathbb{T}))}
\le
e^{C\overline T}
\Bigg(
\|\vartheta_0\|_s
+ \sqrt{\overline T}\,
\|\vartheta\|_{C([0,\overline T];H^s(\mathbb{T}))}
\left(
\|u\|_{C([0,\overline T];H^s(\mathbb{T}))}^2
+
\|v\|_{C([0,\overline T];H^s(\mathbb{T}))}^2
\right)
\Bigg).
\]
Since $\|u_0\|\le R$ and $\|v_0\|\le R$ for some $R>0$, there exists a constant $C_1>0$,
depending only on $\|p\|_{L^2}$, and another constant $C_2>0$ such that
\[
\|\vartheta\|_{C([0,\overline T];H^s(\mathbb{T}))}
\le
e^{C_2\overline T}
\left(
\|\vartheta_0\|_s
+
2(C_1+R^2)\sqrt{\overline T}\,
\|\vartheta\|_{C([0 \overline T];H^s(\mathbb{T}))}
\right).
\]
If $\overline T$ satisfies
$
e^{C_2\overline T}(C_1+R^2)\sqrt{\overline T} < \tfrac14
$, we obtain the following required estimate for some constant $C(R, \norm{p}_{L^2})>0$ 
	\begin{align*}
		\| \vartheta\|_{C([0, \overline T];H^s(\T))} \leq C \norm{\vartheta_0}_{s}.
	\end{align*} 
If the above condition does not hold on the entire interval $[0,\overline T]$, we divide $[0,\overline T]$ into smaller subintervals on which the inequality $e^{C \overline T}(C_1+R^2)\sqrt{ \overline T}<\frac{1}{4}$ is satisfied. Applying the same argument on each subinterval yields the desired result. 

\end{proof}

\section{Small-time exact controllability to the constant states}\label{sec_con}
This section is devoted to the proof of small-time global exact controllability of the nonlinear system \eqref{ctrl_prblm} and \eqref{ctrl_form1}, that is the proof of \Cref{global_exact} and \Cref{global_exact1}. We make two separate sections for the Cahn-Hilliard and Kuramoto-Sivashinsky equations.
\subsection{Cahn-Hilliard equation}
Let us rewrite the Cahn-Hilliard system
\begin{align}\label{ctrl_prblm_ch}
\begin{cases}
\partial_t u + \partial_{x}^4 u+  \partial_{x}^2u =\pa_x^2(u^3) +\left(\mu_4p_4+\mu_5 p_5\right)u, &  t>0, \, x \in \mathbb{T},\\
u(0, x) = u_0(x), & x \in \mathbb{T}.
\end{cases}
\end{align}
In this section, we first prove the following local exact controllability result:
\begin{proposition}\label{exact}
	Let $T>0$ and $\Phi>0.$ Assume $\mu_4, \mu_5\in H^1(\T)$ satisfying \eqref{est_mu1}. Then there exists $R>0$ such that for any $u_0\in L^2(\T),$ satisfying $\norm{u_0-\Phi}_{L^2(\T)}<R$, there exist  controls $p_4, p_5 \in L^2((0,T);\R)$, such that the solution $u$ of \eqref{ctrl_prblm_ch} satisfies $u(T,\cdot)=\Phi$ in $\T.$
\end{proposition}
Introduce the change of variable $v=u-\Phi.$ Then $v$ is the solution of the following control problem
\begin{align}\label{Phi_shifted_ctrl_prblm}
\begin{cases}
\partial_t v+ \partial^4_{x} v+    \partial^2_{x}v - 3 \Phi^2 \partial^2_x v =  (p_4\mu_4+p_5\mu_5) (\Phi +v)+ F_{CH}, &  (t,x) \in (0, T) \times \mathbb{T},\\
v(0, x) = v_0(x):=u_0-\Phi, & x \in \mathbb{T},
\end{cases}
\end{align}
where $F_{CH} := 6 v (\partial_x v)^2 + 6 \Phi (\partial_x v)^2 + 3 v^2 \partial^2_x v + 6 v \partial^2_x v\Phi.$ Consequently, \Cref{exact} reduces to a corresponding local null controllability problem for \eqref{Phi_shifted_ctrl_prblm}.
\subsubsection{\bf Controllability of the linearized system}First let us consider the linearized control problem
\begin{align}\label{Phi_shifted_ctrl_prblm_lin}
\begin{cases}
\partial_t v + \partial^4_{x} v +    \partial^2_{x}v - 3 \Phi^2 \partial^2_x v =  (p_4\mu_4+p_5\mu_5) \Phi, &  (t,x) \in (0, T) \times \mathbb{T},\\
v(0, x) = v_0(x), & x \in \mathbb{T}.
\end{cases}
\end{align}
Equation \eqref{Phi_shifted_ctrl_prblm_lin} can equivalently be rewritten in the following abstract form
\begin{equation}\label{abstract_CH}
\begin{cases}
	\dfrac{d}{dt}v = \mc A v + \mc B (p_4,p_5), & t \in (0,T),\\
	v(0) = v_0,
\end{cases}
\end{equation}
where 
the operator $\mc A : \mc D(\mc A)\subset L^2(\T)\to L^2(\T)$ is thus given by
\begin{equation*}
\mc Av =  -\pa_x^4 v- (  1 - 3\Phi^2) \pa_x^2 v, \text{ with } \mc D(\mc A):=H^4(\T).
\end{equation*}
Clearly, $\mc A$ is densely defined, and its adjoint $ \mc A^*: \mc D(\mc A^*)\subset L^2(\T)\to L^2(\T)$ is
\begin{equation*}
\mc A^* v = -\pa_x^4 v- (  1 - 3\Phi^2)\pa_x^2 v, \text{ with } \mc D(\mc A^*):=H^4(\T).
\end{equation*}
We can prove that $\mc A$ generates an analytic semigroup $\{\mc S(t)\}_{t\ge 0}$ on $L^2(\mathbb{T})$. The control operator $\mc B$ $\in \mathcal L(\mathbb{R}^2, L^2(\T))$ satisfies $\mc B(p_4,p_5):=(p_4\mu_4+p_5\mu_5)\Phi.$ 
As $v_0\in L^2(\T),$ $p_4, \, p_5\in L^2(0,T),$ and $\mu_4,\, \mu_5 \in H^1({\T})$, equation \eqref{Phi_shifted_ctrl_prblm_lin} possesses a unique mild solution $v\in C([0,T];L^2(\T))\cap L^2((0,T);H^2(\T)).$
Moreover, certain examples for $\mu_4$ and $\mu_5$ satisfying \eqref{est_mu1} can be found in \cite[Example 4.2]{Duca_Pozzoli_Urbai_JMPA_2025}.

\noindent
The eigen-elements of the operator $\mc A^*$ are given by \begin{align*}
&\hspace{4cm}\text{ Eigenvalues}: \lambda_k=-k^4+ (  1 - 3 \Phi^2) k^2, \quad  \forall\, k\in\mathbb{N}. \\
&\text{ Eigenfunctions }: c_0(x)=\frac{1}{\sqrt{2\pi}}, \qquad 
c_k(x)=\frac{1}{\sqrt{\pi}}\cos(kx), \quad  
s_k(x)=\frac{1}{\sqrt{\pi}}\sin(kx), \qquad \forall\, k\in\mathbb{N}^*.
\end{align*}
These functions form a Hilbert basis of $L^2(\mathbb{T})$.

\noindent
We prove the following controllability result for the linearized system.
\begin{proposition}\label{exact_control_lin_CH}
Let $T>0$ be given and assume that $\mu_4, \mu_5 \in H^1(\mathbb{T})$ satisfy \eqref{est_mu1}.
Then for any  $v_0\in L^2(\T)$, there exist controls $p_4, \, p_5 \in L^2(0,T)$ such that the system \eqref{Phi_shifted_ctrl_prblm_lin} satisfies $v(T)=0$. Moreover, the controls satisfy
	\begin{equation}\label{cost_CH}
\norm{p_4}_{L^2(0,T)}+\norm{p_5}_{L^2(0,T)}\leq Ce^{\frac{C }{T}}\norm{v_0}_{L^2(\T)},
	\end{equation} 
	for some constant $C>0$ which is independent of $T$ and $v_0$.
\end{proposition}
\begin{proof}
   The proof proceeds by considering the following three steps of the method of moments.

    \noindent
\textbf{Step 1. Reduction to the moment problems.}
 At first, let us consider the following adjoint system 
\begin{equation}\label{adjoint_CH}
\begin{cases}
-\dfrac{d}{dt}\phi = \mc A^*\phi , & t \in (0,T),\\
	 \phi(T) = \phi_T.
\end{cases}
\end{equation}
Taking the inner product of \eqref{Phi_shifted_ctrl_prblm_lin} with $\phi$  in $L^2(\T)$, where $\phi$ is the solution of the adjoint equation \eqref{adjoint_CH}, and then integrating over $(0,T)$ we have
\begin{align}\label{equivalent_m_CH}
\ip{v(T,\cdot)}{\phi_T}_{L^2(\T)}-	\ip{v_0}{\phi(0,\cdot)}_{L^2(\T)}=\Phi\int_{0}^{T} {\Big\langle{p_4(t)\mu_4+p_5(t)\mu_5, \phi(t,\cdot)}\Big\rangle}_{L^2(\T)} dt.	
\end{align}	
To prove $v(T)=0,$ it is enough to establish that for all $\phi_T\in L^2(\T)$, the following identity holds:
\begin{align}\label{equivalent_CH}
\ip{v_0}{\phi(0,\cdot)}_{L^2(\T)}+\Phi\int_{0}^{T} p_4(t) \ip{\mu_4}{\phi(t,\cdot)}_{L^2(\T)} dt +\Phi\int_{0}^{T} p_5(t) \ip{\mu_5}{\phi(t,\cdot)}_{L^2(\T)} dt =0.
\end{align}
Our next task is to convert the above identity into a sequential problem by using the orthonormal eigenbasis $\{c_0,c_k,s_k\}_{k\in \mathbb{N}^*}$. To this end, we successively choose $\phi_T=c_0,c_k,$ and $ s_k$. As $\lambda_k$ is the eigenvalue of the operator ${\mc A}^*$, and in view of the assumption \eqref{est_mu1} together with the orthonormality of the eigenfunction $\{c_0, c_k, s_k\}_{k\in \mathbb{N}^*},$  the solution of the adjoint problem \eqref{adjoint_CH} reads as follows
\begin{equation}\label{eq:phi_CH}
\phi(t,x)= c_0, \quad e^{\lambda_k(T-t)}c_k(x), \quad e^{\lambda_k(T-t)}s_k(x).
\end{equation}
Plugging \eqref{eq:phi_CH} into \eqref{equivalent_CH}, we arrive at the following identities, which are equivalent to \eqref{equivalent_CH}.
\begin{align}\label{moment_CH}
\begin{cases}
	&-\frac{e^{\lambda_k T}\ip{v_0}{c_k}_{L^2(\T)}}{\Phi \ip{\mu_4}{c_k}_{L^2(\T)}}=\int_{0}^{T}p_4(t) e^{\lambda_k(T-t)}dt
=\int_{0}^{T}h_4(t) e^{\lambda_k t}dt \, \quad
 \forall k \, \in \mathbb{N},\\[10pt]
 & -\frac{e^{\lambda_k T}\ip{v_0}{s_k}_{L^2(\T)}}{\Phi \ip{\mu_5}{s_k}_{L^2(\T)}}=\int_{0}^{T}p_5(t) e^{\lambda_k(T-t)}dt
=\int_{0}^{T}h_5(t) e^{\lambda_k t}dt \, \quad
 \forall k \, \in \mathbb{N}^*,
 \end{cases}
\end{align}
where $h_i(t)=p_i(T-t),$ $i=4,5.$ Thus, it is enough to establish the existence and a suitable norm estimate for $h_i.$

\noindent
\textbf{Step 2. Existence of a biorthogonal family of exponentials.}
We first find the existence of $h_4$. Let us denote, for all $k\in \mathbb{N}^*,$ $\Lambda_k=-\lambda_{k-1}+ 1,$  and the collection $\Lambda =\{\Lambda_k, k\in \mathbb{N}^*\}.$
Our next goal is to check that the collection $\Lambda$ satisfies all the hypotheses of \cite[Theorem IV.1.10]{Boyer23}.
\begin{itemize}
\item[H1:] There exists $\theta > 0$ such that the family $ \Lambda\subset \mathbb{C}$
satisfies the following sector condition with parameter $\theta:$ 
\begin{equation*}
\Lambda \;\subset	S_{\theta} \;\stackrel{\text{def}}{=}\; 
\left\{\, z \in \mathbb{C} \;\middle|\; 
\Re z > 0,\; | \Im z | < (\sinh \theta)\, (\Re z) \right\}.
\end{equation*}
By the definition of $\Lambda,$ for all $k \in \mathbb{N}^*, \Lambda_k$ are positive real numbers, so the required condition is verified with some suitable $\theta>0.$

\item[H2:] Let $\kappa > 0 $. Define the counting function $
\mc	N_{\Lambda}(r) := \#\{\, \lambda \in \Lambda : |\lambda| \le r \,\}.
	$ The family 
$	\Lambda$ satisfies the asymptotic assumptions 
	\begin{align}\label{cn}
	\mc N_{\Lambda}(r) \le \kappa\, r^{1/4}, \qquad \forall \, r > 0.
	\end{align}
Using the definition of $\lambda_k$,
	$\Lambda_k=-\lambda_{k-1} +1= (k-1)^{4}-(  1 - 3 \Phi^2)(k-1)^{2} +1.$
    
\textit{Case1.} If $1-3\Phi^2\leq 0,$ then 
	$|\Lambda_k|\ge (k-1)^{4}-(  1 - 3 \Phi^2)(k-1)^{2} +1\geq (k-1)^4.$
Hence, if 
	$|\Lambda_k|\le r$, then
$	k\le 1+r^{1/4}.$  Therefore $
	\mc N_{\Lambda}(r) \le 1+r^{1/4}.$ 

\textit{Case 2.}
    If $1-3\Phi^2>0,$
	then there exists $C_1>0$ such that
	$|\Lambda_k|\ge C_1 (k-1)^4,$
	for $k\in \mathbb{N}^*.$ Thus similarly as Case 1, we have $\mc N_{\Lambda}(r) \le 1+C_1 r^{1/4}.$

     For small $r$, we may choose $\tilde r>0$ such that $\mc N_{\Lambda}(r)=0$ for all $r<\tilde r$. Thus, the required bound is verified. Next, by possibly increasing constant $C>0$, the same estimate \eqref{cn}
 holds for all $r\ge \tilde r$. Hence, the bound $\mc N_{\Lambda}(r)\le C r^{1/4}$ is valid uniformly for all $r>0$.
\item[H3:] Let $\rho > 0$ be given. The family $\Lambda$
	satisfies the gap condition with parameter $\rho$ if we have
	$$
	|\Lambda_m - \Lambda_n| \ge \rho, \qquad \forall\, m \ne n \in \mathbb{N}^*.
	$$
	This condition is obvious with $\rho= 3\Phi^2.$
\end{itemize}
Thus using \cite[Theorem IV.1.10]{Boyer23}, there exists a sequence
$\{e_k\}_{k\in\mathbb{N}^*} \subset L^2(0,T)$
such that, for all \(k,j \in \mathbb{N}^*\),
\begin{equation}\label{biorth_CH}
\int_{0}^{T} e_k(t)\, e^{-\Lambda_j t}\, dt = \delta_{k,j},
\end{equation}
and there exists a constant $C>0$ independent of $T$ such that
$\forall k \, \in\mathbb{N}^*$ 
\begin{equation}\label{nrm_ek}
\|e_k\|_{L^2(0,T)} \le C\, e^{C\left(\sqrt{\Lambda_k} + \frac{1}{T}\right)}.
\end{equation}

\noindent
\textbf{Step 3. Construction and estimates of the controls.}
Let us define the control function $h_4$ as follows:
\begin{align*}
	h_4(t):=-\sum_{k \in \mathbb{N}}\frac{e^{\lambda_k T}\ip{v_0}{c_k}_{L^2(\T)}}{\Phi \ip{\mu_4}{c_k}_{L^2(\T)}} \, e^{-t}\, e_{k+1}(t).
	\end{align*}
Clearly, this $h_4$ satisfies the second equation of \eqref{moment_CH}. We just need to show that $h_4\in L^2(0,T).$
Thus using \eqref{est_mu1} and \eqref{nrm_ek},
\begin{align*}
\norm{h_4}_{L^2(0,T)}&\leq \left(Ce^{\frac{C}{T}}+C \sum_{k \in \mathbb{N}^*} k^{2\theta_1} e^{C k^2 + \frac{C}{T}}e^{\Big(-k^4+ (1 - 3\Phi^2)k^2\Big) T}\right) \norm{v_0}_{L^2(\T)}.
\end{align*}
Observe that, there exists $k_0 \in \mathbb{N}$ such that $
-k^4+ (1 - 3\Phi^2)k^2 \le -C_1 k^4, \text{ for all } k > k_0,
$
for some constant $C_1 > 0$. Moreover, one can absorb $ k^{2\theta_1} $ in $e^{Ck^2}$ for all $k\in \mathbb{N}^*$ with a possibly large constant $C>0.$ Consequently, we estimate the control as follows:
\begin{align*}
\norm{h_4}_{L^2(0,T)}
&\leq \left(Ce^{\frac{C}{T}}+Ce^{CT}+C\sum_{k> k_0 } e^{C k^2 + \frac{C}{T}}e^{-C_1 {k^4} T} \right)\norm{v_0}_{L^2(\T)}.
\end{align*}
Using Young's inequality we have $Ck^2\leq \frac{C^2}{C_1 T}+\frac{C_1 k^4T}{4},$ and putting this in the above estimate there exists constant $C>0$ such that
\begin{align*}
	\norm{h_4}_{L^2(0,T)}&\leq C \left(e^{CT}+ e^{\frac{C}{T}}\right)\norm{v_0}_{L^2(\T)}.
\end{align*}
	Without loss of generality, we may assume that $T<1$. In this case, we obtain the desired control cost estimate
\begin{align*}
\|h_4\|_{L^2(0,T)} \le C \, e^{\frac{C}{T}}\, \|v_0\|_{L^2(\T)}.
\end{align*}
The case $T \ge 1$ can be reduced to the previous one. Indeed, any continuation by zero of a control defined on $(0,1/2)$ is also a control on $(0,T)$, and the estimate follows from the decrease of the control cost with respect to time.

A similar argument establishes the existence of $h_5$ together with the required cost estimate. This completes the proof of \Cref{exact_control_lin_CH}.
\end{proof}

\subsubsection{\bf Source term method and local controllability of the nonlinear problem}\label{sec_source_term} This section is devoted to the proof of local exact controllability (\Cref{exact}) of the nonlinear system \eqref{ctrl_prblm_ch}.
The strategy is to employ the source term method \cite{LTT} followed by the Banach fixed-point theorem to ensure local exact controllability. We first choose constants  $p>0$, $q>1$ in such a way that 
\begin{align}\label{choice-p_q}
	1<q<\sqrt{2}, \ \ \text{and} \ \ p> \frac{q^2}{2-q^2}.
\end{align}
We fix a constant $M>0$ and redefine the control cost as $Me^{M/T}$, as obtained in the control cost estimate \eqref{cost} for the corresponding linearized control problem \Cref{exact_control_lin_CH}. We then define the functions
\begin{equation}\label{def_weight_func}
	\rho_0(t)=\left\{\begin{array}{ll} e^{-\frac{pM}{(q-1)(T-t)}} & t \in [0, T), \\ 0 & t=T, \end{array}\right.   \quad 
	\rho_{\F}(t)= \left\{\begin{array}{ll} e^{-\frac{(1+p)q^2 M}{(q-1)(T-t)}} & t \in [0, T), \\ 0 & t=T. \end{array}\right.
\end{equation} 
Note that the functions $\rho_0$ and $\rho_{\F}$ are continuous and non-increasing in $[0,T]$. 
We next introduce the following weighted spaces
\begin{subequations}
	\begin{align}
		\label{space_F}
		&\F:= \left\{f\in L^1(0,T; L^{2}(\T)) \ \Big| \ \frac{f}{\rho_\F} \in L^1(0,T; L^{2}(\T))  \right\}, \\
		\label{space_Y} 
		&	\Y := \left\{ u\in C([0, T]; L^2(\T)) \Big| \frac{u}{\rho_0} \in C([0,T]; L^2(\T)) \cap L^2(0,T; H^2(\T))\right\}, \\
		\label{space_V}
		&	\V:= \left\{ p\in L^2(0,T)  \ \Big| \ \frac{p}{\rho_0}\in L^2(0,T)   \right\}.
	\end{align}
\end{subequations}
The norms associated with these weighted spaces are defined accordingly below.
\begin{align*}
&\norm{v}_{\mc S}:=\norm{\rho_0^{-1}v}_{C([0,T]; L^2(\T)) }+\norm{\rho_0^{-1}v}_{L^2(0,T; H^2(\T))}\\
	&\|f\|_{\F} := \| \rho^{-1}_\F f\|_{ L^1(0, T; L^2(\T)}, \quad 
	\|h\|_{\V} := \|\rho^{-1}_0 h\|_{L^2(0, T)}.
\end{align*}
Consider the linearized system with nonhomogeneous
source term:
\begin{equation}\label{abstract_source}
	\begin{cases}
		\dfrac{d}{dt}v = \mc Av + \mc B(p_4,p_5)+f, & t \in (0,T),\\
		v(0) = v_0.
	\end{cases}
\end{equation}
 Using arguments similar to those in Propositions~2.3 and Proposition~2.8 of \cite{LTT}, we obtain the following result.
\begin{proposition}\label{Proposition-weighted}
	Let $T>0$.  For any $f\in \F$ and $v_0 \in L^2(\T)$, there exist controls $p_4, p_5\in \V$ such that \eqref{abstract_source} admits a unique solution 
	$v\in \Y$ satisfying $v(T)=0.$
	Further, the solution and the control satisfy
	\begin{align}\label{estimate-weighted}
		\notag\left\|\frac{v}{\rho_0} \right\|_{C([0,T]; L^2(\T))}+\left\|\frac{v}{\rho_0} \right\|_{L^2(0,T; H^2(\T))} & + \left\|\frac{p_4}{\rho_0} \right\|_{L^2(0,T)}+\left\|\frac{p_5}{\rho_0} \right\|_{L^2(0,T)}\\
		&\leq  C \left(\|v_0\|_{ L^2(\T)} + \left\|{\frac{f}{\rho_\F}}\right\|_{L^1(0,T;L^2(\T)} \right),
	\end{align}
	where the constant $C>0$ does not depend on $v_0$, $f$, $p_4$, $p_5$ and $T$. 
\end{proposition}
We are now in a position to prove \Cref{exact} using a standard fixed-point argument.
\subsubsection{\bf Proof of \Cref{exact}}

\begin{proof}
For any $f \in \mc S,$ consider the map 
\begin{align*}
f \xmapsto{\mc F} 6 v (\partial_x v)^2 + 6 \Phi (\partial_x v)^2 + 3 v^2 \partial^2_x v + 6 v \partial^2_x v\Phi + (p_4\mu_4 + p_5\mu_5) v,
\end{align*}
where \((v, p_4, p_5) \in \mc Y \times \mc V \times \mc V,\) is the solution of \eqref{abstract_source} which satisfies \eqref{estimate-weighted}.
We show that
\begin{itemize}
\item the map $\mathcal{F}$ is well defined on $\mathcal{S}$;
\item there exists $R>0$ such that 
$\mathcal{F}\big(B(0,R)\big) \subset B(0,R),$ where $B(0,R)$ denotes the closed ball in $\mathcal{S}$ centered at the origin with radius $R$;
\item there exists $R>0$ such that the map $\mc F : B(0, R) \to B(0, R)$ is a strict contraction map.
\end{itemize}
To establish the existence of a fixed point for $\mathcal{F}$, we use the Banach fixed-point theorem.
First, observe that
\begin{align*}
\| \mc F(f) \|_{\mc S}
\le 
\left\| \frac{6 v (\partial_x v)^2 + 6 \Phi (\partial_x v)^2 + 3 v^2 \partial^2_x v + 6 v \partial^2_x v\Phi}{\rho_{\mc S}} \right\|_{L^1(0,T;L^{2}(\mathbb{T}))}
+ 
\left\| \frac{(p_4\mu_4 + p_5\mu_5) v}{\rho_{\mc S}} \right\|_{L^1(0,T;L^{2}(\mathbb{T}))}.
\end{align*}
To estimate the first term in , we use an interpolation argument. 
As, $v \in C([0,T];L^{2}(\T)) \cap L^{2}(0,T;H^{2}(\T))$. 
Then, for all $t\in(0,T)$, there exists a constant $C_{1}>0$, independent of $t$ and $v$, such that
\begin{align*}
\|v(t,\cdot)\|_{H^{1}}
\le C_{1}\,\|v(t,\cdot)\|_{H^{2}}^{1/2}
\|v(t,\cdot)\|_{L^{2}}^{1/2}.
\end{align*}
Consequently, we obtain
\begin{align}\label{est113}
\|v\|_{L^{4}(0,T;H^{1}(\T))}
\le C_{1}\,\|v\|_{C([0,T];L^{2}(\T))}^{1/2}
\|v\|_{L^{2}(0,T;H^{2}(\T))}^{1/2}.
\end{align}
Note that, the assumption 
$p > \frac{q^2}{2 - q^2}$ implies 
$2p > (1 + p) q^2$ which further
implies that
\begin{equation}\label{est111}
	\frac{\rho_0^2}{\rho_{\mathcal S}} , \ \frac{\rho_0^3}{\rho_{\mathcal S}}\in C([0,T]).
\end{equation}
Therefore, using \eqref{est111} and \eqref{est113}, one can estimate nonlinear terms in the following manner.
\begin{align}\label{est114}
\notag&\norm{\frac{v(\pa_xv)^2}{\rho_{\mc S}}}_{L^1(0,T;L^2(\T))}\leq C\int_{0}^{T} \frac{\rho_0^3(t)}{\rho_{\mc S}(t)}\norm{\frac{v}{\rho_0}}_{H^1}^2\norm{\frac{v}{\rho_0}}_{H^2}\leq C\norm{\frac{v}{\rho_0}}^2_{L^{4}(0,T;H^{1}(\T))}\norm{\frac{v}{\rho_0}}_{L^{2}(0,T;H^{2}(\T))}\\
\notag & \hspace{8cm}\leq C \left\|\frac{v}{\rho_0}\right\|_{C([0,T];L^{2}(\T))}
\left\|\frac{v}{\rho_0}\right\|_{L^{2}(0,T;H^{2}(\T))}^2.\\
\notag &\norm{\frac{(\pa_xv)^2}{\rho_{\mc S}}}_{L^1(0,T;L^2(\T))}+\norm{\frac{v\pa_x^2v}{\rho_{\mc S}}}_{L^1((0,T);L^2(\T))}\leq C \left\|\frac{v}{\rho_0}\right\|_{L^{2}(0,T;H^{2}(\T))}^2.\\
& \norm{\frac{v^2 \pa_x^2v}{\rho_{\mc S}}}_{L^1(0,T;L^2(\T))}\leq C\left\|\frac{v}{\rho_0}\right\|_{C([0,T];L^{2}(\T))}
\left\|\frac{v}{\rho_0}\right\|_{L^{2}(0,T;H^{2}(\T))}^2.
\end{align}
Combining \eqref{est114} and \eqref{estimate-weighted}, we deduce
\begin{align*}
\| \mc F (f) \|_{\mc S} 
&\le C \left( \left\| \frac{v}{\rho_0}\right\|_{C([0, T]; L^2(\T))} \left\| \frac{v}{\rho_0}\right\|_{L^2(0, T; H^2(\T))}^2 + \left\| \frac{v}{\rho_0}\right\|_{L^2(0, T; H^2(\T))}^2 \right)\\
&\hspace{2cm}+ C \left(
\left\| \mu_4 \frac{p_4}{\rho_0}  \frac{v}{\rho_0} \right\|_{L^1(0,T;L^2(\mathbb{T}))} + 
\left\| \mu_5 \frac{p_5}{\rho_0}  \frac{v}{\rho_0} \right\|_{L^1(0,T;L^2(\mathbb{T}))}\right)\\
& \le C_0 \Bigg[\left(\|v_0\|_{ L^2(\T)} + \left\|f\right\|_{\mc S} \right)^3+\left(\|v_0\|_{ L^2(\T)} + \left\|f\right\|_{\mc S} \right)^2\Bigg],
\end{align*}
for some positive constant $C_0,$ this, together with the uniqueness of $v$ in \Cref{Proposition-weighted}, proves the well-definedness of $\mathcal F$.

To ensure that \( B(0,R) \) is invariant under \( \mathcal F \) for some \( R>0 \),
 we choose
\[
0 < R < \min\left\{\frac{1}{4 C_0^{1/2}},\, \frac{1}{8 C_0}\right\}=:R_1.
\]
Then, by the above estimate, for any $v_0 \in L^2(\mathbb T)$ satisfying
$\|v_0\|_{L^2(\mathbb T)} \le R$, the closed ball $B(0,R)$ is invariant under
$\mathcal F$.

Consider any two $f, \overline f \in B(0, R).$ Then by the \Cref{Proposition-weighted} there exist controls $p_4, p_5, \overline p_4, \overline p_5 \in \mc V$ for the system \eqref{abstract_source} with solutions $v, \overline v \in \mc Y$ associated $f, \overline f.$ Then compute
\begin{align*}
& \| \mc F (f) - \mc F (\overline f) \|_{\mc S}\\
& \le 6 \left\| \frac{v (\pa_x v)^2 - \overline v (\pa_x\overline v)^2}{\rho_{\mc S}} \right\|_{L^1(0,T;L^{2}(\mathbb{T}))} + 6\Phi \left\| \frac{\Big((\pa_x v)^2 + v \pa_x^2v \Big) - \Big((\pa_x \overline v)^2 + \overline v \pa_x^2 \overline v \Big)}{\rho_{\mc S}} \right\|_{L^1(0,T;L^{2}(\mathbb{T}))}\\
&  + 3 \left\| \frac{v^2 \pa_x^2 v - \overline v^2 \pa_x^2 \overline v}{\rho_{\mc S}} \right\|_{L^1(0,T;L^{2}(\mathbb{T}))} 
+  \left\| \frac{\mu_4 (p_4 v- \overline p_4 \overline v)}{\rho_{\mc S}} \right\|_{L^1(0,T;L^{2}(\mathbb{T}))} + \left\| \frac{\mu_5 (p_5 v - \overline p_5 \overline v)}{\rho_{\mc S}} \right\|_{L^1(0,T;L^{2}(\mathbb{T}))}.
\end{align*}
Again, using \eqref{est111} and \eqref{est113}, one can estimate nonlinear terms one by one in the following way.
\begin{align}\label{est115}
&\left\| \frac{v (\pa_x v)^2 - \overline v (\pa_x\overline v)^2}{\rho_{\mc S}} \right\|_{L^1(0,T;L^{2}(\mathbb{T}))} \notag \\
& \le C \int_0^T \frac{\rho_0^3(t)}{\rho_{\mc S}(t)} \left\| \frac{v - \overline v}{\rho_0}\right\|_{L^2} \left\| \frac{v}{\rho_0}\right\|_{H^2} ^2 + C \int_0^T \frac{\rho_0^3(t)}{\rho_{\mc S}(t)} \left\| \frac{v - \overline v}{\rho_0}\right\|_{H^2} \left( \left\| \frac{v}{\rho_0}\right\|_{H^1}^2 + \left\| \frac{\overline v}{\rho_0}\right\|_{H^1}^2  \right)\notag \\
& \le C \left\| \frac{v - \overline v}{\rho_0}\right\|_{C([0, T]; L^2(\T))} \left\| \frac{v}{\rho_0}\right\|_{L^2(0, T; H^2(\T))}^2 +  C \left\| \frac{v - \overline v}{\rho_0}\right\|_{L^2(0, T; H^2(\T))} \Bigg[ \left\| \frac{v}{\rho_0}\right\|_{L^4(0, T; H^1(\T))}^2 \notag \\
& \hspace{8cm}+ \left\| \frac{ \overline v}{\rho_0}\right\|_{L^4(0, T; H^1(\T))}^2 \Bigg].
\end{align}
The second and third terms admit the following estimates.
\begin{align}\label{est116}
& \left\| \frac{\Big((\pa_x v)^2 + v \pa_x^2v \Big) - \Big((\pa_x \overline v)^2 + \overline v \pa_x^2 \overline v \Big)}{\rho_{\mc S}} \right\|_{L^1(0,T;L^{2}(\mathbb{T}))} \notag \\
& \le C \left\| \frac{v - \overline v}{\rho_0}\right\|_{L^2(0, T; H^2(\T))} \Bigg[ \left\| \frac{v}{\rho_0} \right\|_{C([0, T]; L^2(\T))}  + \left\| \frac{v}{\rho_0} \right\|_{L^2(0, T; H^2(\T))} + \left\| \frac{\overline v}{\rho_0} \right\|_{C([0, T]; L^2(\T))}  + \left\| \frac{\overline v}{\rho_0} \right\|_{L^2(0, T; H^2(\T))}\Bigg] \notag \\
& \hspace{4cm}+ C \left\| \frac{v - \overline v}{\rho_0}\right\|_{C([0, T]; L^2(\T))}\left\| \frac{\overline v}{\rho_0} \right\|_{L^2(0, T; H^2(\T))},
\end{align}
and 
 \begin{align}\label{est117}
&\left\| \frac{v^2 \pa_x^2 v - \overline v^2 \pa_x^2 \overline v}{\rho_{\mc S}} \right\|_{L^1(0,T;L^{2}(\mathbb{T}))} \le C \left\| \frac{v - \overline v}{\rho_0}\right\|_{L^2(0, T; H^2(\T))}  \left\| \frac{v}{\rho_0} \right\|_{C([0, T]; L^2(\T))}   \left\| \frac{v}{\rho_0} \right\|_{L^2(0, T; H^2(\T))} \notag \\
&\hspace{6cm}+ C \Bigg[ \left\| \frac{v - \overline v}{\rho_0}\right\|_{C([0, T]; L^2(\T))} +  \left\| \frac{v - \overline v}{\rho_0}\right\|_{L^2(0, T; H^2(\T))}\Bigg]\times \notag\\
&\Bigg[ \left( \left\| \frac{v}{\rho_0} \right\|_{C([0, T]; L^2(\T))}  + \left\| \frac{v}{\rho_0} \right\|_{L^2(0, T; H^2(\T))} + \left\| \frac{ \overline v}{\rho_0} \right\|_{C([0, T]; L^2(\T))}  + \left\| \frac{\overline v}{\rho_0} \right\|_{L^2(0, T; H^2(\T))} \right) \left\| \frac{\overline v}{\rho_0} \right\|_{L^2(0, T; H^2(\T))} \Bigg]
\end{align}

Finally, from the above estimates \eqref{est115}-\eqref{est117} we have
\begin{align*}
\| \mc F (f) - \mc F (\overline f) \|_{\mc S} \le C \| v - \overline v \|_{\mc Y} \Bigg[\left(\|v_0\|_{ L^2(\T)} + \left\|f\right\|_{\mc S} \right)^2+\left(\|v_0\|_{ L^2(\T)} + \left\|f\right\|_{\mc S} \right)\Bigg] \\
+ C \left(\|v_0\|_{ L^2(\T)} + \left\|f\right\|_{\mc S}\right)
\Big[ \|p_4 - \overline p_4 \|_{\mc V} + \|p_5 - \overline p_5 \|_{\mc V  } \Big].
\end{align*}
We have shown that for any $v_0 \in L^2(\mathbb T)$ satisfying
$\|v_0\|_{L^2(\mathbb T)} \le R$, the closed ball $B(0,R)$ is invariant under $\mc F.$ Using this fact 
\begin{align*}
\| \mc F (f) - \mc F (\overline f) \|_{\mc S} \le 4 R^2C \| v - \overline v \|_{\mc Y} + 2RC 
\Big[  \| v - \overline v \|_{\mc Y} +\|p_4 - \overline p_4 \|_{\mc V} + \|p_5 - \overline p_5 \|_{\mc V  } \Big].
\end{align*}
By the linearity of the solution associated with \eqref{abstract_source}, it follows from \cref{Proposition-weighted} that
\begin{align*}
\| \mc F (f) - \mc F (\overline f) \|_{\mc S} & \le (4 R^2 C_1 + 2RC_1) \| f - \overline f\|_{\mc S}\\
& \le \frac{1}{2} \| f - \overline f\|_{\mc S},
\end{align*}
by further choosing 
\[
0 < R < \min\left\{ R_1, \frac{1}{4 C_1^{1/2}},\, \frac{1}{8 C_1}\right\}.
\]
 For the above choice of \( R \), let \( v_0 \in L^2(\mathbb T) \) satisfy
\[
\|v_0\|_{L^2(\mathbb T)} \le R.
\]
By the Banach fixed point theorem, the map
\[
\mathcal F : B(0,R) \to B(0,R)
\]
admits a unique fixed point, denoted by \( \widehat f \in B(0,R) \).

Thanks to \Cref{Proposition-weighted}, there exist controls and the corresponding solution
\[
(v, p_4, p_5) \in \mathcal Y \times \mathcal V \times \mathcal V
\]
to the system \eqref{abstract_source} associated with the source term \( \widehat f \in B(0,R) \), which satisfy \eqref{estimate-weighted}. By the definition of the space \( \mathcal Y \) and the property \( \rho_0(T)=0 \), we conclude that the equation \eqref{abstract_source} is locally null controllable. This completes the proof of \cref{exact}.
\end{proof}
\subsection{Proof of \Cref{global_exact}}
The proof of \Cref{global_exact} follows from \Cref{main_thm_approx} and \Cref{exact}. 
Indeed, we divide the time interval $[0,T]$ into two subintervals $[0,T/2]$ and $[T/2,T]$. 
We choose the radius $R$ appearing in \Cref{exact} depending on $\Phi$ and $T/2$. 
By \Cref{main_thm_approx}, for this prescribed $R>0$, there exist controls 
$(p_1,p_2,p_3) \in L^{2}((0,T/2);\mathbb{R}^3)$ such that
\[
\|u(T/2)-\Phi\|_{H^s} < R.
\]
We then apply \Cref{exact} on the interval $[T/2,T]$ to conclude global exact controllability, 
using two additional controls $(p_4,p_5) \in L^{2}((T/2,T);\mathbb{R}^2)$. \qed

\subsection{Kuramoto-Sivashinsky equation}
In this section, we prove the global exact controllability of the Kuramoto--Sivashinsky equation
\begin{align*}
\begin{cases}
\partial_t u + \partial_x^4 u + \partial_x^2 u + u\,\partial_x u
= \mu_4 p_4 u, & t>0,\ x \in \mathbb{T},\\
u(0,x) = u_0(x), & x \in \mathbb{T}.
\end{cases}
\end{align*}
We only indicate the changes in the proof with respect to the previous case of the Cahn--Hilliard equation.
To this end, we set $v = u - \Phi$. Then $v$ solves the following control problem.
\begin{align}\label{ctrl_prblm5}
	\begin{cases}
		\partial_t v + \partial^4_{x} v + \partial^2_{x}v = -v\partial_x v+ p_4\mu_4 \Phi+ p_4\mu_4 v, &  (t,x) \in (0, T) \times \mathbb{T},\\
		v(0, x) = v_0(x):=u_0-\Phi, & x \in \mathbb{T}.
	\end{cases}
\end{align}
\subsubsection{\bf Controllability of the linearized system}First, let us consider the linearized control problem
\begin{align}\label{ctrl_prblm5_lin}
	\begin{cases}
		\partial_t v + \partial^4_{x} v +  \partial^2_{x}v+\Phi \partial_x v = p_4\mu_4 \Phi, &  (t,x) \in (0, T) \times \mathbb{T},\\
		v(0, x) = v_0(x), & x \in \mathbb{T}.
	\end{cases}
\end{align}
We recall that \eqref{ctrl_prblm5_lin} can equivalently be rewritten in the abstract form
\begin{equation}\label{abstract}
\begin{cases}
	\dfrac{d}{dt}v = \mc Av + \mc B p_4, & t \in (0,T),\\
	v(0) = v_0.
\end{cases}
\end{equation}
where 
the operator $\mc A : \mc D(\mc A)\subset L^2(\T)\to L^2(\T)$ is thus given by
\begin{equation*}
\mc Av =  -\pa_x^4 v- \pa_x^2 v-\Phi \pa_x v, \text{ with } \mc D(\mc A):=H^4(\T).
\end{equation*}
Clearly, $\mc A$ is densely defined, and its adjoint $ \mc A^*:\mc D(\mc A^*)\subset L^2(\T)\to L^2(\T)$ is
\begin{equation*}
\mc A^* v = -\pa_x^4 v- \pa_x^2 v+\Phi \pa_x v, \text{ with } \mc D(\mc A^*):=H^4(\T). \end{equation*}
 We can prove that $\mc A$ generates a strongly continuous semigroup $\{\mc S(t)\}_{t\ge 0}$. On the other hand, the operator $\mc{B} \in \mathcal L(\mathbb{R}, L^2(\T))$ satisfies $\mc Bp_4:=p_4\mu_4\Phi.$ 
The eigen-element of the operator $\mc A^*$ is $$\text{ Eigenvalues}: \lambda_k=-k^4+ k^2+ik \Phi, \,\text{ Eigenfunctions }: \phi_k=\frac{1}{\sqrt{2\pi}}e^{ikx}, k\in \mathbb{Z}.$$ 
As an illustrative example, let us define
$
\mu_4(x) = x^2 (x - 2\pi)^2,\, x \in [0, 2\pi],
$
and extend it by $2\pi$-periodicity so that $\mu_4 \in H^1(\mathbb{T})$. A direct computation shows that
\begin{equation*}
	\ip{\mu_4}{\phi_k}_{L^2(\mathbb{T})}
	= -\frac{24\sqrt{2\pi}}{k^4}, \qquad k \in \mathbb{Z}\setminus\{0\}, \text{ and } \ip{\mu_4}{\phi_0}_{L^2(\mathbb{T})} > 0.
\end{equation*}
In particular, we deduce the existence of $C > 0$ and $\theta=2$, such that
\begin{equation}\label{est_q0}
\forall k \in \mathbb{Z}, \qquad
(k^4 + 1)\,
\bigl| \ip{\mu_4}{\phi_k}_{L^2(\mathbb{T})} \bigr|
\geq C.\end{equation}
 
\begin{theorem}\label{exact_control_lin}
	Let $T>0$ be given and assume that $\mu_4 \in H^1(\T)$ satisfies \eqref{est_q1}. Then for every  $v_0\in L^2(\T)$, there exists a control $p_4\in L^2(0,T)$ such that equation \eqref{ctrl_prblm5_lin} satisfies $v(T)=0$. Moreover, the control satisfies
	\begin{equation}\label{cost}
		\norm{p_4}_{L^2(0,T)}\leq Ce^{\frac{C }{T}}\norm{v_0}_{L^2(\T)},
	\end{equation} 
	for some constant $C>0$ which is independent of $T$ and $v_0$.
\end{theorem}
\begin{proof}Using arguments similar to those in the proof of \Cref{exact_control_lin_CH}, we obtain the following identity, which is equivalent to the null controllability problem for the concerned system.
\begin{align}\label{moment}
\nonumber	-\frac{e^{\lambda_k T}\ip{v_0}{\phi_k}_{L^2(\T)}}{\Phi \ip{\mu_4}{\phi_k}_{L^2(\T)}}=&\int_{0}^{T}p_4(t) e^{\lambda_k(T-t)}dt\\
=&\int_{0}^{T}h(t) e^{\lambda_k t}dt \, \quad
 \forall \, k \, \in \mathbb{Z},
\end{align}
where $h(t)=p_4(T-t).$ Thus, it is enough to find the existence and suitable norm estimate for $h.$
Using the bijection  $\sigma:\mathbb{N}^*\mapsto \mathbb{Z}$, defined by
\begin{equation*}
\sigma(m)=
\begin{cases}
	{\frac{m}{2}}, & \text{if } m \text{ is even},\\
	{\frac{1-m}{2}}, & \text{if } m \text{ is odd},
\end{cases}
\end{equation*}
we denote, for all $k\in \mathbb{N}^*,$ $\Lambda_k=-\lambda_{\sigma(k)}+1.$ Let us also set $\Lambda =\{\Lambda_k, k\in \mathbb{N}^*\}.$
Our next goal is to check that the sequence $\Lambda_k$ satisfies all the hypotheses of \cite[Theorem IV.1.10]{Boyer23}.
\begin{itemize}
	\item[H1:] There exists $\theta > 0$ such that the family $ \Lambda\subset \mathbb{C}$
	satisfies the following sector condition with parameter $\theta$ 
	\begin{equation*}
		\Lambda \;\subset	S_{\theta} \;\stackrel{\text{def}}{=}\; 
		\left\{\, z \in \mathbb{C} \;\middle|\; 
		\Re z > 0,\; | \Im z | < (\sinh \theta)\, (\Re z) \right\}.
	\end{equation*}
	 By the definition of $\Lambda,$ it is clear that $\Re \Lambda_k>0,$ for all $k\in \mathbb{N}^*.$ Furthermore, as $|\Im \Lambda_k|< C\Phi (\Re \Lambda_k)$, for some $C>0$, the required condition is verified with some suitable $\theta>0.$
	\item[H2:] Let $\kappa > 0 $. Define the counting function $
\mc	N_{\Lambda}(r) := \#\{\, \lambda \in \Lambda : |\lambda| \le r \,\}.
	$ The family 
$	\Lambda \subset \mathbb{C} $ satisfies the asymptotic assumptions 
	$$
	\mc N_{\Lambda}(r) \le \kappa\, r^{1/4}, \qquad \forall \quad r > 0.
	$$
	Set $s=\sigma(k)\in\mathbb Z$. Using the definition of $\lambda_s$,
	$\Lambda_k=-\lambda_s+1=s^{4}-s^{2}+1 - i s\Phi.$
	Thus $\Re\Lambda_k=s^{4}-s^{2}+1$, and therefore
	$|\Lambda_k|\ge |\Re\Lambda_k|=s^{4}-s^{2}+1.$
	For $s\in \mathbb{Z},$ we have $s^{4}-s^{2}+1\ge \tfrac12 s^{4}$. Hence, if 
	$|\Lambda_k|\le r,$ then
$	\frac12 s^{4}\le r$ which yields
	$|s|\le (2r)^{1/4}.$ Therefore the number of integer $s$ satisfies the above inequality is $1+2(2r)^{1/4}$. Therefore, we have proved that $
	\mc N_{\Lambda}(r) \le 1+2(2r)^{1/4}.$ Required bound for the counting function and the existence of $\kappa$ is now straightforward. 
	
	\item[H3:] Let $\rho > 0$ be given. The family $\Lambda$
	satisfies the gap condition with parameter $\rho$ if we have
	$$
	|\Lambda_m - \Lambda_n| \ge \rho, \qquad \forall\, m \ne n \in \mathbb{N}^*.
	$$
	This condition is obvious with $\rho=\Phi$.
	\end{itemize}
Thus using \cite[Theorem IV.1.10]{Boyer23}, there exists a sequence
$\{e_k\}_{k\in\mathbb{N}^*} \subset L^2(0,T)$
such that, for all \(k,\, j \in \mathbb{N}^*\),
\begin{equation}\label{biorth}
\int_{0}^{T} e_k(t)\, e^{-\Lambda_j t}\, dt = \delta_{k,j},
\end{equation}
and there exists a constant $C>0$ independent of $T$ such that
for all $ k\,\in\mathbb{N}^*$ 
\begin{equation*}\|e_k\|_{L^2(0,T)} \le C\, e^{C\left(\sqrt{\Re(\Lambda_k)} + 1/T\right)}.
\end{equation*}
We set, for $k\in\mathbb{Z}$ and $t\in[0,T]$,
\begin{equation*}
\psi_k(t)=e^{-t}\, e_{\sigma^{-1}(k)}(t).
\end{equation*}
For any $k, j \in \mathbb{Z}$, using \eqref{biorth}
\begin{align*}
\int_{0}^{T} \psi_k(t)\, e^{\lambda_j t} \, dt
= \int_{0}^{T} e_{\sigma^{-1}(k)}(t)e^{-\Lambda_{\sigma^{-1}(j)} t} dt=\delta_{k,j}.
\end{align*}
Moreover, we have
\begin{align}\label{norm_psi}
	\norm{\psi_k}_{L^2(0,T)}\leq  C\, e^{C k^2 + \frac{C}{T}}, \quad \, \forall \, k \, \in \mathbb{Z}.
\end{align}
Let us now define the control function $h$ as follows:
\begin{align*}
	h(t):=-\sum_{k \in \mathbb{Z}}\frac{e^{\lambda_k T}\ip{v_0}{\phi_k}_{L^2(\T)}}{\Phi \ip{\mu_4}{\phi_k}_{L^2(\T)}} \psi_k(t).
	\end{align*}
Clearly this $h$ satisfies \eqref{moment}. We just need to show that $h\in L^2(0,T).$
Thus using \eqref{est_q1} and \eqref{norm_psi},
\begin{align*}
	\norm{h}_{L^2(0,T)}&\leq C \sum_{k \in Z} (k^{2\theta}+1) e^{C k^2 + \frac{C}{T}}e^{(-k^4+k^2) T} \norm{v_0}_{L^2(\T)}\\
	&\leq C\left( e^{CT}+\sum_{|k|\geq 2 } e^{C k^2 + \frac{C}{T}}e^{-\frac{k^4}{3} T}\right) \norm{v_0}_{L^2(\T)}.
\end{align*}
Using Young's inequality we have $Ck^2\leq \frac{C^2}{T}+\frac{k^4T}{4},$ and putting this in the above estimate there exists constant $C>0$ such that
\begin{align*}
	\norm{h}_{L^2(0,T)}&\leq C e^{\frac{C}{T}}\norm{v_0}_{L^2(\T)}.
\end{align*}
This completes the proof.
	\end{proof}
The proof of local exact controllability for the Kuramoto--Sivashinsky equation follows from arguments analogous to those used in \Cref{sec_source_term} for the Cahn--Hilliard equation. 
The proof of global exact controllability, namely \Cref{global_exact1}, follows the same lines as the proof of \Cref{global_exact}. 
We omit the details for brevity.

\medskip

{
\paragraph{\bf \emph{Acknowledgement}}
 The major part of this work was done while Subrata Majumdar was a postdoctoral fellow at the Instituto de Matemáticas of the Universidad Nacional Autónoma de México (UNAM). He received financial support from the UNAM DGAPA postdoctoral scholarship (POSDOC) and the UNAM-DGAPA-PAPIIT grant IN117525 (Mexico). Currently, his work is supported by the Institute Postdoctoral Fellowship from the TIFR Centre for Applicable Mathematics, Bangalore. He also thanks Debayan Maity for fruitful discussions. Debanjit Mondal received the Integrated PhD Fellowship from IISER Kolkata and gratefully acknowledges Shirshendu Chowdhury for introducing him to the geometric control perspective and for drawing his attention to the related literature. He expresses his gratitude
to Indranil Chowdhury and the Department of Mathematics \& Statistics, IIT Kanpur for their hospitality and assistance during his visit. Finally, he thanks Karine Beauchard and Eugenio Pozzoli for pointing out a flaw in the proof of the saturating space argument in the earlier version of this manuscript.
 }

\bibliographystyle{alpha}
\bibliography{reference}

@article{beauchard2025small,
  title={Small-time approximate controllability of bilinear Schr{\"o}dinger equations and diffeomorphisms},
  author={Beauchard, Karine and Pozzoli, Eugenio},
JOURNAL = {Ann. Inst. H. Poincaré C Anal. Non Linéaire},
  year={2025}
}

@inproceedings{BP25,
  title={Examples of small-time controllable Schr{\"o}dinger equations},
  author={Beauchard, Karine and Pozzoli, Eugenio},
  booktitle={Annales Henri Poincar{\'e}},
  pages={1--30},
  year={2025},
  organization={Springer}
}

@article {Cer10,
    AUTHOR = {Cerpa, Eduardo},
     TITLE = {Null controllability and stabilization of the linear
              {K}uramoto-{S}ivashinsky equation},
   JOURNAL = {Commun. Pure Appl. Anal.},
  FJOURNAL = {Communications on Pure and Applied Analysis},
    VOLUME = {9},
      YEAR = {2010},
    NUMBER = {1},
     PAGES = {91--102},
      ISSN = {1534-0392,1553-5258},
   MRCLASS = {93B05 (35Q53 93C20 93D15)},
  MRNUMBER = {2556747},
MRREVIEWER = {Patrick\ Martinez},
       DOI = {10.3934/cpaa.2010.9.91},
       URL = {https://doi.org/10.3934/cpaa.2010.9.91},
}

@article{HsM25,
  title={On the controllability of the Kuramoto-Sivashinsky equation on multi-dimensional cylindrical domains},
  author={Hern{\'a}ndez-Santamar{\'\i}a, V{\'\i}ctor and Majumdar, Subrata},
  journal={arXiv preprint arXiv:2508.00812},
  year={2025}
}

@book {Kh10,
    AUTHOR = {Khapalov, Alexander Y.},
     TITLE = {Controllability of partial differential equations governed by
              multiplicative controls},
    SERIES = {Lecture Notes in Mathematics},
    VOLUME = {1995},
 PUBLISHER = {Springer-Verlag, Berlin},
      YEAR = {2010},
     PAGES = {xvi+284},
      ISBN = {978-3-642-12412-9},
   MRCLASS = {93-02 (74H45 74M05 76D55 76Z10 93B05 93C20)},
  MRNUMBER = {2641453},
MRREVIEWER = {Iasson\ Karafyllis},
       DOI = {10.1007/978-3-642-12413-6},
       URL = {https://doi.org/10.1007/978-3-642-12413-6},
}

@incollection {DR98,
    AUTHOR = {D\'{\i}az, J. I. and Ramos, A. M.},
     TITLE = {On the approximate controllability for higher order parabolic
              nonlinear equations of {C}ahn-{H}illiard type},
 BOOKTITLE = {Control and estimation of distributed parameter systems
              ({V}orau, 1996)},
    SERIES = {Internat. Ser. Numer. Math.},
    VOLUME = {126},
     PAGES = {111--127},
 PUBLISHER = {Birkh\"{a}user, Basel},
      YEAR = {1998},
      ISBN = {3-7643-5835-1},
   MRCLASS = {93C20 (35K55 35Q80 93B05)},
  MRNUMBER = {1627671},
MRREVIEWER = {Sergei\ A.\ Avdonin},
}

@article {Siv83,
    AUTHOR = {Sivashinsky, G. I.},
     TITLE = {On cellular instability in the solidification of a dilute
              binary alloy},
   JOURNAL = {Phys. D},
  FJOURNAL = {Physica D. Nonlinear Phenomena},
    VOLUME = {8},
      YEAR = {1983},
    NUMBER = {1-2},
     PAGES = {243--248},
      ISSN = {0167-2789,1872-8022},
   MRCLASS = {80A15 (76E15)},
  MRNUMBER = {724591},
       DOI = {10.1016/0167-2789(83)90321-4},
       URL = {https://doi.org/10.1016/0167-2789(83)90321-4},
}

@article {NG95,
    AUTHOR = {Novick-Cohen, A. and Grinfeld, M.},
     TITLE = {Counting the stationary states of the {S}ivashinsky equation},
   JOURNAL = {Nonlinear Anal.},
  FJOURNAL = {Nonlinear Analysis. Theory, Methods \& Applications. An
              International Multidisciplinary Journal},
    VOLUME = {24},
      YEAR = {1995},
    NUMBER = {6},
     PAGES = {875--881},
      ISSN = {0362-546X,1873-5215},
   MRCLASS = {35Q53 (34B15 34E15 35K55)},
  MRNUMBER = {1320694},
MRREVIEWER = {Song\ Mu\ Zheng},
       DOI = {10.1016/0362-546X(94)00121-W},
       URL = {https://doi.org/10.1016/0362-546X(94)00121-W},
}

@book {Mir19,
    AUTHOR = {Miranville, Alain},
     TITLE = {The {C}ahn-{H}illiard equation},
    SERIES = {CBMS-NSF Regional Conference Series in Applied Mathematics},
    VOLUME = {95},
      NOTE = {Recent advances and applications},
 PUBLISHER = {Society for Industrial and Applied Mathematics (SIAM),
              Philadelphia, PA},
      YEAR = {2019},
     PAGES = {xiv+216},
      ISBN = {978-1-611975-91-8},
   MRCLASS = {35-02 (76-02 82C26)},
  MRNUMBER = {4001523},
MRREVIEWER = {Ryan\ W.\ Murray},
       DOI = {10.1137/1.9781611975925},
       URL = {https://doi.org/10.1137/1.9781611975925},
}

@article{CH58,
  title={Free energy of a nonuniform system. I. Interfacial free energy},
  author={Cahn, John W and Hilliard, John E},
  journal={The Journal of chemical physics},
  volume={28},
  number={2},
  pages={258--267},
  year={1958},
  publisher={American Institute of Physics}
}

@article{CH59,
  title={Free energy of a nonuniform system. III. Nucleation in a two-component incompressible fluid},
  author={Cahn, John W and Hilliard, John E},
  journal={The Journal of chemical physics},
  volume={31},
  number={3},
  pages={688--699},
  year={1959},
  publisher={American Institute of Physics}
}

@article {Tad86,
    AUTHOR = {Tadmor, Eitan},
     TITLE = {The well-posedness of the {K}uramoto-{S}ivashinsky equation},
   JOURNAL = {SIAM J. Math. Anal.},
  FJOURNAL = {SIAM Journal on Mathematical Analysis},
    VOLUME = {17},
      YEAR = {1986},
    NUMBER = {4},
     PAGES = {884--893},
      ISSN = {0036-1410},
   MRCLASS = {35K55 (35Q20)},
  MRNUMBER = {846395},
MRREVIEWER = {J. R. Ockendon},
       DOI = {10.1137/0517063},
       URL = {https://doi.org/10.1137/0517063},
}

@article{KT75,
  title={On the formation of dissipative structures in reaction-diffusion systems: Reductive perturbation approach},
  author={Kuramoto, Yoshiki and Tsuzuki, Toshio},
  journal={Progress of Theoretical Physics},
  volume={54},
  number={3},
  pages={687--699},
  year={1975},
  publisher={Oxford University Press}
}

@article{KT76,
  title={Persistent propagation of concentration waves in dissipative media far from thermal equilibrium},
  author={Kuramoto, Yoshiki and Tsuzuki, Toshio},
  journal={Progress of theoretical physics},
  volume={55},
  number={2},
  pages={356--369},
  year={1976},
  publisher={Oxford University Press}
}

@article {FCGBdeT10,
    AUTHOR = {Fern\'{a}ndez-Cara, Enrique and Gonz\'{a}lez-Burgos, Manuel and de
              Teresa, Luz},
     TITLE = {Boundary controllability of parabolic coupled equations},
   JOURNAL = {J. Funct. Anal.},
  FJOURNAL = {Journal of Functional Analysis},
    VOLUME = {259},
      YEAR = {2010},
    NUMBER = {7},
     PAGES = {1720--1758},
      ISSN = {0022-1236},
   MRCLASS = {93B05 (35K15 93C20)},
  MRNUMBER = {2665408},
MRREVIEWER = {Abdes-Samed Bernoussi},
       DOI = {10.1016/j.jfa.2010.06.003},
       URL = {https://doi.org/10.1016/j.jfa.2010.06.003},
}

@article {AKBGBdT16,
    AUTHOR = {Ammar Khodja, Farid and Benabdallah, Assia and
              Gonz\'{a}lez-Burgos, Manuel and de Teresa, Luz},
     TITLE = {New phenomena for the null controllability of parabolic
              systems: minimal time and geometrical dependence},
   JOURNAL = {J. Math. Anal. Appl.},
  FJOURNAL = {Journal of Mathematical Analysis and Applications},
    VOLUME = {444},
      YEAR = {2016},
    NUMBER = {2},
     PAGES = {1071--1113},
      ISSN = {0022-247X},
   MRCLASS = {93B05 (35K51 93A13 93C20)},
  MRNUMBER = {3535750},
MRREVIEWER = {S\'{e}rgio da Silva Rodrigues},
       DOI = {10.1016/j.jmaa.2016.06.058},
       URL = {https://doi.org/10.1016/j.jmaa.2016.06.058},
}

@article {FR2,
	AUTHOR = {Fattorini, H. O. and Russell, D. L.},
	TITLE = {Uniform bounds on biorthogonal functions for real exponentials
	with an application to the control theory of parabolic
	equations},
	JOURNAL = {Quart. Appl. Math.},
	FJOURNAL = {Quarterly of Applied Mathematics},
	VOLUME = {32},
	YEAR = {1974/75},
	PAGES = {45--69},
	ISSN = {0033-569X},
	MRCLASS = {42A52 (93C20)},
	MRNUMBER = {510972},
	MRREVIEWER = {F. M. Kirillova},
	DOI = {10.1090/qam/510972},
	URL = {https://doi.org/10.1090/qam/510972},
}

@article {FR1,
	AUTHOR = {Fattorini, H. O. and Russell, D. L.},
	TITLE = {Exact controllability theorems for linear parabolic equations
	in one space dimension},
	JOURNAL = {Arch. Rational Mech. Anal.},
	FJOURNAL = {Archive for Rational Mechanics and Analysis},
	VOLUME = {43},
	YEAR = {1971},
	PAGES = {272--292},
	ISSN = {0003-9527},
	MRCLASS = {93B05},
	MRNUMBER = {335014},
	MRREVIEWER = {F. M. Kirillova},
	DOI = {10.1007/BF00250466},
	URL = {https://doi.org/10.1007/BF00250466},
}

@article {CM11,
    AUTHOR = {Cerpa, Eduardo and Mercado, Alberto},
     TITLE = {Local exact controllability to the trajectories of the 1-{D}
              {K}uramoto-{S}ivashinsky equation},
   JOURNAL = {J. Differential Equations},
  FJOURNAL = {Journal of Differential Equations},
    VOLUME = {250},
      YEAR = {2011},
    NUMBER = {4},
     PAGES = {2024--2044},
      ISSN = {0022-0396,1090-2732},
   MRCLASS = {93B05 (35K25 35K55 35Q53 93C20)},
  MRNUMBER = {2763563},
MRREVIEWER = {Patrick\ Martinez},
       DOI = {10.1016/j.jde.2010.12.015},
       URL = {https://doi.org/10.1016/j.jde.2010.12.015},
}

@article {Guz20,
    AUTHOR = {Guzm\'{a}n, Patricio},
     TITLE = {Local exact controllability to the trajectories of the
              {C}ahn-{H}illiard equation},
   JOURNAL = {Appl. Math. Optim.},
  FJOURNAL = {Applied Mathematics and Optimization},
    VOLUME = {82},
      YEAR = {2020},
    NUMBER = {1},
     PAGES = {279--306},
      ISSN = {0095-4616,1432-0606},
   MRCLASS = {93B05 (35K35 93B07 93C20)},
  MRNUMBER = {4123289},
MRREVIEWER = {Jone\ Apraiz},
       DOI = {10.1007/s00245-018-9500-2},
       URL = {https://doi.org/10.1007/s00245-018-9500-2},
}

@article {LTT,
	AUTHOR = {Liu, Y. and Takahashi, T. and Tucsnak, M.},
	TITLE = {Single input controllability of a simplified fluid-structure
	interaction model},
	JOURNAL = {ESAIM Control Optim. Calc. Var.},
	FJOURNAL = {ESAIM. Control, Optimisation and Calculus of Variations},
	VOLUME = {19},
	YEAR = {2013},
	NUMBER = {1},
	PAGES = {20--42},
	ISSN = {1292-8119,1262-3377},
	MRCLASS = {93B05 (35K59 74F10 76D55 93C25)},
	MRNUMBER = {3023058},
	MRREVIEWER = {Daniel\ Y.\ Toundykov},
	NODOI = {10.1051/cocv/2011196},
	URL = {https://doi.org/10.1051/cocv/2011196},
}

@misc{Boyer23,
	title={Controllability of linear parabolic equations and systems},
	author={Boyer, Franck},
	journal={Lecture Notes},
	howpublished={{L}ecture notes, \url{https://hal.archives-ouvertes.fr/hal-02470625v4}},
	year={2023}
}

@article {AB2014,
	AUTHOR = {Benabdallah, Assia and Boyer, Franck and Gonz\'{a}lez-Burgos,
	Manuel and Olive, Guillaume},
	TITLE = {Sharp estimates of the one-dimensional boundary control cost
	for parabolic systems and application to the {$N$}-dimensional
	boundary null controllability in cylindrical domains},
	JOURNAL = {SIAM J. Control Optim.},
	FJOURNAL = {SIAM Journal on Control and Optimization},
	VOLUME = {52},
	YEAR = {2014},
	NUMBER = {5},
	PAGES = {2970--3001},
	ISSN = {0363-0129},
	MRCLASS = {93B05 (35K20)},
	MRNUMBER = {3262588},
	MRREVIEWER = {Larissa V. Fardigola},
	DOI = {10.1137/130929680},
	URL = {https://doi.org/10.1137/130929680},
}

@article {Agrachev_Sarychev_2005,
    AUTHOR = {Agrachev, Andrey A. and Sarychev, Andrey V.},
     TITLE = {Navier-{S}tokes equations: controllability by means of low
              modes forcing},
   JOURNAL = {J. Math. Fluid Mech.},
  FJOURNAL = {Journal of Mathematical Fluid Mechanics},
    VOLUME = {7},
      YEAR = {2005},
    NUMBER = {1},
     PAGES = {108--152},
      ISSN = {1422-6928,1422-6952},
   MRCLASS = {93B05 (35Q30 49M25 76D05 76D55)},
  MRNUMBER = {2127744},
MRREVIEWER = {Enrique\ Fern\'{a}ndez Cara},
       DOI = {10.1007/s00021-004-0110-1},
       URL = {https://doi.org/10.1007/s00021-004-0110-1},
}

@article {Agrachev_Sarychev_2006,
    AUTHOR = {Agrachev, Andrey A. and Sarychev, Andrey V.},
     TITLE = {Controllability of 2{D} {E}uler and {N}avier-{S}tokes
              equations by degenerate forcing},
   JOURNAL = {Comm. Math. Phys.},
  FJOURNAL = {Communications in Mathematical Physics},
    VOLUME = {265},
      YEAR = {2006},
    NUMBER = {3},
     PAGES = {673--697},
      ISSN = {0010-3616,1432-0916},
   MRCLASS = {93B05 (35Q30 35Q35 76B75 76D55 93C20)},
  MRNUMBER = {2231685},
MRREVIEWER = {Anna\ V.\ Rozanova-Pierrat},
       DOI = {10.1007/s00220-006-0002-8},
       URL = {https://doi.org/10.1007/s00220-006-0002-8},
}

@article {Peng_Gao_2022,
    AUTHOR = {Gao, Peng},
     TITLE = {Irreducibility of {K}uramoto-{S}ivashinsky equation driven by
              degenerate noise},
   JOURNAL = {ESAIM Control Optim. Calc. Var.},
  FJOURNAL = {ESAIM. Control, Optimisation and Calculus of Variations},
    VOLUME = {28},
      YEAR = {2022},
     PAGES = {Paper No. 20, 22},
      ISSN = {1292-8119,1262-3377},
   MRCLASS = {60H15},
  MRNUMBER = {4388378},
       DOI = {10.1051/cocv/2022014},
       URL = {https://doi.org/10.1051/cocv/2022014},
}

@article {Pozzoli_2024,
    AUTHOR = {Pozzoli, Eugenio},
     TITLE = {Small-time global approximate controllability of bilinear wave
              equations},
   JOURNAL = {J. Differential Equations},
  FJOURNAL = {Journal of Differential Equations},
    VOLUME = {388},
      YEAR = {2024},
     PAGES = {421--438},
      ISSN = {0022-0396,1090-2732},
   MRCLASS = {93B05 (35G05 93B27 93C10 93C20)},
  MRNUMBER = {4699507},
       DOI = {10.1016/j.jde.2024.01.031},
       URL = {https://doi.org/10.1016/j.jde.2024.01.031},
}

@article {Duca_Nersesyan_JEMS_2025,
    AUTHOR = {Duca, Alessandro and Nersesyan, Vahagn},
     TITLE = {Bilinear control and growth of {S}obolev norms for the
              nonlinear {S}chr\"odinger equation},
   JOURNAL = {J. Eur. Math. Soc. (JEMS)},
  FJOURNAL = {Journal of the European Mathematical Society (JEMS)},
    VOLUME = {27},
      YEAR = {2025},
    NUMBER = {6},
     PAGES = {2603--2622},
      ISSN = {1435-9855,1435-9863},
   MRCLASS = {35Q55 (35R60 37L55 60H30 81Q93 93B05)},
  MRNUMBER = {4889252},
MRREVIEWER = {Juan\ Huang},
       DOI = {10.4171/jems/1420},
       URL = {https://doi.org/10.4171/jems/1420},
}

@book {Adams_Fournier_2003,
    AUTHOR = {Adams, Robert A. and Fournier, John J. F.},
     TITLE = {Sobolev spaces},
    SERIES = {Pure and Applied Mathematics (Amsterdam)},
    VOLUME = {140},
   EDITION = {Second},
 PUBLISHER = {Elsevier/Academic Press, Amsterdam},
      YEAR = {2003},
     PAGES = {xiv+305},
      ISBN = {0-12-044143-8},
   MRCLASS = {46E35 (46-01 46-02 46B70 46Exx)},
  MRNUMBER = {2424078},
}

@unpublished{Duca_Takashi_2025,
  TITLE = {{Small-time global controllability of the Burgers equation via bilinear controls}},
  AUTHOR = {Duca, Alessandro and Takahashi, Tak{\'e}o},
  URL = {https://hal.science/hal-04906284},
  NOTE = {working paper or preprint},
  YEAR = {2025},
  MONTH = Jan,
  HAL_ID = {hal-04906284},
  HAL_VERSION = {v1},
}

@article {Duca_Pozzoli_Urbai_JMPA_2025,
    AUTHOR = {Duca, Alessandro and Pozzoli, Eugenio and Urbani, Cristina},
     TITLE = {On the small-time bilinear control of a nonlinear heat
              equation: global approximate controllability and exact
              controllability to trajectories},
   JOURNAL = {J. Math. Pures Appl. (9)},
  FJOURNAL = {Journal de Math\'ematiques Pures et Appliqu\'ees. Neuvi\`eme
              S\'erie},
    VOLUME = {203},
      YEAR = {2025},
     PAGES = {Paper No. 103758, 41},
      ISSN = {0021-7824,1776-3371},
   MRCLASS = {93B05 (35K05 35Q93 93C20)},
  MRNUMBER = {4924255},
MRREVIEWER = {Luis\ Gerardo\ M\'armol},
       DOI = {10.1016/j.matpur.2025.103758},
       URL = {https://doi.org/10.1016/j.matpur.2025.103758},
}

@article {Ball_marsder_1982,
    AUTHOR = {Ball, J. M. and Marsden, J. E. and Slemrod, M.},
     TITLE = {Controllability for distributed bilinear systems},
   JOURNAL = {SIAM J. Control Optim.},
  FJOURNAL = {SIAM Journal on Control and Optimization},
    VOLUME = {20},
      YEAR = {1982},
    NUMBER = {4},
     PAGES = {575--597},
      ISSN = {0363-0129},
   MRCLASS = {49E15 (47D05 49B22 58D99)},
  MRNUMBER = {661034},
MRREVIEWER = {Henry\ Hermes},
       DOI = {10.1137/0320042},
       URL = {https://doi.org/10.1137/0320042},
}

@article {Lenhart_1992,
    AUTHOR = {Lenhart, Suzanne M. and Bhat, Mahadev G.},
     TITLE = {Application of distributed parameter control model in wildlife
              damage management},
   JOURNAL = {Math. Models Methods Appl. Sci.},
  FJOURNAL = {Mathematical Models and Methods in Applied Sciences},
    VOLUME = {2},
      YEAR = {1992},
    NUMBER = {4},
     PAGES = {423--439},
      ISSN = {0218-2025,1793-6314},
   MRCLASS = {92D25 (49K20 49N55)},
  MRNUMBER = {1189059},
MRREVIEWER = {C.\ P.\ Ortlieb},
       DOI = {10.1142/S0218202592000259},
       URL = {https://doi.org/10.1142/S0218202592000259},
}

@article {Cannarsa_Floridia_Khapalov_2017,
    AUTHOR = {Cannarsa, Piermarco and Floridia, Giuseppe and Khapalov,
              Alexander Y.},
     TITLE = {Multiplicative controllability for semilinear
              reaction-diffusion equations with finitely many changes of
              sign},
   JOURNAL = {J. Math. Pures Appl. (9)},
  FJOURNAL = {Journal de Math\'ematiques Pures et Appliqu\'ees. Neuvi\`eme
              S\'erie},
    VOLUME = {108},
      YEAR = {2017},
    NUMBER = {4},
     PAGES = {425--458},
      ISSN = {0021-7824,1776-3371},
   MRCLASS = {93B05 (35K20 35K55 35K57 35K58 93C20)},
  MRNUMBER = {3698164},
MRREVIEWER = {Sylvain\ Ervedoza},
       DOI = {10.1016/j.matpur.2017.07.002},
       URL = {https://doi.org/10.1016/j.matpur.2017.07.002},
}

@article {Cannarsa_Floridia_2011,
    AUTHOR = {Cannarsa, Piermarco and Floridia, Giuseppe},
     TITLE = {Approximate multiplicative controllability for degenerate
              parabolic problems with {R}obin boundary conditions},
   JOURNAL = {Commun. Appl. Ind. Math.},
  FJOURNAL = {Communications in Applied and Industrial Mathematics},
    VOLUME = {2},
      YEAR = {2011},
    NUMBER = {2},
     PAGES = {e-376, 16},
      ISSN = {2038-0909},
   MRCLASS = {93B05 (34B24 35K65 93C20)},
  MRNUMBER = {2873616},
MRREVIEWER = {Ilhem\ Hamchi},
       DOI = {10.1685/journal.caim.376},
       URL = {https://doi.org/10.1685/journal.caim.376},
}

@article {Khapalov_2002,
    AUTHOR = {Khapalov, Alexander Y.},
     TITLE = {Global non-negative controllability of the semilinear
              parabolic equation governed by bilinear control},
   JOURNAL = {ESAIM Control Optim. Calc. Var.},
  FJOURNAL = {ESAIM. Control, Optimisation and Calculus of Variations.
              European Series in Applied and Industrial Mathematics},
    VOLUME = {7},
      YEAR = {2002},
     PAGES = {269--283},
      ISSN = {1292-8119,1262-3377},
   MRCLASS = {93B05 (35K55 93C20)},
  MRNUMBER = {1925029},
MRREVIEWER = {Martin\ Gugat},
       DOI = {10.1051/cocv:2002011},
       URL = {https://doi.org/10.1051/cocv:2002011},
}

@article {Florida_2014,
    AUTHOR = {Floridia, Giuseppe},
     TITLE = {Approximate controllability for nonlinear degenerate parabolic
              problems with bilinear control},
   JOURNAL = {J. Differential Equations},
  FJOURNAL = {Journal of Differential Equations},
    VOLUME = {257},
      YEAR = {2014},
    NUMBER = {9},
     PAGES = {3382--3422},
      ISSN = {0022-0396,1090-2732},
   MRCLASS = {93B05 (35K20 35K58 35K65)},
  MRNUMBER = {3258141},
MRREVIEWER = {Larissa\ V.\ Fardigola},
       DOI = {10.1016/j.jde.2014.06.016},
       URL = {https://doi.org/10.1016/j.jde.2014.06.016},
}

@article {FloridaINitsch_Trombetti_2020,
    AUTHOR = {Floridia, G. and Nitsch, C. and Trombetti, C.},
     TITLE = {Multiplicative controllability for nonlinear degenerate
              parabolic equations between sign-changing states},
   JOURNAL = {ESAIM Control Optim. Calc. Var.},
  FJOURNAL = {ESAIM. Control, Optimisation and Calculus of Variations},
    VOLUME = {26},
      YEAR = {2020},
     PAGES = {Paper No. 18, 34},
      ISSN = {1292-8119,1262-3377},
   MRCLASS = {93B05 (35K20 35K57 35K58 35K65 93C20)},
  MRNUMBER = {4065619},
MRREVIEWER = {Michela\ Egidi},
       DOI = {10.1051/cocv/2019066},
       URL = {https://doi.org/10.1051/cocv/2019066},
}

@article {Alabau_Cannarsa_Urbani_21,
    AUTHOR = {Alabau-Boussouira, Fatiha and Cannarsa, Piermarco and Urbani,
              Cristina},
     TITLE = {Superexponential stabilizability of evolution equations of
              parabolic type via bilinear control},
   JOURNAL = {J. Evol. Equ.},
  FJOURNAL = {Journal of Evolution Equations},
    VOLUME = {21},
      YEAR = {2021},
    NUMBER = {1},
     PAGES = {941--967},
      ISSN = {1424-3199,1424-3202},
   MRCLASS = {93D20 (35K10 35Q93 93B60 93C25)},
  MRNUMBER = {4238231},
MRREVIEWER = {Akira\ Ichikawa},
       DOI = {10.1007/s00028-020-00611-z},
       URL = {https://doi.org/10.1007/s00028-020-00611-z},
}

@article {Alabau_cannarsa_Urbani_22,
    AUTHOR = {Alabau-Boussouira, Fatiha and Cannarsa, Piermarco and Urbani,
              Cristina},
     TITLE = {Exact controllability to eigensolutions for evolution
              equations of parabolic type via bilinear control},
   JOURNAL = {NoDEA Nonlinear Differential Equations Appl.},
  FJOURNAL = {NoDEA. Nonlinear Differential Equations and Applications},
    VOLUME = {29},
      YEAR = {2022},
    NUMBER = {4},
     PAGES = {Paper No. 38, 32},
      ISSN = {1021-9722,1420-9004},
   MRCLASS = {35Q93 (35K90 93B05 93C10 93C25)},
  MRNUMBER = {4416801},
MRREVIEWER = {Nguyen Van Thanh},
       DOI = {10.1007/s00030-022-00770-7},
       URL = {https://doi.org/10.1007/s00030-022-00770-7},
}

@article {Cannarsa_Duca_Urbani_22,
    AUTHOR = {Cannarsa, Piermarco and Duca, Alessandro and Urbani, Cristina},
     TITLE = {Exact controllability to eigensolutions of the bilinear heat
              equation on compact networks},
   JOURNAL = {Discrete Contin. Dyn. Syst. Ser. S},
  FJOURNAL = {Discrete and Continuous Dynamical Systems. Series S},
    VOLUME = {15},
      YEAR = {2022},
    NUMBER = {6},
     PAGES = {1377--1401},
      ISSN = {1937-1632,1937-1179},
   MRCLASS = {93B05 (35R02 93B70 93C20)},
  MRNUMBER = {4421539},
MRREVIEWER = {Julian\ Edward},
       DOI = {10.3934/dcdss.2022011},
       URL = {https://doi.org/10.3934/dcdss.2022011},
}

@unpublished{Buffe_Duca_25,
  TITLE = {{Exact controllability to eigensolutions of the fractional heat equation via bilinear controls on N-dimensional domains}},
  AUTHOR = {Buffe, R{\'e}mi and Duca, Alessandro},
  URL = {https://hal.science/hal-04622031},
  NOTE = {working paper or preprint},
  YEAR = {2025},
  MONTH = May,
  HAL_ID = {hal-04622031},
  HAL_VERSION = {v2},
}

\bigskip \bigskip
\noindent


\bigskip
\bigskip 
\begin{flushleft}
	
	\textbf{Subrata Majumdar\,\orcidlink{0000-0001-6724-6943}}\\
	Centre For Applicable Mathematics,	Tata Institute of Fundamental Research\\
	Post Bag No 6503, Sharada Nagar \\
	Bengaluru 560065, India\\
	\texttt{subrata26@tifrbng.res.in, subratamaumdar634@gmail.com}\\
	
\end{flushleft}
\medskip
\begin{flushleft}
	 \textbf{Debanjit Mondal\,\orcidlink{0009-0002-8709-3774}}\\
	Department of Mathematics and Statistics\\
	Indian Institute of Science Education and Research Kolkata, \\
	Campus road, Mohanpur, West Bengal 741246, India \\
	\texttt{wrmarit@gmail.com, dm20ip005@iiserkol.ac.in}\\
	
\end{flushleft}

\end{document}